\documentclass{emsprocart}

\usepackage[utf8]{inputenc}

\usepackage[T1]{fontenc}

\contact[thiel@mathematik.uni-stuttgart.de]{Ulrich Thiel, Fachbereich Mathematik, Universität Stuttgart, Pfaffenwaldring 57, 70569 Stuttgart, Germany}

\usepackage[american]{babel}
\usepackage[protrusion=true,expansion=true,babel=true,final=true]{microtype}

\newtheorem{theorem}{Theorem}[section]
\newtheorem{corollary}[theorem]{Corollary}
\newtheorem{lemma}[theorem]{Lemma}
\newtheorem{proposition}[theorem]{Proposition}
\newtheorem{conjecture}[theorem]{Conjecture}
\newtheorem{example}[theorem]{Example}

\newtheorem{problem}[theorem]{Problem}

\theoremstyle{definition}
\newtheorem{definition}[theorem]{Definition}
\newtheorem{remark}[theorem]{Remark}

\usepackage{tikz-cd}
\usepackage{tikz}
\usetikzlibrary{arrows,positioning} 

\usepackage{longtable}

\newcommand{\ol}[1]{\overline{#1}}
\newcommand{\wt}[1]{\widetilde{#1}}

\newcommand{\msf}[1]{\mathsf{#1}}
\newcommand{\mbf}[1]{\mathbf{#1}}
\newcommand{\mrm}[1]{\mathrm{#1}}
\newcommand{\tn}[1]{\textnormal{#1}}
\newcommand{\bbN}{\mathbb{N}}

\newcommand{\bbZ}{\mathbb{Z}}

\newcommand{\bbQ}{\mathbb{Q}}

\newcommand{\bbC}{\mathbb{C}}
\newcommand{\fh}{\mathfrak{h}}

\newcommand{\fm}{\mathfrak{m}}
\newcommand{\fM}{\mathfrak{M}}

\newcommand{\subs}{\subseteq}
\newcommand{\sups}{\supseteq}
\newcommand{\dopgleich}{\mathrel{\mathop:}=}

\usepackage{mathrsfs}  
\newcommand{\mscr}[1]{\mathscr{#1}}

\newcommand{\eps}{\varepsilon}
\newcommand{\fp}{\mathfrak{p}}

\newcommand{\fq}{\mathfrak{q}}

\usepackage{listings}

\newcommand{\specialcell}[2][c]{%
  \begin{tabular}[#1]{@{}c@{}}#2\end{tabular}}

\newcommand{\pideg}{\msf{PI}\tn{-}\msf{deg}}

\DeclareSymbolFont{sfoperators}{OT1}{cmss}{m}{n}
\DeclareSymbolFontAlphabet{\mathsf}{sfoperators}

\makeatletter
\def\operator@font{\mathgroup\symsfoperators}
\makeatother

\DeclareMathAlphabet{\mathbfsf}{\encodingdefault}{\sfdefault}{bx}{n}

\newcommand{\BBB}{\mathbfsf{B}}
\newcommand{\ccc}{\mathbfsf{c}}

\newcommand{\CCC}{\mathbfsf{C}}
\newcommand{\PPP}{\mathbfsf{P}}
\newcommand{\XXX}{\mathbfsf{X}}
\newcommand{\HHH}{\mathbfsf{H}}
\newcommand{\UUU}{\mathbfsf{U}}
\newcommand{\ZZZ}{\mathbfsf{Z}}

\newcommand{\HH}{\mathsf{H}}

\newcommand{\ZZ}{\mathsf{Z}}
\newcommand{\CM}{\mathsf{CM}}
\newcommand{\Eu}{\mathsf{Eu}}
\newcommand{\eu}{\mathbfsf{eu}}
\newcommand{\eee}{\mathbfsf{e}}
\DeclareMathOperator{\BlEx}{BlEx}
\DeclareMathOperator{\BlGen}{BlGen}
\DeclareMathOperator{\DecGen}{DecGen}
\DeclareMathOperator{\End}{End}
\DeclareMathOperator{\GL}{GL}
\DeclareMathOperator{\Hom}{Hom}
\DeclareMathOperator{\Aut}{Aut}
\DeclareMathOperator{\Mat}{Mat}
\DeclareMathOperator{\Ker}{Ker}

\DeclareMathOperator{\Max}{Max}
\DeclareMathOperator{\Spec}{Spec}
\DeclareMathOperator{\SL}{SL}
\DeclareMathOperator{\Sp}{Sp}
\DeclareMathOperator{\Ref}{Ref}
\DeclareMathOperator{\Irr}{Irr}
\DeclareMathOperator{\gr}{gr}

\DeclareMathOperator{\Frac}{Frac}

\DeclareMathOperator{\codim}{codim}
\DeclareMathOperator{\id}{id}

\DeclareMathOperator{\Rad}{Rad}
\DeclareMathOperator{\Ann}{Ann}
\DeclareMathOperator{\Bl}{Bl}

\newcommand{\word}[1]{\textit{#1}}
\newcommand{\rarr}{\rightarrow}

\usepackage{amsbsy}

\usepackage[pdftex]{hyperref}
\hypersetup
{
    colorlinks=true,
    linkcolor=black,
    urlcolor=black,
    filecolor=black,
    citecolor=black
}

\title[Restricted rational Cherednik algebras]{Restricted rational Cherednik algebras}

\author[Ulrich Thiel]{Ulrich Thiel\thanks{The author was partially supported by DFG SPP 1388.}}

\begin{document}

\begin{abstract}
We give an overview of the representation theory of restricted rational Cherednik algebras. These are certain finite-dimensional quotients of rational Cherednik algebras at $t=0$. Their representation theory is connected to the geometry of the Calogero--Moser space, and there is a lot of evidence that they contain certain information about Hecke algebras even though the precise connection is so far unclear. We outline the basic theory along with some open problems and conjectures, and give explicit results in the cyclic and dihedral cases.
\end{abstract}

%

\maketitle

\section*{Introduction}

In 2002 Etingof and Ginzburg \cite{EG} introduced the so-called \word{rational Cherednik algebras}. These are (non-commutative) $\bbC$-algebras $\HHH_{t,c}$  deforming the skew group ring $\bbC \lbrack \fh \oplus \fh^* \rbrack \rtimes W$ attached to a finite reflection group $W$ acting on a finite-dimensional complex vector space $\fh$. They depend on the choice of two parameters $t$ and $c$, where $t$ is just a complex number and $c$ is a map $\Ref(W) \rarr \bbC$ from the set of reflections of $W$ to the complex numbers which is invariant under $W$-conjugation of reflections. 

There is the following dichotomy in the behavior of rational Cherednik algebras which separates them into two quite distinct worlds: if $t \neq 0$, then the center $\ZZ(\HHH_{t,c})$ is as small as possible, i.e., $\ZZ(\HHH_{t,c}) = \bbC$, and if $t = 0$, then $\ZZ(\HHH_{0,c})$ is so large that the infinite-dimensional algebra $\HH_{0,c}$ becomes a finite module over $\ZZ(\HHH_{0,c})$. Moreover, if $t \neq 0$, which after rescaling is equivalent to $t=1$, then $\HHH_{1,0} = \mathcal{D}(\fh) \rtimes W$, where $\mathcal{D}(\fh)$ is the ring of differential operators on $\fh$. So, the algebras $\HHH_{1,c}$ deform the skew differential operator ring and this signifies for example that the theory of $\mathcal{D}$-modules plays a role for $t \neq 0$. The representation theory of $\HHH_{t,c}$ and the methods to study it thus heavily depend on whether $t = 0$ or $t \neq 0$. 

The case $t \neq 0$ has attracted a lot of interest. One of the many reasons for this is the work of Ginzburg--Guay--Opdam--Rouquier \cite{GGOR}. They defined a certain subcategory $\mathcal{O}_c$ of the category $\HHH_{1,c}\textnormal{-}\msf{mod}$ of finitely generated $\HHH_{1,c}$-modules and  showed that it is a highest weight category whose standard modules are naturally indexed by $\Irr W$, the set of irreducible complex representations of $W$. In particular, the simple objects in $\mathcal{O}_c$ are naturally indexed by $\Irr W$. Moreover, they constructed an exact functor
\[
\msf{KZ}: \mathcal{O}_c \rarr \mathcal{H}_{q(c)}\textnormal{-}\msf{mod}
\]
to the module category of the cyclotomic Hecke algebra attached to $W$ at a certain parameter $q(c)$ derived from $c$. This functor was used to prove properties of $\mathcal{O}_c$ using properties of cyclotomic Hecke algebras. Astonishingly, it was also used the other way around by Losev \cite{Losev-Hecke} to prove a weak version of the freeness conjecture about cyclotomic Hecke algebras by Broué--Malle--Rouquier \cite{BMR}. 

So, for $t \neq 0$ we have a strong connection between rational Cherednik algebras and Hecke algebras. This is different for $t=0$ since the construction of the $\msf{KZ}$-functor does not work here anymore. It is not even clear what a correct analog of the category $\mathcal{O}_c$ could be. Still, it seems that there is some information about Hecke algebras contained in $\HHH_{0,c}$—as if the $\msf{KZ}$-functor left some traces in the limit process $t \rarr 0$. There is a certain natural finite-dimensional quotient $\ol{\HHH}_c$ of $\HHH_{0,c}$, called the \word{restricted} rational Cherednik algebra, which controls a lot of the representation theory of $\HH_{0,c}$. This quotient was first studied more closely by Gordon \cite{Gordon-Baby} who showed that there is a natural bijection $\Irr W \simeq \Irr \ol{\HHH}_c$ between isomorphism classes of simple modules (recall the bijection $\Irr W \simeq \Irr \mathcal{O}_c$ for $t \neq 0$ from above!). This implies in particular that the block structure of $\ol{\HHH}_c$ yields a natural $c$-dependent partition $\CM_c$ of $\Irr W$, called the \word{Calogero--Moser $c$-families}. Assuming Lusztig's conjectures P1--P15, see \cite{Lusztig-Hecke}, Gordon and Martino \cite{GM} showed that for $W$ of type $B_n$ these families match up with the Kazhdan--Lusztig families coming  from the Hecke algebra (see \cite{KL} and \cite{Geck-cells-families}), and they conjecture that this holds for all Weyl groups. The Kazhdan--Lusztig families play an important role in the representation theory of the finite groups of Lie type, so it is astonishing that such information seems to be encoded in $\ol{\HHH}_c$. When assuming Lusztig's conjectures, the Kazhdan--Lusztig families coincide with the so-called Lusztig families, also coming from the Hecke algebra (see \cite{Lusztig-Hecke}), so it makes sense to conjecture equality of the Calogero--Moser families and the Lusztig families for all Coxeter groups, see \cite{Bonnafe-b}. This was in fact shown to be true for almost all Coxeter groups by the work of Etingof--Ginzburg \cite{EG}, Gordon \cite{Gordon-Baby}, Gordon--Martino \cite{GM}, Bellamy \cite{Bellamy-Thesis, Bellamy-Singular}, Martino \cite{Martino-blocks}, by the author \cite{Thiel-Counter}, and by Bonnafé and the author \cite{Bonnafe-Thiel}. The conjecture is only open for the four exceptional groups $H_4$, $E_6$, $E_7$, and $E_8$. This is a hint that there is information about Hecke algebras contained in $\ol{\HHH}_c$. The problem is: we do not really know why—the proofs are obtained ``simply'' by computing both sides separately and comparing. 

What is nice about the Calogero--Moser families is that they have a geometric interpretation: the spectrum $\XXX_c$ of the center $\ZZZ_c$ of $\HHH_{0,c}$ is an irreducible variety with a natural $\bbC^*$-action, called the \word{Calogero--Moser space}, and we have a natural bijection $\CM_c \simeq \XXX_c^{\bbC^*}$. Hence, due to the positive result about the Gordon--Martino conjecture we also get a geometric description of Lusztig's families. This gives us new methods to study these families.  The Calogero--Moser space $\XXX_c$ carries a natural Poisson bracket deforming the one on $(\fh \oplus \fh^*)/W$ coming from the natural symplectic form on $\fh \oplus \fh^*$. We note that the Poisson bracket on $\XXX_c$ actually comes from the commutator in the Cherednik algebra $\HHH_{t,c}$ for $t \neq 0$, so the worlds $t = 0$ and $t\neq 0$ are indeed not entirely separated. It was shown by Brown and Gordon \cite{BG} that $\XXX_c$ admits a stratification into so-called \word{symplectic leaves}. Using Poisson geometry of $\XXX_c$ and the above bijection, Bellamy \cite{Bellamy-cuspidal} sorted out special Calogero--Moser families, namely those where the corresponding $\bbC^*$-fixed point of $\XXX_c$ lies on a zero-dimensional symplectic leaf. These families are called \word{cuspidal}. Bellamy and the author \cite{Bellamy-Thiel} conjecture that the \textit{cuspidal} Calogero--Moser families coincide with the \textit{cuspidal} Lusztig families, thus extending the Gordon--Martino conjecture. Here, a Lusztig family is cuspidal if it is (up to sign) minimal with respect to Lusztig's $\mathbf{j}$-induction, see \cite{Lusztig-Hecke}. It was shown by Bellamy and the author \cite{Bellamy-Thiel}, and by Bonnafé and the author \cite{Bonnafe-Thiel}, that this conjecture is true for all Coxeter groups except possibly $H_4$, $E_6$, $E_7$, and $E_8$. And again: we do not know the conceptual reason for this.

The conjectural connection between Calogero--Moser spaces and Hecke algebras was lifted to a new level by the work of Bonnafé and Rouquier \cite{BR, BR-2}. For an arbitrary complex reflection group $W$ they used  a Galois covering of the Calogero--Moser space to construct a $c$-dependent decomposition of $W$ into so-called \word{Calogero--Moser $c$-cells} and to construct so-called \word{Calogero--Moser $c$-cellular characters}. They conjecture that for Coxeter groups these objects coincide with the Kazhdan--Lusztig cells and cellular characters, respectively. This can be viewed as a deep refinement of the Gordon--Martino conjecture. 

Motivated by Bonnafé and Rouquier we call the quest for connections between the Calogero–Moser and the Kazhdan–Lusztig world the ``CM vs.\ KL program''. In Figure \ref{cmvskl} we summarize the situation (we note that this picture may not yet be complete). 
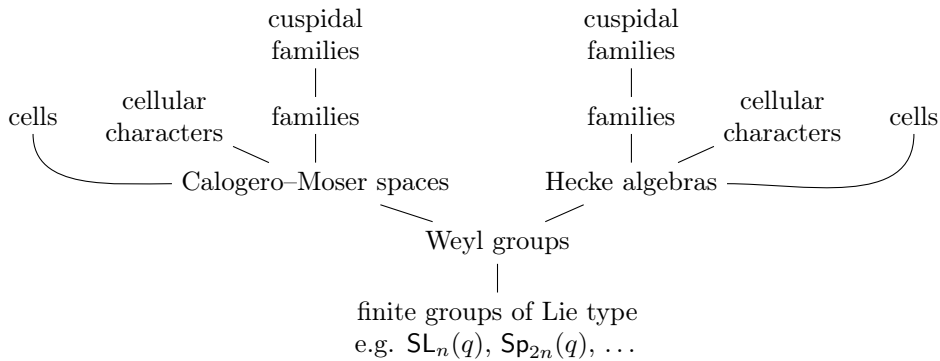
\begin{figure}[htbp]
\begin{tikzpicture}[node distance=0.38cm]

\node[align=center] (weyl) {Weyl groups}; 

\node[above=of weyl] (dummy) {};

\node[align=center, below=of weyl] (fgl) {finite groups of Lie type\\e.g. $\SL_n(q)$, $\Sp_{2n}(q)$, \ldots};

\node[align=center, right=of dummy] (hecke) {Hecke algebras}; 

\node[align=center, left=of dummy] (cm) {Calogero--Moser spaces}; 

\node[align=center, above=of hecke] (lus-families) {families}; 

\node[align=center, right=of lus-families] (lus-cell-chars) {cellular\\characters}; 

\node[align=center, right=of lus-cell-chars] (lus-cells) {cells};

\node[align=center, above=of lus-families] (lus-cusp) {cuspidal\\families};

\node[align=center, above=of cm] (cm-families) {families};

\node[align=center, left=of cm-families] (cm-cellular) {cellular\\characters};

\node[align=center, left=of cm-cellular] (cm-cells) {cells};

\node[align=center, above=of cm-families] (cm-cusp) {cuspidal\\families};

\draw [-] (fgl) -- (weyl) ;
\draw [-] (weyl) -- (cm) ;
\draw [-] (weyl) -- (hecke) ;
\draw [-] (cm) -- (cm-families) ;
\draw [-] (cm) -- (cm-cellular) ;
\draw [-] (cm) to [out=180,in=270] (cm-cells) ;

\draw [-] (hecke) -- (lus-families) ;
\draw [-] (hecke) -- (lus-cell-chars) ;
\draw [-] (hecke) to [out=0,in=270] (lus-cells) ;

\draw [-] (lus-families) -- (lus-cusp) ;

\draw [-] (cm-families) -- (cm-cusp) ;
\end{tikzpicture}
\caption{Calogero--Moser vs. Kazhdan--Lusztig} \label{cmvskl}
\end{figure}

There are at least two motivations for this program:
\begin{enumerate}
\item We might obtain new tools to study the Kazhdan--Lusztig side.
\item On the Calogero--Moser side everything is naturally defined for \textit{complex} reflection groups. So, once both sides match up for Weyl  groups (or, more generally, Coxeter groups), we immediately have an extension to complex reflection groups. We might thus get a new point of view of the \textit{spetses} by Broué–Malle–Michel \cite{BMM1,BMM2} and new tools for studying modular representation theory of finite groups of Lie type.
\end{enumerate}

This should motivate studying Calogero--Moser spaces and rational Cherednik algebras at $t=0$ more closely. The restricted rational Cherednik algebras form one important piece in this picture. Next to information related to the CM vs.\ KL program  the restricted rational Cherednik algebras contain further interesting information and combinatorics, in particular the graded decomposition matrices of baby Verma modules and the graded $W$-characters of simple modules. A further motivating problem is to understand if any of this information gives a hint about the limit process $t \rarr 0$, in particular about an appropriate analog of category $\mathcal{O}_c$ and the $\msf{KZ}$-functor at $t=0$. 

\subsection*{Outline}
We will exclusively concentrate on the CM side here, and, even more exclusively, on the representation theory of \textit{restricted} rational Cherednik algebras. The main problems about restricted rational Cherednik algebras we want to highlight here are presented in Section \ref{rrca_problems}. In Chapter \ref{RCA} we review the few bits about rational Cherednik algebras at $t=0$ we need to be able to introduce the restricted rational Cherednik algebras in Chapter \ref{RRCA}. Theorem \ref{pi_degree} about the PI-degree and Theorem \ref{morita_dcp_closed_points} about the restricted double centralizer property seem not be covered in the literature so far. In Chapter \ref{gen_rep_theory} we discuss several genericity properties and Chapter \ref{toolbox} is a toolbox where we collect several results that help to solve specific problems. In Chapter \ref{what_we_know}, and in particular in Table \ref{summary_table} on page \pageref{summary_table}, we give a summary of what has been solved already. In Chapter \ref{conjectures_chapter} we list some conjectures and some further open problems. In Chapters \ref{Cyclic_section} and \ref{Dihedral} we give explicit solutions for cyclic and dihedral groups. 


\subsection*{Literature}
We cover material (in varying detail) from the following papers, listed in chronological order:
\begin{itemize} \setlength\itemsep{0pt}
\item 2002: Etingof--Ginzburg \cite{EG}.
\item 2003: Gordon \cite{Gordon-Baby}.
\item 2006: Martino \cite{Martino-Thesis}.
\item 2008: Brown--Gordon--Stroppel \cite{BGS} and Gordon \cite{Gordon-Quiver}.
\item 2009: Bellamy \cite{Bellamy-Singular}, Gordon--Martino \cite{GM}.
\item 2010: Bellamy \cite{Bellamy-Thesis}, Martino \cite{Martino-Conjecture}.
\item 2011: Bellamy \cite{Bellamy-cuspidal}.
\item 2012: Bellamy \cite{Bellamy-blocks}.
\item 2013: Bonnafé--Rouquier \cite{BR}.
\item 2014: Bellamy \cite{Bellamy-End}, Martino \cite{Martino-blocks}, the author \cite{Thiel-Counter,Thiel-Diss}.
\item 2015: Bellamy and the author \cite{Bellamy-Thiel}, Bonnafé \cite{Bonnafe-b}, the author \cite{Thiel-Champ}.
\item There is furthermore work in progress by Bellamy and the author \cite{Bellamy-Thiel-weight} and by Bonnafé and the author \cite{Bonnafe-Thiel}.
\end{itemize}

\subsection*{The bigger context}
We cannot go into details here about the broader context of rational Cherednik algebras and the geometry of Calogero--Moser spaces. There are several extremely useful sources about this and we warmly recommend them to the reader: apart from the original paper by Etingof--Ginzburg \cite{EG}, there are excellent survey papers by Gordon \cite{Gordon-SRA, Gordon-RCA} and by Rouquier \cite{Rouquier-RCA}, and excellent lecture notes by Bellamy \cite{Bellamy-SRA}, by Chlouveraki \cite{Chlouveraki-SRA}, and by Losev \cite{Losev-SRA}. We furthermore highly recommend the wonderful manuscript \cite{BR} by Bonnafé and Rouquier. 


\subsection*{Assumptions}
Throughout, if nothing else is mentioned, we denote by $K \subs \bbC$ a field and by $\fh$ a finite-dimensional $K$-vector space. All rings are associative and unital. Modules are left-modules if nothing else is mentioned. To shorten notations we will use the unusual notation $A^\natural \dopgleich \Spec(A)$ for the scheme associated to a commutative ring $A$.\footnote{The \LaTeX \ code for the symbol $\natural$ is \texttt{\textbackslash natural}, so we thought that this is a natural choice.}

\subsection*{Acknowledgements}

I would like to thank Gwyn Bellamy and Gunter Malle for the abundance of comments they have given me on a preliminary version of this paper. I would also like to thank Cédric Bonnafé for providing me with many helpful insights. I furthermore would like to thank the referee for several remarks.

\tableofcontents

\section{Rational Cherednik algebras at $t=0$} \label{RCA}

In this chapter we review the few bits about rational Cherednik algebras at $t=0$ we need to be able to introduce the restricted rational Cherednik algebras in the next chapter. We will concentrate on ring-theoretic and representation-theoretic aspects and only briefly comment on connections to geometry. Our philosophy throughout is to view the rational Cherednik algebra as a sheaf of algebras on the parameter scheme and to take restrictions to closed subschemes and fibers in not necessarily closed points into account. To this end, we introduce the notion of a \word{geometric $\CCC$-algebra}, where $\CCC = K \lbrack \mscr{C} \rbrack$ is the generic parameter ring for rational Cherednik algebras at $t=0$. By this we simply mean a ring of the form $(\CCC/\fp)_{\fq/\fp}$ for prime ideals $\fp,\fq \in \CCC^\natural = \Spec(\CCC)$ with $\fq \sups \fp$. Essentially all of the basic properties still hold when we take such a ring as base ring for the rational Cherednik algebra. The reader may, however, simply assume throughout that $K=\bbC$ and that $c \in \CCC^\natural$ is a closed point, thus corresponding to an ordinary complex parameter.

\subsection{Reflection groups}

Let $W \subs \GL(\fh)$ be a finite \word{reflection group}, i.e., $W$ is generated by its set
\begin{equation}
\Ref(W) \dopgleich \lbrace s \in W \mid \codim \Ker(\id_\fh-s) = 1 \rbrace
\end{equation}
of \word{reflections}. We consider $\fh$ as a (faithful) $W$-module and denote the action of $w \in W$ on $y \in \fh$ by $\,^w y$. When choosing an isomorphism $\fh \simeq K^n$ we get an embedding $W \hookrightarrow \GL_n(K) \hookrightarrow \GL_n(\bbC)$ and we call the $\GL_n(\bbC)$-conjugacy class of the image of $W$ in $\GL_n(\bbC)$ the \word{type} of $W$. This does not depend on the choice of the isomorphism $\fh \simeq K^n$. Shephard and Todd \cite{ST} classified the types of irreducible (and thus of all) reflection groups. Standard representatives of the types are (without overlap!):
\begin{enumerate}
\item[(G1)] The symmetric group $S_{n+1}$ for $n \geq 4$ in the irreducible $n$-dimensional reflection representation obtained by taking the quotient of the natural action of $S_{n+1}$ on $\bbC^{n+1}$ by the line $e_1+\ldots+e_{n+1}$, where $(e_i)$ is the standard basis of $\bbC^n$.

\item[(G2)] The groups $G(m,p,n)$ of monomial matrices with $m,n>1$, $p \geq 1$ a divisor of $m$, such that $(m,p,n) \neq (2,2,2), (4,4,2)$. See \cite{ST} for the definition of these groups.

\item[(G3)] Cyclic groups $C_m$ for $m \geq 2$ acting by a primitive $m$-th root of unity on $\bbC$.
 
\item[{\parbox[t]{0.059\linewidth}{(G4--G37)}}] \parbox[t]{\linewidth}{The 34 groups denoted by $G_4,\ldots,G_{37}$, called \word{exceptional groups}. See \cite{ST} for the definition of these groups.}
\end{enumerate}

We also refer to \cite{LT} for a very nice treatment of the Shephard--Todd classification. An abstract finite group can have several non-isomorphic reflection representations, see \cite{Bessis-corps} for more details. We would like to mention the following result proven in \cite[Theorem 15.54]{Thiel-Diss} using the Shephard--Todd classification. It shows that the a~priori misleading term ``reflection \textit{group}'' is justified in the end.

\begin{theorem} \label{refl_grp_same_type_if_iso}
Irreducible finite reflection groups have the same type if and only if they are isomorphic as abstract groups. 
\end{theorem}

We have a natural action of $W$ on $\fh^*$ given by $(\,^w x)(y) = x(\,^{w^{-1}} y)$. This defines a subgroup $W^* \subs \GL(\fh^*)$, which is clearly again a reflection group and isomorphic to $W$ as an abstract group. Theorem \ref{refl_grp_same_type_if_iso} implies:

\begin{corollary}
If $W \subs \GL(\fh)$ is an irreducible finite reflection group, then the dual group $W^* \subs \GL(\fh^*)$ is of the same type as $W$.
\end{corollary}

\begin{remark}
Theorem \ref{refl_grp_same_type_if_iso} does not hold when dropping the assumption that $W$ is irreducible: a counter-example is given by the Weyl groups $G_2$ and $A_2 \times A_1$.\footnote{I would like to thank Cédric Bonnafé for pointing this out.}
\end{remark}

\subsection{Rational Cherednik algebras at $t=0$}

We denote by $\langle \cdot,\cdot \rangle:\fh \times \fh^* \rarr K$, $(y,x) \dopgleich x(y)$, the canonical pairing. Let $s \in \Ref(W)$. We call a non-zero element $\alpha_s^\vee \in \msf{Im}(\id_{\fh}-s)$ a \word{root} of $s$ and call a non-zero element $\alpha_s \in \msf{Im}(\id_{\fh^*}-(s^*)^{-1})$ a \word{coroot} of $s$. Note that roots and coroots are unique up to scalars since the image spaces are one-dimensional. Also note that $\langle \alpha_s^\vee, \alpha_s \rangle \neq 0$ since $s$ is diagonalizable. If $\eps_s$ denotes the (unique) non-trivial eigenvalue of $s$, we have
\begin{equation}
s(y) = y - (1-\eps_s)\frac{\langle y,\alpha_s \rangle}{\langle \alpha_s^\vee,\alpha_s \rangle} \alpha_s^\vee
\end{equation}
for all $y \in \fh$. We define a form $(\cdot,\cdot)_s:\fh \times \fh^* \rarr K$ by
\begin{equation} \label{cherednik_coefficient}
(y,x)_s \dopgleich \frac{ \langle \alpha_s^\vee,x \rangle \langle y,\alpha_s \rangle}{\langle \alpha_s^\vee,\alpha_s \rangle } \;,
\end{equation} 
where $\alpha_s^\vee$ and $\alpha_s$ are arbitrary roots and coroots, respectively. It is easy to see that the definition is independent of the choice of the root and coroot. \\

Let $\mscr{C}$ be the $K$-vector space of maps $c:\Ref(W) \rarr K$ which are constant on conjugacy classes of reflections. The dimension of this space is clearly equal to the number of conjugacy classes of reflections of $W$. For $s \in \Ref(W)$ let $\ccc(s)$ be the linear form on $\mscr{C}$ given by evaluation on $s$. Then 
\begin{equation}
\ccc \dopgleich (\ccc(s))_{s \in \Ref(W)/W}
\end{equation}
is a basis of $\mscr{C}^*$ and so the coordinate ring of $\mscr{C}$ is given by   
\begin{equation}
\CCC \dopgleich K \lbrack \mscr{C} \rbrack = K \lbrack \ccc \rbrack \;.
\end{equation}
We can consider $\ccc$ also as a map 
\begin{equation}
\ccc:\Ref(W) \rarr \CCC, \; s \mapsto \ccc(s) \;,
\end{equation}
which is constant on conjugacy classes of reflections. \\

By $\CCC \langle \fh \oplus \fh^* \rangle$ we denote the tensor algebra of $\fh \oplus \fh^* $ over $\CCC$ and by $\CCC \langle \fh \oplus \fh^* \rangle \rtimes W$ we denote the semidirect product with $W$, i.e., $\CCC \langle \fh \oplus \fh^* \rangle \rtimes W$ is as a $\CCC$-module isomorphic to $\CCC \langle \fh \oplus \fh^* \rangle \otimes_{\CCC} \CCC W$ with the usual multiplications inside $\CCC \langle \fh \oplus \fh^* \rangle$ and $\CCC W$ and intertwining action $wy = \,^w y w$ and $wx = \,^w x w$ for $y \in \fh$ and $x \in \fh^*$.

\begin{definition}[Etingof--Ginzburg] \label{gen_RCA_def}
The \word{generic rational Cherednik algebra at $t=0$} of $W$ is the quotient $\HHH$ of the $\CCC$-algebra $\CCC \langle \fh \oplus \fh^* \rangle \rtimes W$ by the ideal generated by the relations
\begin{equation}
\lbrack y,y' \rbrack = 0 = \lbrack x,x' \rbrack
\end{equation}
and
\begin{equation} \label{gen_RCA_def_yx}
\lbrack y,x \rbrack = \sum_{s \in \Ref(W)} (y,x)_s \ccc(s) s \in \CCC W \;,
\end{equation}
for $y,y' \in \fh$ and $x,x' \in \fh^*$.
\end{definition} 

We can consider $\HHH$ as a sheaf of algebras over $\CCC^\natural \dopgleich \Spec(\CCC)$ and there are two natural operations we can perform. For $c \in \CCC^\natural$ we can form the quotient
\begin{equation} \label{rca_res}
\HHH/c \HHH = \CCC/c \otimes_\CCC \HHH \;,
\end{equation}
which is naturally a $\CCC/c$-algebra and corresponds to the restriction of $\HHH$ to the zero locus of $c$ in $\CCC^\natural$. As a next step we can extend to the residue field $\msf{k}_\CCC(c) \dopgleich \Frac(\CCC/c)$ of $c$ to obtain
\begin{equation} \label{rca_spec}
\HHH_c \dopgleich \msf{k}_\CCC(c) \otimes_\CCC \HHH \;,
\end{equation}
which is naturally a $\msf{k}_\CCC(c)$-algebra. This is called the \word{fiber} (or \word{specialization}) of $\HHH$ in $c$. In principle, this might be confused with localization in $c$ but we think it is the most consistent notation. The set $\CCC^\natural(K)$ of $K$-points of $\CCC^\natural$ can be naturally identified with $\mscr{C}$, and for such a point $c$ the fiber $\HHH_c$ is simply the $K$-algebra with presentation as in Definition \ref{gen_RCA_def} but with $\CCC$ replaced by $K$ and $\ccc(s)$ replaced by $c(s)$. \\

To cover both (\ref{rca_res}) and (\ref{rca_spec}) in one setting, we introduce the following  concept.

\begin{definition}
A \word{geometric} $\CCC$-algebra is a localization of an integral quotient of $\CCC$, i.e., a ring of the form $c = (\CCC/\fp)_{\fq/\fp}$ for prime ideals $\fp,\fq$ of $\CCC$ with $\fq \sups \fp$. We then define
\begin{equation}
\HHH_c \dopgleich c \otimes_\CCC \HHH = (\HHH/\fp \HHH)_{\fq/\fp} \;.
\end{equation}
\end{definition}
The $c$-algebra $\HHH_c$ has the same presentation as in Definition \ref{gen_RCA_def} but with $\CCC$ replaced by $c$ and $\ccc(s)$ replaced by its image in $c$. Note that $\HHH_\CCC = \HHH$ and that $\HHH_c$ for $c \in \CCC^\natural$ as defined in (\ref{rca_spec}) is equal to $\HHH_{\msf{k}_\CCC(c)}$.

\subsection{Gradings} \label{gradings_general}

For any geometric $\CCC$-algebra $c$ we can equip $c \langle \fh \oplus \fh^* \rangle \rtimes W$ with a $\bbZ$-grading defined by
\begin{equation} \label{z_grading_def}
\deg(\fh^*) = 1, \; \deg(\fh) = -1, \; \deg(W) = 0, \;.
\end{equation}
It is clear that $c \lbrack \fh \oplus \fh^* \rbrack \rtimes W$ is a graded quotient of $c \langle \fh \oplus \fh^* \rangle \rtimes W$ and by the defining relations of $\HHH_c$ it is also clear that $\HHH_c$ is a graded quotient of $c \langle \fh \oplus \fh^* \rangle \rtimes W$.

On the generic algebra $\HHH$ one can define a finer grading, namely an $(\bbN \times \bbN)$-grading. We follow Bonnafé and Rouquier \cite[\S4.2]{BR}. An $(\bbN \times \bbN)$-grading on $c \langle \fh \oplus \fh^* \rangle \rtimes W$ is defined by 
\begin{equation} \label{generic_grading}
\deg(\fh^*) = (0,1), \; \deg(\fh) = (1,0), \; \deg(W) = (0,0), \; \deg(\mscr{C}^*) = (1,1) \;.
\end{equation}
The relations for $\HHH$ are clearly homogeneous with respect to this grading, so the above defines an $(\bbN\times \bbN)$-grading on $\HHH$. This induces via $\bbN \times \bbN \rarr \bbZ$, $(i,j) \mapsto j-i$, the $\bbZ$-grading on $\HHH$ just defined. Via the map $\bbN \times \bbN \rarr \bbN$, $(i,j) \mapsto i+j$, it also induces an $\bbN$-grading on $\HHH$ which is defined by
\begin{equation} \label{generic_n_grading}
\deg(\fh^*) = 1, \; \deg(\fh) = 1, \; \deg(W) = 0, \; \deg(\mscr{C}^*) = 2 \;.
\end{equation}

\subsection{PBW theorem} \label{pbw_section}
Let $c$ be a geometric $\CCC$-algebra. We consider $c \langle \fh \oplus \fh^* \rangle \rtimes W$ with its standard grading defined by
\begin{equation}
\deg(\fh^*) = 1, \; \deg(\fh) = 1, \; \deg(W) = 0\;.
\end{equation}
With respect to this grading the quotient $\HHH_c$ is not graded anymore but the grading induces a filtration on $\HHH_c$. Whereas the left hand side of the relation (\ref{gen_RCA_def_yx}) is of $\bbN$-degree $2$, the right hand side is of $\bbN$-degree $0$. Hence, relation (\ref{gen_RCA_def_yx}) becomes trivial in the associated graded $\gr(\HHH_c)$ of $\HHH_c$ with respect to the filtration. This, and the fact that $c \lbrack \fh \oplus \fh^* \rbrack \rtimes W$ is a graded quotient of $c \langle \fh \oplus \fh^* \rangle \rtimes W$, implies that  the quotient morphism $c \langle \fh \oplus \fh^* \rangle \rtimes W \twoheadrightarrow \HHH_c$ induces a surjective graded $c$-algebra morphism
\begin{equation}
\xi: c \lbrack \fh \oplus \fh^* \rbrack \rtimes W \twoheadrightarrow \gr(\HHH_c) \;.
\end{equation}
This morphism is called the \word{PBW morphism}. The following theorem is called the \word{PBW theorem} and was proven by Etingof and Ginzburg \cite{EG}.\footnote{In \cite{EG} a proof over $\bbC$ is given but one can prove this also in general, see \cite[Théorème 4.1.4]{BR} or \cite[\S16]{Thiel-Diss} for details.}

\begin{theorem}[Etingof--Ginzburg] \label{pbw_theorem}
The PBW morphism $\xi$ is an isomorphism. Hence, there is a $c$-module isomorphism $c \lbrack \fh \oplus \fh^* \rbrack \rtimes W \simeq \HHH_c$ respecting the filtration and all defined gradings. In particular, $\HHH_c$ is a free $c$-module.
\end{theorem}
 
%

%


It is now a standard fact that several ring-theoretic properties of the associated graded are reflected to the original ring.
 
\begin{corollary} \label{RCA_ring_properties}
The ring $\HHH_c$ is prime, noetherian, its center is an integral domain, and its (left/right) global dimension is bounded above by the global dimension of $c \lbrack \fh \oplus \fh^* \rbrack$. In particular, if $c$ is of finite global dimension, so is $\HHH_c$. 
\end{corollary}

\subsection{Spherical subalgebra}

As before, $c$ can be an arbitrary geometric $\CCC$-algebra. The \word{averaging idempotent} in $KW$ is the idempotent 
\begin{equation} \label{avg_idemp}
\eee \dopgleich \frac{1}{|W|} \sum_{w \in W} w \in KW \;.
\end{equation}
It is not hard to see that $\eee $ is indeed an idempotent. It is easy to see that we have
\begin{equation} \label{spherical_undeformed_relations}
(c \lbrack \fh \oplus \fh^* \rbrack \rtimes W) \eee  = c \lbrack \fh \oplus \fh^* \rbrack\eee  \quad \textnormal{and} \quad \eee (c \lbrack \fh \oplus \fh^* \rbrack \rtimes W)\eee  \simeq c \lbrack \fh \oplus \fh^* \rbrack^W \;,
\end{equation}
the latter being a $c$-algebra isomorphism given by right multiplication with $\eee$. Since $KW \subs \HHH_c$, we can consider $\eee $ as an idempotent in $\HHH_c$. 

\begin{lemma} \label{spherical_specialization}
There is a natural identification $\eee (\HHH_c)\eee  = (\eee \HHH\eee )_c$. 
\end{lemma} 

\begin{proof}
Since $c$ is a geometric $\CCC$-algebra, it is of the form $c = (\CCC/\fp)_{\fq/\fp}$ for some $\fp,\fq \in \CCC^\natural$ with $\fq \sups \fp$. Let us first consider the case $c = \CCC/\fp$. Then $\HHH_c = \HHH/\fp\HHH$. We have an exact sequence
\[
0 \rarr \fp \HHH \rarr \HHH \rarr \HHH/\fp \HHH \rarr 0 \;.
\]
Since multiplication with $\eee $ is an exact functor, see beginning of this section, we get an induced exact sequence
\[
0 \rarr \eee (\fp \HHH)\eee  \rarr \eee  \HHH \eee  \rarr \eee (\HHH/\fp \HHH)\eee  \rarr 0 \;.
\]
Clearly, $\eee (\fp \HHH)\eee  = \fp( \eee  \HHH \eee )$, so from the above exact sequence we get an isomorphism 
\[
(\eee \HHH\eee )_c = (\eee \HHH\eee )/\fp(\eee \HHH\eee ) \simeq \eee (\HHH_c)\eee  \;.
\]
Since localization is exact, we get the claimed isomorphism for general $c$.
\end{proof}

We denote the algebra in Lemma \ref{spherical_specialization} by $\UUU_c$ and called it the \word{spherical subalgebra} of $\HHH_c$. Note that it is not a subalgebra in the precise sense since the unit of $\UUU_c$ is equal to $\eee $. We call $\UUU \dopgleich \UUU_\CCC$ the \word{generic spherical subalgebra at $t=0$}. By Lemma \ref{spherical_specialization} we have
\begin{equation}
\UUU_c = c \otimes_{\CCC} \UUU \;.
\end{equation}
The filtration on $\HHH_c$ induces a filtration on the module $\HHH_c\eee $ and a filtration on the spherical subalgebra $\UUU_c$. Moreover, $\UUU_c$ inherits all gradings we defined on $\HHH_c$ since we always have $\deg(\eee ) = 0$.

\begin{lemma} \label{spherical_deformation}
The PBW morphism $\xi:c \lbrack \fh \oplus \fh^* \rbrack \rtimes W \overset{\sim}{\longrightarrow} \gr(\HHH_c)$ induces graded $c$-algebra isomorphisms
\begin{equation}
c \lbrack \fh \oplus \fh^* \rbrack \simeq \gr(\HHH_c\eee ) \quad \tn{and} \quad c \lbrack \fh \oplus \fh^* \rbrack^W \simeq \gr(\UUU_c) \;.
\end{equation}
We thus have $c$-module isomorphsims
\begin{equation}
c \lbrack \fh \oplus \fh^* \rbrack \simeq \HHH_ce \quad \tn{and} \quad c \lbrack \fh \oplus \fh^* \rbrack^W \simeq \UUU_c
\end{equation}
respecting all defined gradings. Hence, $\UUU_c$ is prime, noetherian, and a free $c$-module. Moreover, if $c$ is normal, so is $\UUU_c$.
\end{lemma}

\begin{proof}
This is just a consequence of (\ref{spherical_undeformed_relations}) and standard facts about reflections of properties of the associated graded to the original ring. The normality in case $c$ is normal is seen as follows: since $K \lbrack \fh \oplus \fh^* \rbrack^W$ is just a polynomial ring, it is geometrically normal and so the extension $c \lbrack \fh \oplus \fh^* \rbrack^W = c \otimes_K K \lbrack \fh \oplus \fh^* \rbrack$ is normal by \cite[Tag 06DF]{Stacks} if $c$ is normal. Since $c \lbrack \fh \oplus \fh^* \rbrack^W$ is a noetherian domain, it is already completely integrally closed. This property is now easily seen to be reflected to $\UUU_c \simeq \gr(c \lbrack \fh\oplus\fh^* \rbrack^W)$, implying that $\UUU_c$ is normal.
\end{proof}

\subsection{Double centralizer property} \label{dcp_section}

We use the notations about double centralizer properties from the Appendix \ref{dcp_appendix}. The following theorem was shown by Etingof and Ginzburg \cite[Theorem 1.5(iv)]{EG}.

\begin{theorem} \label{dcp_undeformed}
The pair $(c \lbrack \fh \oplus \fh^* \rbrack \rtimes W, \eee )$ satisfies the double centralizer property, i.e., the natural map
\[
c \lbrack \fh \oplus \fh^* \rbrack \rtimes W \rarr \End_{\eee (c \lbrack \fh \oplus \fh^* \rbrack \rtimes W)\eee }( (c \lbrack \fh \oplus \fh^* \rbrack \rtimes W) \eee ) = \End_{c \lbrack \fh \oplus \fh^* \rbrack^W}(c \lbrack \fh \oplus \fh^* \rbrack)
\]
is an isomorphism.
\end{theorem}

By passing from $\HHH_c$ to its associated graded $c \lbrack \fh \oplus \fh^* \rbrack^W$, one can transfer the double centralizer property also to the deformations $\HHH_c$. This is due to Etingof and Ginzburg \cite[Theorem 1.5(iv)]{EG}.\footnote{Again, it was proven over the complex numbers but it works also over $c$, see \cite[Théorème 4.5.3]{BR} or \cite[\S16]{Thiel-Diss}.}

\begin{theorem}[Etingof--Ginzburg] \label{dcp_theorem}
The pair $(\HHH_c, \eee )$ satisfies the double centralizer property, i.e., the natural map
\[
\HHH_c \rarr \End_{\eee \HHH_c\eee }(\HHH_c \eee ) = \End_{\UUU_c}(\HHH_c\eee )
\]
is an isomorphism.
\end{theorem}

From Lemma \ref{satake_iso} we obtain:

\begin{corollary} \label{satake_deformed}
Multiplication by $\eee $ induces an isomorphism $\ZZ(\HHH_c) \simeq \ZZ(\UUU_c)$ of $c$-algebras respecting the filtration and all defined gradings.
\end{corollary}


%

The following fact due to Etingof and Ginzburg \cite[Theorem 1.6]{EG} is of fundamental importance for the representation theory of rational Cherednik algebras at $t=0$.\footnote{We refer to \cite[Théorème 5.2.8]{BR} for a proof for $\UUU$, which clearly implies the general statement.}

\begin{theorem}[Etingof--Ginzburg] \label{spherical_commutative}
The spherical subalgebra $\UUU_c$ is commutative.
\end{theorem}

From Lemma \ref{spherical_deformation} and Corollary \ref{satake_deformed} we immediately obtain:

\begin{corollary} \label{center_invariants_iso}
We have a graded $c$-algebra isomorphism $c \lbrack \fh \oplus \fh^* \rbrack^W \rarr \gr(\ZZ(\HHH_c))$ and a filtered $c$-module isomorphism 
\begin{equation}
c \lbrack \fh \oplus \fh^* \rbrack^W \overset{\sim}{\longrightarrow} \ZZ(\HHH_c)
\end{equation}
respecting all defined gradings. In particular $\ZZ(\HHH_c)$ is an integral domain, noetherian, and a free $c$-module. Moreover, if $c$ is normal, so is $\ZZ(\HHH_c)$.
\end{corollary}

Let
\begin{equation}
\ZZZ \dopgleich \ZZ(\HHH)
\end{equation}
be the \word{generic center}. Since the spherical subalgebra behaves well under specialization by Lemma \ref{spherical_specialization}, we get the same for the center. 

\begin{corollary} \label{center_specialization}
We have a natural identification 
\begin{equation}
\ZZ(\HHH_c) = \ZZZ_c = c \otimes_{\CCC} \ZZZ \;.
\end{equation}
\end{corollary}


The following lemma is shown in \cite[Corollaire 5.2.11]{BR} for $\HHH$, and it then follows for general $\HHH_c$ by scalar extension using the fact that the center behaves well under specialization by the preceding corollary. 

\begin{lemma}\label{center_direct_summand}
The center $\ZZZ_c$ is a direct summand of $\HHH_c$ as a $\ZZZ_c$-module.
\end{lemma}

\subsection{A central subring}

We have $W \times W^* \subs \GL(\fh \oplus \fh^*)$. This is clearly a reflection group and we consider its invariant ring  
\begin{equation} 
\BBB \dopgleich K \lbrack \fh \oplus \fh^* \rbrack^{W \times W^*} = K \lbrack \fh \oplus \fh^* \rbrack^{W \times W^*} = K \lbrack \fh \rbrack^W \otimes_{K} K \lbrack \fh^* \rbrack^{W^*} \;,
\end{equation}
the so-called \word{bi-invariants} of $W$. We also consider its generic version
\begin{equation} \label{generic_bi_inv}
\PPP \dopgleich \CCC \otimes_K \BBB = \CCC \lbrack \fh \oplus \fh^* \rbrack^{W \times W^*} = \CCC \lbrack \fh \rbrack^W \otimes_{K} \CCC \lbrack \fh^* \rbrack^{W^*} \;.
\end{equation}
For any geometric $\CCC$-algebra $c$ the $c$-algebra
\begin{equation}
\PPP_c \dopgleich c \otimes_\CCC \PPP 
\end{equation}
is simply given by replacing $\CCC$ by $c$ in (\ref{generic_bi_inv}). The following lemma is straightforward.

\begin{lemma} \label{extension_subgroup_degree}
Let $V$ be a finite-dimensional $K$-vector space and let $G \subs \GL(V)$ be a finite subgroup such that $K \lbrack V \rbrack^G$ is a polynomial ring. Then for any subgroup $H$ of $G$ the extension $K \lbrack V \rbrack^G \subs K \lbrack V \rbrack^H$ is free of rank $\lbrack G:H \rbrack$. 
\end{lemma}

\begin{proof}
It is a standard fact that the extension $K \lbrack V \rbrack^G \subs K \lbrack V \rbrack$ is finite. Hence, $K \lbrack V \rbrack^G \subs K \lbrack V \rbrack^H$ must be finite. Since $K \lbrack V \rbrack^{G}$ is polynomial, we thus deduce that it is a graded Noether normalization of $K \lbrack V \rbrack^H$. As $K \lbrack V \rbrack^H$ is graded Cohen--Macaulay by the Eagon--Hochster theorem \cite{Eagon-Hochster}, see also \cite[Theorem 5.5.2]{NS-Invariant-theory}, it now follows from \cite[Corollary 6.7.7]{Smith-Invariant-theory} that the extension $K \lbrack V \rbrack^{G} \subs K \lbrack V \rbrack^H$ is already free. If $K(V)$ denotes the fraction field of $K \lbrack V \rbrack$, then $K(V)^G = \Frac(K \lbrack V \rbrack^G)$ by \cite[Proposition 1.1.1]{Benson-Invariant} and the field extension $K(V)^G \subs K(V)$ is Galois with Galois group $G$. It follows that the degree of the extension $K(V)^G \subs K(V)^H = \Frac(K \lbrack V \rbrack^H)$ is equal to $\lbrack G:H \rbrack$. Since we know that $K \lbrack V \rbrack^G \subs K \lbrack V \rbrack^H$ is free, it follows that its degree is equal to $\lbrack G:H \rbrack$.
\end{proof}

\begin{lemma} \label{invariants_extensions_degrees}
The following holds:
\begin{enumerate}
\item The extensions $c \lbrack \fh \rbrack^W \subs c \lbrack \fh \rbrack$ and $c \lbrack \fh^* \rbrack^W \subs c \lbrack \fh^* \rbrack$ are free of degree $|W|$.
\item The extension $\PPP_c \subs c \lbrack \fh \oplus \fh^* \rbrack$ is free of degree $|W|^2$.
\item The extension $\PPP_c \subs c \lbrack \fh \oplus \fh^* \rbrack^W$ is free of degree $|W|$.
\end{enumerate}
\end{lemma}

\begin{proof}
We just need to prove the assertions for $c=K$, the general result follows by extension to $c$. The first and second assertion simply follow from the Chevalley--Shephard--Todd theorem since $W \subs \GL(\fh)$ and $W \times W^* \subs \GL(\fh \oplus \fh^*)$ are reflection groups. The third assertion follows immediately from Lemma \ref{extension_subgroup_degree}. 
\end{proof}

\begin{theorem}[Etingof--Ginzburg, Gordon] \label{biinvariants_central}
The $c$-module isomorphism $c \lbrack \fh \oplus \fh^* \rbrack \rtimes W \rarr \HHH_c$ restricts to an injective $c$-algebra morphism $\PPP_c \hookrightarrow \HHH_c$. Moreover:
\begin{enumerate}
\item \label{biinvariants_central_1} The $c$-module isomorphisms $c \lbrack \fh \oplus \fh^* \rbrack \rtimes W \overset{\sim}{\longrightarrow} \HHH_c$ and $c \lbrack \fh \oplus \fh^* \rbrack^W \overset{\sim}{\longrightarrow} \ZZZ_c$ are isomorphisms of $\PPP_c$-modules. We thus have the following commutative diagram
\[
\begin{tikzcd}
c \lbrack \fh \oplus \fh^* \rbrack \rtimes W \arrow{r}{\sim} &  \HHH_c \\
c \lbrack \fh \oplus \fh^* \rbrack^{W} \arrow[hookrightarrow]{u} \arrow{r}{\sim} & \ZZZ_c \arrow[hookrightarrow]{u} \\
\PPP_c \arrow[hookrightarrow]{u} \arrow[equals]{r} & \PPP_c \arrow[hookrightarrow]{u} \\
c \arrow[hookrightarrow]{u} \arrow[equals]{r}  & c  \arrow[hookrightarrow]{u}
\end{tikzcd}
\]
\item \label{biinvariants_central_2} $\HHH_c$ is a free $\PPP_c$-module of rank $|W|^3$,
\item \label{biinvariants_central_3} $\ZZZ_c$ is a free $\PPP_c$-module of rank $|W|$,
\item \label{biinvariants_central_5} $\dim \ZZZ_c = \dim c + 2 \dim \fh$.
\end{enumerate}
\end{theorem}

\begin{proof}
 The fact that $\PPP_c$ is a central subalgebra of $\HHH_c$ was proven by Etingof and Ginzburg \cite{EG} for $K=c=\bbC$. The proof given by Gordon \cite{Gordon-Baby} works word for word for arbitrary $c$. Assertion (\ref{biinvariants_central_1}) follows directly from the definition of the isomorphisms in Theorem \ref{pbw_theorem} and Corollary \ref{center_invariants_iso}. Assertions (\ref{biinvariants_central_2}) and (\ref{biinvariants_central_3}) now follow immediately from Lemma \ref{invariants_extensions_degrees}. Since $\PPP_c \subs \ZZZ_c$ and $\PPP_c \subs c \lbrack \fh \oplus \fh^* \rbrack$ are finite, we have $\dim \ZZZ_c = \dim \PPP_c = \dim c \lbrack \fh \oplus \fh^* \rbrack = \dim c + 2\dim \fh$, using the fact that $c$ is noetherian.  
\end{proof}

\begin{remark}
The commutative diagram in Theorem \ref{biinvariants_central} illustrates that the extension $\CCC \lbrack \fh \oplus \fh^* \rbrack^{W \times W^*} \subs \ZZZ$ \textit{deforms} the extension $\CCC \lbrack \fh \oplus \fh^* \rbrack^{W \times W^*} \subs \CCC \lbrack \fh \oplus \fh^* \rbrack^W$ over $\CCC$. This, in a sense, is the starting point of the Calogero--Moser cells by Bonnafé and Rouquier \cite{BR,BR-2}.
\end{remark}

\begin{corollary} \label{center_cm}
If $c$ is Cohen--Macaulay (resp.\ Gorenstein), so is $\ZZZ_c$.
\end{corollary}

\begin{proof}
Assume that $c$ is Cohen--Macaulay. A noetherian commutative ring is Cohen--Macaulay if and only if all its localizations in maximal ideals are Cohen--Macaulay. It thus suffices to show that the localization $(\ZZZ_c)_\fM$ is Cohen--Macaulay for every maximal ideal $\fM$ of $\ZZZ_c$. Since $\PPP_c \subs \ZZZ_c$ is finite, $\fM$ contracts to a maximal ideal $\fm$ of $\PPP_c$. Since $c$ is Cohen--Macaulay by assumption and $ \PPP_c$ is a polynomial ring over $c$, also $\PPP_c$ is Cohen--Macaulay by \cite[Tag 00ND]{Stacks}, hence $(\PPP_c)_\fm$ is Cohen--Macaulay. Since $\PPP_c \subs \ZZZ_c$ is free, also $(\PPP_c)_\fm \subs (\ZZZ_c)_\fM$ is free, in particular faithfully flat, and now \cite[Exercise 2.1.23]{Bruns-Herzog} implies that $(\ZZZ_c)_\fM$ is also Cohen--Macaulay. 
A similar proof shows that $\ZZZ_c$ is Gorenstein if $c$ is Gorenstein. 
\end{proof}

The following general lemma is the so-called \word{Artin--Tate lemma}. A proof can be found in \cite[Theorem 11.4]{DF}. 

\begin{lemma}[Artin--Tate]
Let $R \subs C \subs A$ be rings and suppose that $C$ is central in $A$, $A$ is finitely generated as an $R$-algebra, $A$ is finitely generated as a $C$-module, and $A$ is noetherian. Then $C$ is a finitely generated $R$-algebra.
\end{lemma}

\begin{corollary} \label{center_finite_type}
$\ZZZ_c$ is a finitely generated $c$-algebra.
\end{corollary}

\begin{proof}
We apply the Artin--Tate lemma to the extension $c \subs \ZZZ_c \subs \HHH_c$. Since $c \langle \fh \otimes \fh^* \rangle \rtimes W$ is a finitely generated $c$-algebra and $\HHH_c$ is a quotient thereof, also $\HHH_c$ is a finitely generated $c$-algebra. Since $\PPP_c \subs \HHH_c$ is finite and $ \PPP_c \subs \ZZZ_c$ by Theorem \ref{biinvariants_central}, also $\ZZZ_c \subs \HHH_c$ is finite. Finally, we know from Corollary \ref{RCA_ring_properties} that $\HHH_c$ is noetherian.
\end{proof}

\subsection{Symmetrizing trace}

The following theorem is essentially due to Brown, Gordon, and Stroppel \cite[\S3]{BGS}.\footnote{In \cite{BGS} this is proven over $\bbC$ but one can also give a completely general proof, see \cite[\S17D]{Thiel-Diss} and the argumentation below.}

 \begin{theorem}[Brown–Gordon–Stroppel] \label{bgs_frob}
For any geometric $\CCC$-algebra $c$ the rational Cherednik algebra $\HHH_c$ is a free symmetric Frobenius $\PPP_c$-algebra. 
\end{theorem}

We will give some more details about this, beginning with some general remarks. First, recall that a commutative finite-dimensional $\bbN$-graded connected\footnote{This means that the homogeneous component of $A$ of degree zero is just equal to $K$.} algebra $A$ is called a \textit{Poincaré duality algebra} if there is some $N \in \bbN$ such that 
\begin{enumerate}
\item $A_i = 0$ for all $i>N$.
\item $\dim_K A_N = 1$
\item The pairing $A_i \otimes A_{N-i} \rarr A_N$ induced by multiplication is non-degenerate for all $0 \leq i \leq n$.
\end{enumerate}
Any non-zero element of the top-degree component $A_N$ is called a \textit{fundamental class} of $A$. The following lemma is elementary, see \cite[Lemma 17.32]{Thiel-Diss} for a proof.

\begin{lemma} \label{pda_symmetrizing_trace}
Let $\Omega \in A_N$ be a fundamental class of a Poincaré duality algebra $A$. Then:
\begin{enumerate}
\item The map $\Phi_\Omega:A \rarr K$ mapping $a$ to the coefficient of $\Omega$ in the $N$-th homogeneous part of $a$ makes $A$ into a symmetric Frobenius $K$-algebra. 
\item Let $\mathscr{X} \dopgleich (x_i)_{i=1}^n$ be a homogeneous basis of $A$ which is sorted increasingly by degree and such that $x_n = \Omega$. Let $\mathscr{X}_{i,j}^k$ be the structure constants of $A$ with respect to $\mathscr{X}$, i.e., $x_ix_j = \sum_{k=1}^n \mathscr{X}_{i,j}^k x_k$. Then $\mathscr{X}^n \dopgleich (\mathscr{X}_{i,j}^n) \in \Mat_n(K)$ is invertible and if we define $\mathscr{Y}^n \dopgleich (\mathscr{X}^n)^{-1}$ and  $y_j \dopgleich \sum_{k=1}^n \mathscr{Y}_{k,j}^n x_k$, then $\mathscr{Y} \dopgleich (y_i)_{i=1}^n$ is also a homogeneous basis of $A$ with $\Phi_\Omega(x_i y_j) = \delta_{i,j}$, i.e., $\mathscr{Y}$ is the dual basis of $\mathscr{X}$ with respect to $\Phi_\Omega$ and this basis is in particular again homogeneous.
\end{enumerate}
\end{lemma}

It is a standard fact that the coinvariant algebra $K \lbrack \fh \rbrack^{\mrm{co}(W)}$ of a reflection group $W \subs \GL(\fh)$ is a Poincaré duality algebra with top degree being equal to $\#\Ref(W)$ and fundamental class given by $\ol{\Omega} \dopgleich \prod_{s \in \Ref(W)} \alpha_s$, see \cite[\S20]{Kane}. By Lemma \ref{pda_symmetrizing_trace} we thus have a symmetrizing trace $\Phi_{\ol{\Omega}}: K \lbrack \fh \rbrack^{\mrm{co}(W)} \rarr K$ and we know that it has a pair of \textit{homogeneous} dual bases.  Let $\Omega \in K \lbrack \fh \rbrack$ be a homogeneous lift of $\ol{\Omega}$. We can now define a unique $K\lbrack \fh \rbrack^W$-linear map $\Phi_\Omega:K \lbrack \fh \rbrack \rarr K \lbrack \fh \rbrack^W$ as follows: choose any homogeneous $K \lbrack \fh \rbrack^W$-basis $\mathscr{X}$ of $K\lbrack \fh \rbrack$ containing $\Omega$ and define $\Phi_\Omega(f) = \delta_{f,\Omega}$ for $f \in \mathscr{X}$. With the arguments in the proof of \cite[Lemma 3.5]{BGS}, see also \cite[Proposition 17.33]{Thiel-Diss}, one can now show the following.

\begin{theorem}[Brown–Gordon–Stroppel]
The extension $K\lbrack \fh \rbrack^W \subs K \lbrack \fh \rbrack$ is a free symmetric Frobenius extension. A symmetrizing trace is given by $\Phi_\Omega$.
\end{theorem}

\begin{remark}
We do not know if an arbitrary homogeneous lift of a pair of dual bases for $\Phi_{\ol{\Omega}}$ yields a pair of dual bases for $\Phi_\Omega$. The argumentation in \cite{BGS} uses a rather ingenious condition for Frobenius extensions, \cite[Proposition 2.2]{BGS}, which, however, has the drawback of not giving an explicit dual basis. 
\end{remark}

By scalar extension we see that $c\lbrack \fh \rbrack^W \subs c \lbrack \fh \rbrack$ is a free symmetric Frobenius extension with symmetrizing trace $\Phi_{\Omega,c} \dopgleich c \otimes_K \Phi_\Omega$ for any geometric $\CCC$-algebra $c$. Similar statements hold of course also for the dual representation of $W$, so $K \lbrack \fh^* \rbrack^{\mrm{co}(W)}$ is a Poincaré duality algebra and a homogeneous lift $\Omega^*$ of a fundamental class $\ol{\Omega^*}$ of $K \lbrack \fh^* \rbrack^{\mrm{co}(W)}$ yields a symmetrizing trace $\Phi_{\Omega^*,c}:c \lbrack \fh^* \rbrack \rarr c \lbrack \fh^* \rbrack^W$ for the free symmetric Frobenius extension $c \lbrack \fh^* \rbrack^W \subs c \lbrack \fh \rbrack$. 

As a last ingredient for Theorem \ref{bgs_frob} recall that the group algebra $KW$ is a symmetric Frobenius $K$-algebra with symmetrizing trace $\Phi_W$ defined by $\Phi_W(w) \dopgleich \delta_{w,1}$ for $w \in W$. By scalar extension we get a symmetrizing trace $\Phi_{W,c}:cW \rarr c$.

We can now patch the three symmetrizing traces $\Phi_{\Omega,c}$, $\Phi_{\Omega^*,c}$, and $\Phi_{W,c}$ into one symmetrizing trace $\Phi_{\Omega,\Omega^*,c}:\HHH_c \rarr \PPP_c$ as follows: we choose a homogeneous $K \lbrack \fh \rbrack^W$-basis $\mathscr{X}$ of $K \lbrack \fh \rbrack$ containing $\Omega$ and a homogeneous $K \lbrack \fh^* \rbrack^W$-basis $\mathscr{Y}$ of $K \lbrack \fh^* \rbrack$ containing $\Omega^*$, then by the PBW theorem $(fgw)_{f \in \mathscr{X},g \in \mathscr{Y}, w \in W}$ is a $\PPP_c$-basis of $\HHH_c$ and  we define 
\begin{equation}
\Phi_{\Omega,\Omega,c}(f g w) \dopgleich \Phi_{\Omega,c}(f) \Phi_{\Omega^*,c}(g) \Phi_{W,c}(w) = \delta_{f,\Omega} \delta_{g,\Omega^*}\delta_{w,1} \;.
\end{equation} 

A more precise form of Theorem \ref{bgs_frob} is now:

\begin{theorem}[Brown--Gordon--Stroppel] \label{bgs_frob_explicit}
The map $\Phi_{\Omega,\Omega^*,c}$ is a symmetrizing trace for the $\PPP_c$-algebra $\HHH_c$. Moreover, if $\mathscr{X}^\vee$ is a dual basis of $\mathscr{X}$ with respect to $\Phi_\Omega$ and $\mathscr{Y}^\vee$ is a dual basis of $\mathscr{Y}$ with respect to $\Phi_{\Omega^*}$, then for $f \in \mathscr{X}$, $g \in \mathscr{Y}$, and $w \in W$ the dual element of $fgw$ with respect to $\Phi_{\Omega,\Omega^*,c}$ is given by
\begin{equation} \label{bgs_dual_basis}
(fgw)^\vee = w^{-1}g^\vee f^\vee \;.
\end{equation}
\end{theorem}

\begin{remark}
Even though a dual basis is given in \cite[Corollary 3.7]{BGS}\footnote{Note that there is a typo in the dual basis.}, it seems that the arguments given there do not prove this. The problem here is the use of the condition in \cite[Proposition 2.2]{BGS} which does not give a dual basis. One can, however, rearrange the arguments in the proof of \cite[Proposition 3.5]{BGS} as follows: first, one shows that $\Phi_{\Omega,\Omega^*,c}$ is symmetric and then one can verify directly that (\ref{bgs_dual_basis}) is a dual basis. This is done in \cite[Theorem 17.36]{Thiel-Diss}. We think that the approach of first proving symmetry may also yield to further simplifications in \cite{BGS} for the cases of symmetric Frobenius extensions.
\end{remark}

\subsection{Geometry and representation theory}
The finiteness of $\HHH_c$ over its center $\ZZZ_c$ has a lot of implications for the representation theory of $\HHH_c$ and connects it to the geometry of the spectrum of $\ZZZ_c$. For any geometric $\CCC$-algebra $c$ we call $\XXX_c \dopgleich \ZZZ_c^\natural = \Spec(\ZZZ_c)$ the \word{Calogero--Moser space} in $c$. By the results above this is an integral $c$-scheme of finite type. The morphism $\boldsymbol{\Upsilon}_c:\XXX_c \rarr \PPP_c^\natural$ induced by the embedding $\PPP_c \hookrightarrow \ZZZ_c$ is finite, flat, surjective, and closed. We call $\XXX \dopgleich \XXX_\CCC = \Spec(\ZZZ)$ the \word{generic Calogero--Moser space} and set $\boldsymbol{\Upsilon} \dopgleich \boldsymbol{\Upsilon}_\CCC$. 

First of all, the finiteness of $\HHH_c$ over its center immediately implies:

\begin{lemma}
The ring $\HHH_c$ is a PI ring for any geometric $\CCC$-algebra $c$.
\end{lemma}

A consequence is:

\begin{lemma}
Let $c$ be a geometric $\CCC$-algebra. Let $S \in \Irr \HHH_c$. Then 
\[
\Ann_{\HHH_c}(S) \in \Max(\HHH_c) \;, \quad \Ann_{\ZZZ_c}(S) \in \Max(\ZZZ_c) \;, \quad \Ann_{\PPP_c}(S) \in \Max(\PPP_c) \;.
\]
Hence, there is a natural decomposition
\begin{equation} \label{irr_decomposition}
\Irr \HHH_c = \coprod_{p \in \Max(\PPP_c)} \Irr \HHH_c/p \HHH_c = \coprod_{ \substack{\fm \in \boldsymbol{\Upsilon}^{-1}_c(p) \\ p \in \Max(\PPP_c)}} \Irr \HHH_c/\fm \HHH_c \;,
\end{equation}
the second refining the first.
\end{lemma}

\begin{proof}
An annihilator $\mathfrak{P} \dopgleich \Ann_{\HHH_c}(S)$ of a simple module $S \in \Irr \HHH_c$ is by definition a left primitive ideal of $\HHH_c$. The quotient $\HHH_c/\mathfrak{P}$ is a primitive PI ring (primitive following from $\mathfrak{P}$ being primitive and PI following from the fact that quotients of PI rings are again PI). An application of Kaplansky’s theorem \cite[13.3.8]{MR} now implies that $\HHH_c/\mathfrak{P}$ is a central simple algebra over its center, thus in particular simple and so $\mathfrak{P}$ has to be maximal. That $\Ann_{\ZZZ_c}(S)$ and $\Ann_{\PPP_c}(S)$ are also maximal follows now from the fact that the extensions $\ZZZ_c \subs \HHH_c$ and $\PPP_c \subs \ZZZ_c$ are finite and thus satisfy going up for prime ideals.
\end{proof}

The decomposition (\ref{irr_decomposition}) shows that to describe the simple $\HHH_c$-modules, it is sufficient to describe the simple modules for the \word{restrictions}
\begin{equation}
\HHH_c^p \dopgleich \HHH_c/p \HHH_c
\end{equation}
in maximal ideals $p$ of $\PPP_c$. Note that $\HHH_c^p$ is a finite-dimensional algebra over the field $\PPP_c/p$, and it is even finite-dimensional over $K$ if $c$ is of finite type over $K$, e.g. $c = \CCC$ or $c = K$. We want to give an interpretation of the set-theoretic fiber $\boldsymbol{\Upsilon}_c^{-1}(p)$. Since $\HHH_c^p$ is a finite-dimensional algebra over a field, it has a block decomposition. We denote the set of its blocks by $\Bl(\HHH_c^p)$. Furthermore, we define
\begin{equation}
\ZZZ_c^p \dopgleich \ZZZ_c/p \ZZZ_c 
\end{equation}
for $p \in \Max(\PPP_c)$. 
We have a natural morphism $\ZZZ_c^p \rarr \msf{Z}(\HHH_c^p)$ and since $\ZZZ_c$ is a direct summand of $\HHH_c$ by Lemma \ref{center_direct_summand}, this morphism is in fact injective so that we can identify $\ZZZ_c^p \subs \msf{Z}(\HHH_c^p)$. Note that $\ZZZ_c^p$ is in general \textit{not} equal to the center of $\HHH_c^p$. We will analyze this defect in Theorem \ref{morita_dcp_closed_points}. The spectrum of $\ZZZ_c^p$ is the scheme-theoretic fiber $\boldsymbol{\Upsilon}_c^*(p)$. The following theorem is in principle a consequence of a general result due to Müller \cite{Mueller-Cliques}. In \cite[Theorem F]{Thiel-Blocks} we have given full details, and an application of this to the finite free extension $\PPP_c \subs \HHH_c$ yields:

\begin{theorem} \label{mueller_blocks}
Let $c$ be a geometric $\CCC$-algebra and let $p \in \Max(\PPP_c)$. There is a canonical bijection
\begin{equation}
\Bl(\HHH_c^p) \simeq \Bl(\ZZZ_c^p) \simeq \boldsymbol{\Upsilon}_c^{-1}(p) \;.
\end{equation}
The simple $\HHH_c$-modules lying over $\fm \in \boldsymbol{\Upsilon}_c^{-1}(p)$, i.e., those annihilated by $\fm$, are precisely the simple modules in the corresponding block of $\HHH_c^p$.
\end{theorem}
\begin{figure}[htbp] 
\begin{center}
\begin{tikzpicture}
    \fill[green,opacity=0.8] (0,5) ellipse (3cm and 0.7cm);
    \coordinate[label=left:${\mathsf{Max}_{l} \ \HHH_{c}} \simeq \Irr \HHH_c $] (H) at (5.95,5); 
    
   \fill[green,opacity=0.8] (0,2) ellipse (2cm and 0.5cm); 
    \coordinate[label=left:${\XXX_c}$] (Z) at (4.63,2.1); 
     
    \fill[green,opacity=0.8] (0,0) ellipse (1cm and 0.3cm);
    \coordinate[label=left:${\PPP_c}$] (sZ) at (4.63,0); 
    
    \coordinate (l1) at (4.2,4.7);
    \coordinate (l2) at (4.2,2.4);
    \draw[->,thick] (l1) -- (l2); 
    
    \coordinate (l3) at (4.2,1.8);
    \coordinate (l4) at (4.2,0.3);
    \draw[->,thick] (l3) -- (l4); 
    \coordinate[label=${\boldsymbol{\Upsilon}_c}$] (B) at (5.1,0.8);

    \coordinate (m0) at (0,0); 
    \coordinate[label=${p}$] (zero) at (-0.3,-0.3); 
     \fill (m0) circle (2pt); 
    \coordinate (m1) at (-0.8,2.2);
    \coordinate (m2) at (0.1,2);
    \coordinate (m3) at (1.4,2.1);
    \fill (m1) circle (2pt);
    \fill (m2) circle (2pt);
    \fill (m3) circle (2pt);
    \draw (m0) -- (m1); 
    \draw (m0) -- (m2); 
    \draw (m0) -- (m3);  
  
    \coordinate (m11) at (-2,5);
    \coordinate (m12) at (-1.8,5.2);
    \coordinate (m13) at (-1.4,5.3);
    \coordinate (m14) at (-1,5);
    \fill (m11) circle (2pt);
    \fill (m12) circle (2pt);
    \fill (m13) circle (2pt);
    \fill (m14) circle (2pt); 
    \draw (m1) -- (m11);
    \draw (m1) -- (m12);
    \draw (m1) -- (m13);
    \draw (m1) -- (m14); 
    
    \coordinate (m21) at (0,5);
    \fill (m21) circle (2pt);
    \draw (m2) -- (m21); 
    
    \coordinate (m31) at (1.5,5.1);
    \coordinate (m32) at (1.9,4.8);
    \fill (m31) circle (2pt);
    \fill (m32) circle (2pt);
    \draw (m3) -- (m31);
    \draw (m3) -- (m32); 
    
    \fill[blue,opacity=0.5] (0,5) ellipse (0.3cm and 0.2cm);
    \fill[blue,opacity=0.5] (-1.5,5.05) ellipse (0.7cm and 0.4cm);
    \fill[blue,opacity=0.5] (1.65,4.95) ellipse (0.55cm and 0.3cm);
\end{tikzpicture}
\end{center} \caption{Illustration of Theorem \ref{mueller_blocks}: The blocks of $\HHH_c^p$ correspond to the closed points in the Calogero--Moser space $\XXX_c$ lying over $p$. The simple modules lying in each block are illustrated by the darker areas at the top.} \label{mueller_picture}
\end{figure}
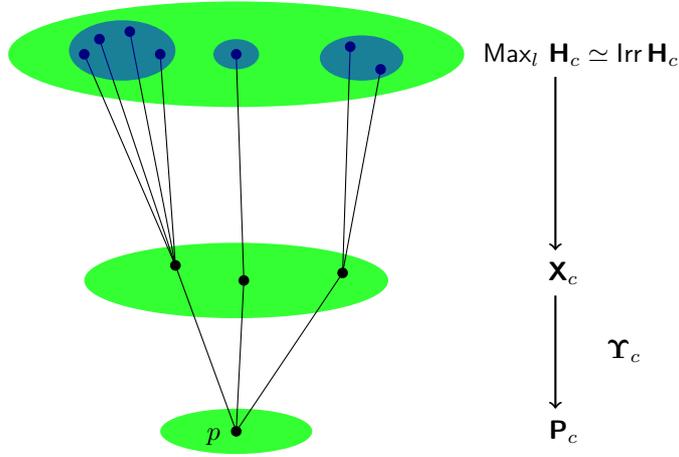

We can now state the following striking theorem.

\begin{theorem}[Etingof--Ginzburg, Brown--Goodearl, Brown] \label{rca_reps_geometry}
Suppose that $K$ is algebraically closed and that $c \in \mscr{C}$, i.e., $c$ is a closed point of $\CCC^\natural$. Then the PI-degree of $\HHH_c$ is equal to $|W|$, so $\dim_{K} S \leq |W|$ for any $S \in \Irr \HHH_c$. Moreover, setting $\fm_S \dopgleich \Ann_{\ZZZ_c} (S)$ and $p_S \dopgleich \Ann_{\PPP_c}(S)$, the following are equivalent:
\begin{enumerate}
\item \label{rca_reps_geometry:smooth} $\fm_S$ is a smooth point of $\XXX_c$.
\item \label{rca_reps_geometry:az} $\dim_{K} S = |W|$.
\item \label{rca_reps_geometry:reg} $S \simeq KW$ as $KW$-modules. 
\item \label{rca_reps_geometry:block} The block of $\HHH_c^{p_S}$ containing $S$ contains up to isomorphism no further simple modules. 
\end{enumerate}

\end{theorem}

\begin{proof}
The implication (\ref{rca_reps_geometry:smooth}) $\Rightarrow$ (\ref{rca_reps_geometry:az}) follows from the result \cite[Theorem 1.7(iv)]{EG} by Etingof and Ginzburg. An application of a result by Brown and Goodearl \cite[Lemma 3.3]{Brown-Goodearl-Homological} shows that the (dense) Azumaya locus is contained in the smooth locus of $\XXX_c$. The Azumaya locus is the locus over which simple $\HHH_c$-modules have maximal dimension, and this maximal dimension is equal to the PI-degree of $\HHH_c$. Hence, the maximal dimension, and thus the PI-degree, is equal to $|W|$. This furthermore proves the implication (\ref{rca_reps_geometry:az}) $\Rightarrow$ (\ref{rca_reps_geometry:smooth}). The implication (\ref{rca_reps_geometry:reg}) $\Rightarrow$ (\ref{rca_reps_geometry:az}) is obvious, and the implication (\ref{rca_reps_geometry:smooth}) $\Rightarrow$ (\ref{rca_reps_geometry:reg}) was shown by Etingof and Ginzburg in \cite[Theorem 1.7(iv)]{EG}. Suppose that $\dim_K S = |W|$ and let $\fM \dopgleich \Ann_{\HHH_c}(S)$. We already know that the PI-degree of $\HHH_c$ is equal to $|W|$, so it follows from \cite[Theorem III.1.6]{Brown-Goodearl-Quantum} that $\HHH_c/\fM \simeq \Mat_{|W|}(K)$ and that $\fM = \fm_S \HHH_c$. Hence, there is only one simple $\HHH_c$-module lying over $\fM$ and there is only one maximal ideal of $\HHH_c$ lying over $\fm_S$. Consequently, there is only one simple $\HHH_c$-module lying over $\fm_S$. Now, it follows from Theorem \ref{mueller_blocks} that in the block of $\HHH_c^{p_S}$ containing $S$ the is up to isomorphism no other simple module. This proves the implication (\ref{rca_reps_geometry:az}) $\Rightarrow$ (\ref{rca_reps_geometry:block}). The converse of this implication is due to Brown and appeared in \cite[Lemma 7.2]{Gordon-Baby} by Gordon. 
\end{proof}

We do not know to which extend Theorem \ref{rca_reps_geometry} can be generalized to arbitrary fields $K$ and geometric $\CCC$-algebras $c$. At least we can show that the result about the PI-degree holds in general:

\begin{theorem} \label{pi_degree}
The PI-degree of $\HHH_c$ is equal to $|W|$ for any geometric $\CCC$-algebra~$c$.
\end{theorem}

\begin{proof}
Recall from Corollary \ref{RCA_ring_properties} that $\HHH_c$ is prime and from Corollary \ref{center_invariants_iso} that $\ZZZ_c$ is an integral domain. Let $E_c$ be the fraction field of $\PPP_c$ and let $F_c$ be the fraction field of $\ZZZ_c$. By Posner's theorem, see \cite[Theorem 13.6.5]{MR}, the set of non-zero central elements in $\HHH_c$ is an Ore set so that the localization $Q_c \dopgleich (\ZZZ_c \setminus \lbrace 0 \rbrace)^{-1} \HHH_c = F_c \otimes_{\ZZZ_c} \HHH_c$ exists. By \cite[13.3.6]{MR}, Kaplansky's theorem \cite[Theorem 13.3.8]{MR},  Posner's theorem, and \cite[13.6.7]{MR} we have
\begin{equation}
\pideg(\HHH_c) = \pideg(Q_c) = \sqrt{ \dim_{F_c} Q_c} \;.
\end{equation}
Since $\HHH_c$ is a free $\PPP_c$-module of rank $|W|^3$ by Theorem \ref{biinvariants_central}, it follows that 
\[
\dim_{E_c}( E_c \otimes_{\PPP_c} \HHH_c) = |W|^3 \;.
\]
Similarly, since $\ZZZ_c$ is a free $\PPP_c$-module of rank $|W|$ by Theorem \ref{biinvariants_central}, it follows that 
\[
\dim_{E_c} F_c = |W| \;.
\]
Furthermore, we have
\[
E_c \otimes_{\PPP_c} \HHH_c = (\PPP_c \setminus \lbrace 0 \rbrace)^{-1} \HHH_c \subs (\ZZZ_c \setminus \lbrace 0 \rbrace)^{-1} \HHH_c = F_c \otimes_{\ZZZ_c} \HHH_c = Q_c \;.
\]
Hence,
\[
|W|^3 = \dim_{E_c} (E_c \otimes_{\PPP_c} \HHH_c) \leq \dim_{E_c} Q_c = |W| \cdot \dim_{F_c} Q_c
\]
and therefore
\[
\pideg(\HHH_c) = \sqrt{\dim_{F_c} Q_c} \geq \sqrt{|W|^2} = |W| \;.
\]
Now, let
\[
\msf{PIGen}(\HHH) \dopgleich \lbrace \mathfrak{P} \in \Spec(\HHH) \mid \pideg(\HHH/\mathfrak{P}) = \pideg(\HHH) \rbrace \subs \Spec(\HHH) \;.
\]
This set contains the zero ideal (recall that $\HHH$ is prime) and therefore it is dense in $\Spec(\HHH)$. Let $f:\Spec(\HHH) \rarr \Spec(\PPP)$ be the morphism induced by the inclusion $\PPP \hookrightarrow \HHH$. This morphism is finite, thus surjective by \cite[Theorem 10.2.9]{MR}. Hence, the image $f(\msf{PIGen}(\HHH))$ is dense in $\Spec(\PPP)$. Since $\Spec(\PPP)$ is just an affine space, the $K$-points are dense in $\Spec(\PPP)$, see also \cite[Corollary III.18.3]{Borel}, and so there is a $K$-point $\fm$ in $f(\msf{PIGen}(\HHH))$. Since $\PPP(K) = \mathscr{C} \times (\fh/W \times \fh^*/W)(K)$, the point $\fm$ corresponds to a point $(c,p)$. Since $f$ is surjective, there is some $\mathfrak{P} \in \msf{PIGen}(\HHH)$ with $f(\mathfrak{P}) = \fm$. Now, $\HHH/\mathfrak{P}$ is a quotient of $\HHH/\fm\HHH$ and $\HHH/\fm\HHH$ is a quotient of $\HHH_c$, so it follows that
\[
\pideg(\HHH) = \pideg(\HHH/\mathfrak{P}) \leq \pideg(\HHH/\fm \HHH) \leq \pideg(\HHH_c) \leq \pideg(\bbC \otimes_K \HHH_c) \;,
\]
the last inequality following from the fact that $\HHH_c \subs \HHH_c^\bbC = \bbC \otimes_K \HHH_c$. From Theorem \ref{rca_reps_geometry} we know that $\pideg(\HHH_c^\bbC) \leq |W|$, so this shows us that $\pideg(\HHH) \leq |W|$. If $c$ is any geometric $\CCC$-algebra, then $\HHH_c$ is a (central) localization of a quotient of $\HHH$, so 
\[
\pideg(\HHH_c) \leq \pideg(\HHH) \leq |W| \;.
\]
by \cite[Corollary I.13.3]{Brown-Goodearl-Quantum} which states that subrings of the fraction field of a prime PI ring have the same PI-degree as this ring. Combined with the inequality above, we get $\pideg(\HHH_c) = |W|$. 
\end{proof}

\subsection{Restricted double centralizer property}

We want to analyze the defect of $\ZZZ_c^p$ not being equal to the center of $\HHH_c^p$. Let $c$ be a geometric $\CCC$-algebra and let $p \in \Max(\PPP_c)$. The averaging idempotent $\eee$ from (\ref{avg_idemp}) clearly also defines an idempotent in $\HHH_c^p$. We want to consider the double centralizer property of $(\HHH_c^p,\eee)$. Since $\eee K \lbrack \fh \oplus \fh^* \rbrack \eee = K \lbrack \fh \oplus \fh^* \rbrack^W$, it follows that $\eee \PPP_c \eee = \PPP_c$, so in particular $p \subs \UUU_c$ and we can define
\begin{equation}
\UUU_c^p = \UUU_c/p  \UUU_c \;.
\end{equation}
Since $\UUU_c$ is commutative by Theorem \ref{spherical_commutative}, so is $\UUU_c^p$. Let
\begin{equation}
E_c^p: \HHH_c^p\tn{-}\msf{mod} \rarr \UUU_c^p\tn{-}\msf{mod}
\end{equation}
be the \word{restricted Schur functor} given by multiplication with $\eee $, see Section \ref{dcp_appendix}. Recall from Lemma \ref{dcp_lemma} that $(\HHH_c^p, \eee )$ satisfies the double centralizer property if and only if $E_c^p$ induces an equivalence between $\HHH_c^p\tn{-}\mathsf{proj}$ and $\UUU_c^p\tn{-}\mathsf{proj}$. We record the following property about the restrictions $\HHH_c^p$ which  follows from Theorem \ref{bgs_frob} and the standard fact, see \cite[III.4.8, Corollary 1]{BG}, that a symmetrizing trace is induced on any quotient.

\begin{corollary} \label{restrictions_symmetric}
The restriction $\HHH_c^p$ is a symmetric Frobenius $(\PPP_c/p)$-algebra.
\end{corollary}

\begin{theorem} \label{morita_dcp_closed_points}
Let $c$ be a Gorenstein geometric $\CCC$-algebra and let $p \in \Max(\PPP_c)$. The following are equivalent:
\begin{enumerate}
\item \label{morita_dcp_closed_points:az} The families of $\HHH_c^p$ are singletons. 
\item \label{morita_dcp_closed_points:morita} $E_c^p$ is an equivalence.
\item \label{morita_dcp_closed_points:dcp} $(\HHH_c^p, \eee )$ satisfies the double centralizer property.
\end{enumerate}
If this holds, then $\msf{Z}(\HHH_c^p) \simeq \UUU_c^p \simeq \ZZZ_c^p$.
\end{theorem}

\begin{proof}  
The implication (\ref{morita_dcp_closed_points:morita}) $\Rightarrow$ (\ref{morita_dcp_closed_points:dcp}) is obvious by Lemma \ref{dcp_lemma}.
Let $\fm_1,\ldots,\fm_r$ be the maximal ideals of $\ZZZ_c$ lying over $p$. We will make use of the general ring theoretic fact that if $e$ is an idempotent of a ring $A$, then $A$ is Morita equivalent to $eAe$ if and  only if $A = AeA$. 

Suppose that (\ref{morita_dcp_closed_points:az}) holds. Then for each $i$ there is only one maximal ideal $M_i$ of $\HHH_c$ lying over $\fm_i$. It thus follows from Müller's theorem \cite[Theorem A.2.2]{Jategaonkar-Localization} that each $M_i$ is localizable and that the localization $(\HHH_c)_{M_i}$ is equal to the localization of $\HHH_c$ in the multiplicative set $\ZZZ_c \setminus \fm_i$, which is just the localization $(\HHH_c)_{\fm_i}$ of the $\ZZZ_c$-module $\HHH_c$ in $\fm_i$. The proof of \cite[Lemma 7.2]{Gordon-Baby}\footnote{The arguments work word for word for general $c$.} shows that $(\HHH_c)_{\fm_i}$ is Morita equivalent to $\eee (\HHH_c)_{\fm_i} \eee$. This implies that $(\HHH_c)_{\fm_i} = (\HHH_c)_{\fm_i} \eee (\HHH_c)_{\fm_i} = (\HHH_c \eee \HHH_c)_{\fm_i}$. Reduction in $p$ yields $(\HHH_c^p)_{\fm_i} =  (\HHH_c^p \eee \HHH_c^p)_{\fm_i}$. Since the $\fm_i$ are the maximal ideals of $\ZZZ_c/p\ZZZ_c$, this shows that $\HHH_c^p = \HHH_c^p \mbf{e} \HHH_c^p$. Hence, $E_c^p$ is an equivalence and this proves the implication (\ref{morita_dcp_closed_points:az}) $\Rightarrow$ (\ref{morita_dcp_closed_points:morita}).

Now, suppose that (\ref{morita_dcp_closed_points:dcp}) holds. Assume that (\ref{morita_dcp_closed_points:az}) were not be true, i.e., the blocks are not singletons. Since $(\HHH_c,\eee)$ satisfies the double centralizer property, multiplication with $\eee$ yields an isomorphism $\ZZZ_c \overset{\sim}{\rarr} \eee \HHH_c \eee$ by Lemma \ref{satake_iso}. Reduction in $p$ gives an isomorphism $\ZZZ_c^p \overset{\sim}{\rarr} \eee \HHH_c^p \eee$. This implies in particular that the dimension of these two finite-dimensional $(\PPP_c/p)$-vector spaces is equal. We know from Theorem \ref{biinvariants_central} that $\ZZZ_c$ is a free $\PPP_c$-module. For any $i$ we thus get a flat morphism $(\PPP_c)_{p} \rarr (\ZZZ_c)_{\fm_i}$ of noetherian local rings. By Corollary \ref{center_cm} the center $\ZZZ_c$ is Gorenstein, hence $(\ZZZ_c)_{\fm_i}$ is Gorenstein and now \cite[3.3.15]{Bruns-Herzog} shows that $(\ZZZ_c^p)_{\fm_i} = (\ZZZ_c/p\ZZZ_c)_{\fm_i}$ is Gorenstein. This implies that $\ZZZ_c^p$ itself is Gorenstein. As this ring is artinian and thus zero-dimensional, it must be self-injective. The canonical morphism $\ZZZ_c^p \rarr \ZZ(\HHH_c^p)$ is injective since $\ZZZ_c$ is a direct summand of $\HHH_c$ as a $\ZZZ_c$-module by Lemma \ref{center_direct_summand}. By assumption, there is a block of $\HHH_c^p$ containing more than one simple module and so \cite[2.8 and 2.9]{Gordon-pos-char} shows that the injective morphism $\ZZZ_c^p \rarr \ZZ(\HHH_c^p)$ is \textit{not} surjective. For this argument we use that $\HHH_c^p$ is a symmetric algebra by Corollary \ref{restrictions_symmetric}. In particular, the dimension of the $(\PPP_c/p)$-vector space $\ZZ(\HHH_c^p)$ is larger than the dimension of $\eee \HHH_c^p \eee$ by the above. Hence, $\ZZ(\HHH_c^p)$ and $\eee \HHH_c^p\eee$ cannot be isomorphic. But this contradicts by Lemma \ref{satake_iso} the assumption that $(\HHH_c^p,\eee)$ satisfies the double centralizer property. This proves (\ref{morita_dcp_closed_points:dcp}) $\Rightarrow$ (\ref{morita_dcp_closed_points:az}).
\end{proof}

%
%

\section{Restricted rational Cherednik algebras} \label{RRCA}

In the last paragraphs we have seen that we can study $\HHH_c$ by studying its restrictions $\HHH_c^p$ in closed points $p$ of $\PPP_c^\natural$. We want to consider one particular restriction which is naturally defined for any complex reflection group $W$ and for any geometric $\CCC$-algebra $c$. In general, assume that $p \in \BBB^\natural(K) = (\fh/W)(K) \times (\fh^*/W)(K)$ is a $K$-point. Considered as a maximal ideal in $\BBB$ this generates for any geometric $\CCC$-algebra $c$ a prime ideal in $\PPP_c = c \otimes_K \BBB$ with quotient $\PPP_c/p\PPP_c = c$, so this is a maximal ideal in $\PPP_c$ if and only if $c$ is a field. But even if $p$ is not a maximal ideal in $\PPP_c$, we can still define the restrictions
\begin{equation}
\HHH_c^p \dopgleich \HHH_c/p\HHH_c \;, \quad \ZZZ_c^p \dopgleich \ZZZ_c/p\ZZZ_c \;, \quad \UUU_c^p \dopgleich \UUU_c/p \UUU_c
\end{equation}
as before. These are all naturally $c$-algebras. In particular, for $c = \CCC$ we get \word{generic versions}
\begin{equation}
\HHH^p \dopgleich \HHH/p\HHH \;, \quad \ZZZ^p \dopgleich \ZZZ/p\ZZZ \;, \quad \UUU^p \dopgleich \UUU/p \UUU
\end{equation}
of these restrictions. These are $\CCC$-algebras, so we can consider their scalar extensions to geometric $\CCC$-algebras $c$. This is compatible with the former construction, i.e.,
\begin{equation}
(\HHH^p)_c = c \otimes_\CCC \HHH^p = \HHH_c^p \;, \quad (\ZZZ^p)_c = c \otimes_\CCC \ZZZ^p = \ZZZ_c^p \;, \quad (\UUU^p)_c = c \otimes_\CCC \UUU^p = \UUU_c^p \;.
\end{equation}

There seems to be only one particular point of $\fh/W \times \fh^*/W$ which is naturally defined for any complex reflection group $W$, namely the origin $0 \dopgleich (0,0)$. The corresponding maximal ideal of $\BBB$ is given by
\begin{equation} \label{origin_ideal}
K \lbrack \fh \rbrack_+^W K \lbrack \fh^* \rbrack^W + K \lbrack \fh^* \rbrack^W_+ K \lbrack \fh \rbrack^W \;,
\end{equation}
where $(-)_+$ denotes the ideal generated by the invariants with non-zero constant term (the augmentation ideal with respect to the natural $\bbN$-grading on the invariant rings). It is generated by one (any) system of fundamental invariants of $K \lbrack \fh \rbrack^W$ and of $K \lbrack \fh^* \rbrack^W$. For a geometric $\CCC$-algebra $c$ we set:
\begin{equation}
\ol{\HHH}_c \dopgleich \HHH_c^0 \;, \quad \ol{\ZZZ}_c \dopgleich \ZZZ_c^0 \;, \quad \ol{\UUU}_c \dopgleich \UUU_c^0
\end{equation}
and
\begin{equation}
\ol{\HHH} \dopgleich \HHH^0 \;, \quad \ol{\ZZZ} \dopgleich \ZZZ^0 \;, \quad \ol{\UUU} \dopgleich \UUU^0
\end{equation}
for their generic versions. We call $\ol{\HHH}_c$ \textit{the} restricted rational Cherednik algebra in $c$. 

To make this clear: $\ol{\HHH}_c$ is simply the $c$-algebra with presentation as in Definition \ref{gen_RCA_def} (with $\ccc(s)$ replaced by its image in $c$) where we additionally mod out a system of fundamental invariants for $K \lbrack \fh \rbrack^W$ and for $K \lbrack \fh^* \rbrack^W$.

\subsection{Grading}
The ideal in (\ref{origin_ideal}) is clearly a homogeneous ideal of $\BBB$ with respect to the $\bbZ$-grading. Hence, the quotient $\ol{\HHH}_c$ is naturally $\bbZ$-graded.

\subsection{Triangular decomposition}

Recall that the \word{coinvariant algebra} of $W$ is the quotient
\begin{equation} \label{coinv_def}
K \lbrack \fh \rbrack^{\mrm{co}(W)} \dopgleich K \lbrack \fh \rbrack/K \lbrack \fh \rbrack^W_+ K \lbrack \fh \rbrack \;.
\end{equation} 
The action  of $W$ makes this into a graded $W$-module and it is a classical fact that it is isomorphic to the regular $W$-module $KW$. 
 We can extend the coinvariant algebra to any geometric $\CCC$-algebra $c$, this just amounts to replacing $K$ by $c$ in (\ref{coinv_def}). The same is of course true for the dual coinvariant algebra $K \lbrack \fh^* \rbrack^{\mrm{co}(W)}$.

\begin{corollary} \label{rrca_first_props:pbw}
The $c$-module isomorphism $\HHH_c \simeq c \lbrack \fh \oplus \fh^* \rbrack \rtimes W \simeq c \lbrack \fh \rbrack \otimes_{c} c W \otimes_{c} c \lbrack \fh^* \rbrack$ from the PBW theorem (Theorem \ref{pbw_theorem}) induces a $c$-module  isomorphism
\begin{equation} \label{rrca_triangular_dec}
\ol{\HHH}_c \simeq c \lbrack \fh \rbrack^{\mrm{co}(W)} \otimes_{c} cW \otimes_{c} c \lbrack \fh^* \rbrack^{\mrm{co}(W)} \;.
\end{equation}
This isomorphism respects the filtration and all defined gradings. In particular, $\ol{\HHH}_c$ is a free $c$-module with
\begin{equation}
\dim_{c} \ol{\HHH}_c = |W|^3 \;.
\end{equation}
\end{corollary}

\subsection{Automorphisms}

The following point of view is due to Bonnafé and Rouquier \cite[\S4.6]{BR}. Let $\Aut_K(\HHH)$ be the group of $K$-algebra automorphisms of $\HHH$. The $(\bbN \times \bbN)$-grading on $\HHH$ induces a natural group morphism
\begin{equation} \label{rca_automorphism}
\msf{bigr}: K^\times \times K^\times \rarr \Aut_K(\HHH) 
\end{equation}
given by 
\begin{equation}
\msf{bigr}_{\xi,\xi'}(h) = \xi^i \xi'^j h
\end{equation}
for $(\xi,\xi') \in K^\times \times K^\times$ and $(\bbN \times \bbN)$-homogeneous $h \in \HHH$ of degree $(i,j)$. Explicitly, we have
\begin{equation} \label{rca_auto_def}
\msf{bigr}_{\xi,\xi'}(y) = \xi y, \quad \msf{bigr}_{\xi,\xi'}(x) = \xi' x, \quad \msf{bigr}_{\xi,\xi'}(f) = \xi \xi' f, \quad \msf{bigr}_{\xi,\xi'}(w) = w
\end{equation}
for $y \in \fh$, $x \in \fh^*$, $f \in \mscr{C}^*$, and $w \in W$. The action of $\msf{bigr}_{\xi,1}$ on $\mscr{C}^*$ is just the natural action of $K^\times$ on the $K$-vector space $\mscr{C}^*$. This naturally induces a $K^\times$-action on $\CCC$. If $c$ is a geometric $\CCC$-algebra and $\xi \in K^\times$, we denote by $\xi c$ the geometric $\CCC$-algebra obtained by twisting the $\CCC$-action with the automorphism defined by $\xi$ on $\CCC$. The $K$-algebra automorphism $\msf{bigr}_{\xi,1}$ of $\HHH$ then yields a $K$-algebra isomorphism
\begin{equation}
\HHH_c \overset{\sim}{\longrightarrow} \HHH_{\xi c} \;.
\end{equation}
The automorphism $\msf{bigr}_{\xi,1}$ stabilizes the origin $0$ of $\fh/W\times\fh^*/W$, so it induces a $K$-algebra automorphism of $\ol{\HHH}$ and a $K$-algebra isomorphism
\begin{equation} \label{rrca_scaling}
\ol{\HHH}_c \overset{\sim}{\longrightarrow} \ol{\HHH}_{\xi c} \;.
\end{equation}

\subsection{Symmetrizing trace}

As a special case of Corollary \ref{restrictions_symmetric} we obtain:

\begin{corollary}
For any geometric $\CCC$-algebra $c$ the restricted rational Cherednik algebra $\ol{\HHH}_c$ is a free symmetric Frobenius $c$-algebra.
\end{corollary}

\subsection{Baby Verma modules}

The triangular decomposition in (\ref{rrca_triangular_dec}) implies a rich combinatorial structure for the representation theory of $\ol{\HHH}_c$. This was discovered by Gordon \cite{Gordon-Baby}, using a general theory of Holmes and Nakano \cite{HN}. We will now go through this theory applied to restricted rational Cherednik algebras and refer to \cite{HN} for further aspects of the abstract setting. We note that there are some recent developments by Bellamy and the author \cite{Bellamy-Thiel-weight} in the abstract setting with applications to restricted rational Cherednik algebras on which we can comment here only briefly. \\

We assume throughout that $c$ is a prime ideal of $\CCC$. From the defining relations of $\ol{\HHH}_c$ it is clear that the vector space isomorphism above embeds both $\msf{k}_\CCC(c)\lbrack \fh \rbrack^{\mrm{co}(W)} \rtimes W$ and $\msf{k}_\CCC(c)\lbrack \fh^* \rbrack^{\mrm{co}(W)} \rtimes W$ into $\ol{\HHH}_c$ as graded subalgebras. In particular, all three algebras in  the triangular decomposition are naturally graded subalgebras of $\ol{\HHH}_c$:
\begin{equation} \label{rrca_trinagular_dec_2}
\ol{\HHH}_c \simeq \msf{k}_\CCC(c)\lbrack \fh \rbrack^{\mrm{co}(W)} \otimes_{\msf{k}_\CCC(c)} \underbrace{\msf{k}_\CCC(c) W \otimes_{\msf{k}_\CCC(c)} \msf{k}_\CCC(c)\lbrack \fh^* \rbrack^{\mrm{co}(W)}}_{\simeq \msf{k}_\CCC(c)\lbrack \fh^* \rbrack^{\mrm{co}(W)} \rtimes W} \;.
\end{equation}
We will consider them with the induced $\bbZ$-grading. By definition, $\msf{k}_\CCC(c)\lbrack \fh \rbrack^{\mrm{co}(W)}$ is concentrated in non-negative degree, the group algebra $\msf{k}_\CCC(c)W$ is concentrated in degree zero, and $\msf{k}_\CCC(c)\lbrack \fh^* \rbrack^{\mrm{co}(W)}$ is concentrated in non-positive degree. Let $\msf{k}_\CCC(c)\lbrack \fh^* \rbrack^{\mrm{co}(W)}_-$ be the ideal of $\msf{k}_\CCC(c)\lbrack \fh^* \rbrack^{\mrm{co}(W)}$ formed by the elements of negative degree.\footnote{If we consider $\msf{k}_\CCC(c)\lbrack \fh^* \rbrack^{\mrm{co}(W)}$ with the natural $\bbN$-grading, this is the augmentation ideal $\msf{k}_\CCC(c)\lbrack \fh^* \rbrack^{\mrm{co}(W)}_+$. This is a bit confusing but consistent with the current context.} The group algebra $\msf{k}_\CCC(c)W$ is clearly the (graded) quotient of $\msf{k}_\CCC(c)\lbrack \fh^* \rbrack^{\mrm{co}(W)} \rtimes W$ by $\msf{k}_\CCC(c)\lbrack \fh^* \rbrack^{\mrm{co}(W)}_-$. Hence, we can consider any $\msf{k}_\CCC(c)W$-module naturally as a $\msf{k}_\CCC(c)\lbrack \fh^* \rbrack^{\mrm{co}(W)} \rtimes W$-module by inflation, i.e., $\msf{k}_\CCC(c)\lbrack \fh^* \rbrack^{\mrm{co}(W)}_-$ acts trivially. We thus have a sequece of  functors
\begin{equation}
\begin{tikzcd}
&  & \ol{\HHH}_c\tn{-}\msf{(gr)mod} \\
KW\tn{-}\msf{(gr)mod} \arrow{r}[swap]{\sim}{\otimes_{\msf{k}_\CCC(c)}} & \msf{k}_\CCC(c)W\tn{-}\msf{(gr)mod} \arrow{r}{\msf{Inf}} & (\msf{k}_\CCC(c)\lbrack \fh^* \rbrack^{\mrm{co}(W)} \rtimes W)\tn{-}\msf{(gr)mod}  \arrow{u}[swap]{\otimes_{\msf{k}_\CCC(c)\lbrack \fh^* \rbrack^{\mrm{co}(W)} \rtimes W}}
\end{tikzcd}
\end{equation}
Here, we used the classical fact that $KW$ splits already over $K$ due to a theorem by Benard \cite{Benard} and Bessis \cite{Bessis-corps}, so $KW\tn{-}\msf{mod} \simeq \msf{k}_\CCC(c)W\tn{-}\msf{mod}$ naturally via scalar extension. We denote the composition of the functors above by $\Delta_c$. To summarize, we have
\begin{equation} \label{baby_verma_def}
\Delta_c(M) \dopgleich \ol{\HHH}_c \otimes_{\msf{k}_\CCC(c)\lbrack \fh^* \rbrack^{\mrm{co}(W)} \rtimes W} M \;.
\end{equation}
for a $KW$-module $M$. This is the so-called \word{baby Verma module} associated to $M$.\footnote{We note that Gordon \cite{Gordon-Baby} uses $\Delta_c(M) = \ol{\HHH}_c \otimes_{\msf{k}_\CCC(c)\lbrack \fh \rbrack^{\mrm{co}(W)} \rtimes W} M$. The results one gets are essentially the same ``up to twist''. We follow \cite[\S9.2]{BR} here and refer to \cite{Bellamy-Thiel-weight} for comments on the two possible ways to define Verma modules.}

\begin{theorem}[Holmes--Nakano, Gordon] \label{baby_verma_heads}
If $\lambda \in \Irr W$, then $\Delta_c(\lambda)$ has a simple head, i.e., 
\begin{equation}
L_c(\lambda) \dopgleich \Delta_c(\lambda)/\Rad \Delta_c(\lambda)
\end{equation}
is a simple $\ol{\HHH}_c$-module. Furthermore, the map $\lambda \mapsto L_c(\lambda)$ induces a bijection between the sets of isomorphism classes of simple modules.
\end{theorem}




In \cite{HN} an algebraically closed base field is assumed but everything still works over an arbitrary field, see \cite{Bellamy-Thiel-weight}. Particularly useful is the following result.

\begin{proposition}[Bonnafé--Rouquier] \label{rrca_split}
Each $L_c(\lambda)$ is absolutely simple, i.e., it remains simple under any field extension. In particular, the $\msf{k}_{\CCC}(c)$-algebra $\ol{\HHH}_c$ splits.
\end{proposition}

The parametrization of the simple $\ol{\HHH}_c$-modules by the simple $W$-modules allows us to attach the following important invariant to $W$, which was introduced by Gordon \cite{Gordon-Baby}. Namely, since the block structure of $\ol{\HHH}_c$ partitions the set $\Irr \HHH_c$ into \word{families} we can pull back this partition along the map $\Irr W \rarr \Irr \ol{\HHH}_c$, $\lambda \mapsto L_c(\lambda)$, and in this way we get a $c$-dependent partition of $\Irr W$ into so-called \word{Calogero--Moser $c$-families}. We denote this partition by $\CM_c$. Recall from Theorem \ref{mueller_blocks} that we have a canonical bijection 
\begin{equation}
\CM_c \simeq \boldsymbol{\Upsilon}_c^{-1}(0) \;.
\end{equation}

Let us take a closer look at baby Verma modules. By construction, the baby Verma module is naturally graded. The radical of $\Delta_c(\lambda)$ is a graded submodule, so $L_c(\lambda)$ is naturally graded, too. It is a standard fact that the graded simple modules, i.e., the simple objects in the graded module category $\ol{\HHH}_c\tn{-}\msf{grmod}$, are simply the shifts $L_c(\lambda)[n]$ for $n \in \bbZ$. The graded module category $\ol{\HHH}_c\tn{-}\msf{grmod}$ is an abelian category of finite length. In particular, the \word{graded Grothendieck group} $\msf{G}_0^{\msf{gr}}(\ol{\HHH}_c) \dopgleich \msf{K}_0(\ol{\HHH}_c\tn{-}\msf{grmod})$ is defined.  By the aforementioned, $\msf{G}_0^{\msf{gr}}(\ol{\HHH}_c)$ is a free module of rank $\# \Irr W$ over the Laurent polynomial ring $\bbZ \lbrack q,q^{-1} \rbrack$, while the non-graded Grothendieck group $\msf{G}_0(\ol{\HHH}_c) \dopgleich \msf{K}_0(\ol{\HHH}_c\tn{-}\msf{mod})$ is a free $\bbZ$-module of the same rank. Directly from the definition we obtain:

\begin{lemma} \label{baby_verma_as_graded_w}
There is a canonical isomorphism 
\begin{equation} \label{baby_verma_as_graded_w_iso}
\Delta_c(\lambda) \simeq \msf{k}_\CCC(c)\lbrack \fh \rbrack^{\mrm{co}(W)} \otimes_{c} \lambda
\end{equation}
of graded $\msf{k}_\CCC(c)W$-modules. 
\end{lemma}
 
 In particular,  $\Delta_c(\lambda)$ is concentrated in non-negative degree.  Recall that the coinvariant algebra is as a $W$-module just the regular module, so from Lemma \ref{baby_verma_as_graded_w} we get an isomorphism 
\begin{equation}
\Delta_c(\lambda) \simeq \msf{k}_\CCC(c)W \otimes_{c} \lambda
\end{equation}
of non-graded $\msf{k}_\CCC(c)W$-modules. To give a further structural result, first recall that the \word{fake degree} $f_\chi$ of an irreducible character $\chi \in \Irr W$ is defined as the graded multiplicity of $\chi^*$ in the coinvariant algebra $K\lbrack \fh \rbrack^{\mrm{co}(W)}$, so:
\begin{equation}
f_\chi \dopgleich \lbrack K \lbrack \fh \rbrack^{\mrm{co}(W)} : \chi^* \rbrack_W^{\mrm{gr}} \;.
\end{equation}
It follows from Molien's formula that the fake degree is given by the following explicit formula:
\begin{equation} \label{fake_degree_formula}
f_\chi(q) = \prod_{i=1}^n (1-q^{d_i}) \frac{1}{|W|} \sum_{w \in W} \frac{\chi^*(w)}{\det(1-wq)} \in \bbN \lbrack q \rbrack \;,
\end{equation}
where $d_1,\ldots,d_n$ are the \word{degrees} of $W$, i.e., the degrees of one (any) system of fundamental invariants of $K \lbrack \fh \rbrack^W$. From this formula it is easy to see that $f_\chi$ is equal to the graded multiplicity of $\chi$ in the dual coinvariant algebra $K \lbrack \fh^* \rbrack^{\mrm{co}(W)}$. Note that the above formula can be simplified to a summation over the conjugacy classes of $W$ and is thus easy to compute.

\begin{corollary} \label{baby_verma_graded_w}
The graded $W$-module character of $ \Delta_c(\lambda) $ is given by
\begin{equation} \label{baby_verma_graded_w_formula}
\lbrack \Delta_c(\lambda) \rbrack_W^{\mrm{gr}} = \sum_{\mu \in \Irr W} f_\mu(q) \lbrack \mu^* \otimes \lambda \rbrack \;,
\end{equation}
where $f_\chi(q)$ is the fake degree of $\chi$ defined in (\ref{fake_degree_formula}).\footnote{This formula is also given in \cite{Gordon-Baby}. But note that this looks slightly different as the definition of the baby Verma module in \textit{loc. cit.} is dual to ours.}
\end{corollary}

\begin{proof}
Let $(\cdot,\cdot)$ denote the scalar product of characters of $W$. Then
\begin{align*}
& \lbrack \Delta_c(\lambda) \rbrack_W^{\mrm{gr}} = \sum_{\eta \in \Irr W} \lbrack K \lbrack \fh \rbrack^{\mrm{co}(W)} \otimes \lambda : \eta \rbrack_W^{\mrm{gr}} \lbrack \eta \rbrack = \sum_{i \in \bbN} \sum_{\eta \in \Irr W} \lbrack K \lbrack \fh \rbrack^{\mrm{co}(W)}_i \otimes \lambda : \eta \rbrack q^i \lbrack \eta \rbrack \\ \displaybreak[0]
& = \sum_{i \in \bbN} \sum_{\eta \in \Irr W} ( K \lbrack \fh \rbrack^{\mrm{co}(W)}_i \otimes \lambda , \eta) q^i \lbrack \eta \rbrack = \sum_{i \in \bbN} \sum_{\eta \in \Irr W} ( K \lbrack \fh \rbrack^{\mrm{co}(W)}_i , \lambda^* \otimes \eta) q^i \lbrack \eta \rbrack \\ \displaybreak[0]
& = \sum_{i \in \bbN} \sum_{\eta \in \Irr W} \sum_{\mu \in \Irr W} (\lambda^* \otimes \eta, \mu) (K \lbrack \fh \rbrack^{\mrm{co}(W)}_i, \mu) q^i \lbrack \eta \rbrack \\ \displaybreak[0]
& = \sum_{i \in \bbN} \sum_{\eta \in \Irr W} \sum_{\mu \in \Irr W} (\lambda^* \otimes \eta, \mu^*) (K \lbrack \fh \rbrack^{\mrm{co}(W)}_i, \mu^*) q^i \lbrack \eta \rbrack \\ \displaybreak[0]
& = \sum_{\eta \in \Irr W} \sum_{\mu \in \Irr W} (\lambda^* \otimes \eta, \mu^*) \sum_{i \in \bbN} (K \lbrack \fh \rbrack^{\mrm{co}(W)}_i, \mu^*) q^i \lbrack \eta \rbrack \\ \displaybreak[0]
& = \sum_{\eta \in \Irr W} \sum_{\mu \in \Irr W} (\lambda^* \otimes \eta, \mu^*) f_\mu(q) \lbrack \eta \rbrack = \sum_{\eta \in \Irr W} \sum_{\mu \in \Irr W} ( \eta, \lambda \otimes \mu^*) f_\mu(q) \lbrack \eta \rbrack  \\ \displaybreak[0]
& =  \sum_{\mu \in \Irr W} f_\mu(q) \lbrack \lambda \otimes \mu^*\rbrack \;.  
\end{align*} 
\end{proof}

The next lemma is elementary (see \cite{Bellamy-Thiel-weight} for a proof).

\begin{lemma} \label{verma_degree_zero}
The baby Verma module $\Delta_c(\lambda)$ is generated by any non-zero element in degree zero. Moreover, $\lambda \simeq L_c(\lambda)_0 \simeq \Delta_c(\lambda)_0$ as $\msf{k}_\CCC(c)W$-modules. 
\end{lemma}

The following result is proven in \cite{Bellamy-Thiel-weight}.

\begin{proposition}[Bellamy--T.] \label{graded_w_restriction_injective}
The restriction map $\msf{G}_0^{\msf{gr}}(\ol{\HHH}_c) \rarr \msf{G}_0^{\msf{gr}}(KW)$ induced by the embedding $KW \hookrightarrow \ol{\HHH}_c$ is injective. 
\end{proposition}

The reader might want to take a look at \cite{Bellamy-Thiel-weight} for several implications of Proposition \ref{graded_w_restriction_injective}. We note that it does \textit{not} hold in the non-graded case. 

\begin{lemma} \label{simple_W_struct_baby_Verma_struct}
The matrix 
\begin{equation}
\mathbfsf{A} \dopgleich ( \lbrack L_c(\mu) : \eta \rbrack)_W^{\mrm{gr}} )_{\mu,\eta \in \Irr W} \in \Mat_{\#\Irr W}(\bbQ(q))
\end{equation}
is invertible. Moreover, if we define for $\lambda \in \Irr W$ the vector
\begin{equation}
\mathbfsf{v}_\lambda \dopgleich ( \sum_{\mu \in \Irr W} f_\mu(q) \lbrack \mu^* \otimes \lambda : \eta \rbrack)_{\eta \in \Irr W} \in \bbQ(q)^{\Irr W} \;,
\end{equation}
then the unique $\mathbfsf{w}_\lambda \in \bbQ(q)^{\Irr W}$ with $\mathbfsf{v}_\lambda = \mathbfsf{A} \mathbfsf{w}_\lambda$ is given by
\begin{equation}
\mathbfsf{w}_\lambda = ( \lbrack \Delta_c(\lambda) : L_c(\mu) \rbrack^{\mrm{gr}})_{\mu \in \Irr W} \in \bbQ(q)^{\Irr W} \;.
\end{equation}
In other words, knowing the graded $W$-module structure of the simple $\ol{\HHH}_c$-modules is equivalent to knowing the graded decomposition matrices of the baby Verma modules of $\ol{\HHH}_c$ by solving a system of linear equations over $\bbQ(q)$.
\end{lemma}

\begin{proof}
Recall that $\msf{G}_0^{\msf{gr}}(\ol{\HHH}_c)$ is a free $\bbZ \lbrack q,q^{-1} \rbrack$-module with basis $(L_c(\eta))_{\eta \in \Irr W}$ and that similarly $\msf{G}_0^{\msf{gr}}(KW)$ is a free $\bbZ \lbrack q,q^{-1} \rbrack$-module with basis $(\eta)_{\eta \in \Irr W}$. Due to the injectivity of the restriction map $\chi:\msf{G}_0^{\msf{gr}}(\ol{\HHH}_c) \rarr \msf{G}_0^{\msf{gr}}(KW)$ by Proposition \ref{graded_w_restriction_injective}, it follows that the image $(\chi(L_c(\eta)))_{\eta \in \Irr W}$ is a linearly independent subset of $\msf{G}_0^{\mrm{gr}}(KW)$. Extending scalars to $\bbQ(q)$ we thus get two bases of $\bbQ(q) \otimes_{\bbZ \lbrack q,q^{-1} \rbrack} \msf{G}_0^{\mrm{gr}}(KW)$ and the matrix $\mathbfsf{A}$ is simply the base change matrix, thus invertible. The vector $\mathbfsf{v}_\lambda$ just corresponds to the representation of $\lbrack \Delta_c(\lambda) \rbrack$ in the basis $(\eta)_{\eta \in \Irr W}$ by Corollary \ref{baby_verma_graded_w}, and base change with $\mathbfsf{A}^{-1}$ gives the representation in the basis $(L_c(\eta))_{\eta \in \Irr W}$.
\end{proof}

\begin{remark}
Corollary \ref{simple_W_struct_baby_Verma_struct} can of course also be restricted to a single block $B$ of $\ol{\HHH}_c$ resulting in a linear system over $\bbQ(q)^{|\!\Irr B|}$.
\end{remark}

Recall the $K$-algebra isomorphism $\ol{\HHH}_c \rarr \ol{\HHH}_{\xi c}$ for $\xi \in K^\times$ from (\ref{rrca_scaling}). This induces an equivalence of categories
\begin{equation} \label{scaling_cat_equivalence}
\,^\xi(-): \ol{\HHH}_{\xi c}\tn{-}\msf{(gr)mod} \overset{\sim}{\longrightarrow} \ol{\HHH}_c\tn{-}\msf{(gr)mod} \;.
\end{equation}
By definition, see (\ref{rca_auto_def}), the isomorphism induces the identity on the group algebras. More precisely, the diagram
\[
\begin{tikzcd}
\ol{\HHH}_c \arrow{r}{\sim} & \ol{\HHH}_{\xi c} \\
\msf{k}_\CCC(c) W \arrow{r}{\sim}\arrow[hookrightarrow]{u}  & \msf{k}_\CCC(\xi c) W \arrow[hookrightarrow]{u}\\
KW \arrow[equals]{r} \arrow[hookrightarrow]{u} & KW \arrow[hookrightarrow]{u}  
\end{tikzcd}
\]
commutes. Hence, the category equivalence (\ref{scaling_cat_equivalence}) restricts to the identity on the category of (graded) $KW$-modules.

\begin{lemma} \label{verma_simple_scaling}
For $\xi \in K^\times$ and $\lambda \in \Irr W$ we have 
\begin{equation}
\lbrack \,^\xi \Delta_{\xi c}(\lambda) \rbrack = \lbrack \Delta_c(\lambda) \rbrack \in \msf{G}_0^{\mrm{gr}}(\ol{\HHH}_c)
\end{equation}
and
\begin{equation}
\lbrack \,^\xi L_{\xi c}(\lambda) \rbrack = \lbrack L_c(\lambda) \rbrack \in \msf{G}_0^{\mrm{gr}}(\ol{\HHH}_c) \;.
\end{equation}
\end{lemma}

\begin{proof}
As explained above, the graded $W$-module structures of $\,^\xi\Delta_{\xi c}(\lambda)$ and $\Delta_c(\lambda)$ coincide. The graded $W$-modules structure of a baby Verma module is, independently of $c$, always given by (\ref{baby_verma_graded_w_formula}) in Lemma \ref{baby_verma_graded_w}. Hence, $\,^\xi\Delta_{\xi c}(\lambda)$ is a graded $\ol{\HHH}_c$-module which has the same graded $W$-module structure as $\Delta_c(\lambda)$. The first assertion thus follows from Proposition \ref{graded_w_restriction_injective}. We have an exact sequence
\[
0 \rarr \Rad(\Delta_{\xi c}(\lambda)) \rarr \Delta_{\xi c}(\lambda) \rarr L_{\xi c}(\lambda) \rarr 0 
\]
and since $\,^\xi (-)$ is exact, this induces an exact sequence
\[
0 \rarr \,^\xi \Rad(\Delta_{\xi c}(\lambda)) \rarr \,^\xi \Delta_{\xi c}(\lambda) \rarr \,^\xi L_{\xi c}(\lambda) \rarr 0 \;.
\]
Since $\,^\xi L_{\xi c}(\lambda)$ is simple, we must have $\,^\xi L_{\xi c}(\lambda) = L_c(\mu)$ for some $\mu$. The sequence above induces an exact sequence
\[
0 \rarr (\,^\xi \Rad(\Delta_{\xi c}(\lambda)))_0 \rarr (\,^\xi \Delta_{\xi c}(\lambda))_0 \rarr (\,^\xi L_{\xi c}(\lambda))_0 \rarr 0
\]
of the degree zero components as $W$-modules. By Lemma \ref{verma_degree_zero} and the fact that $\,^\xi \Delta_{\xi c}(\lambda)$ and $\Delta_c(\lambda)$ are isomorphic as graded $W$-modules, we know that $(\,^\xi \Delta_{\xi c}(\lambda))_0 \simeq (\Delta_c(\lambda))_0 \simeq \lambda$ and $(\,^\xi L_{\xi c}(\lambda))_0 \simeq L_c(\mu)_0 \simeq \mu$. Hence, the above sequence shows that there exists a surjective $W$-module morphism $\lambda \twoheadrightarrow \mu$. This is only possible if $\lambda \simeq \mu$, i.e., $\,^\xi L_{\xi c}(\lambda) \simeq L_c(\lambda)$.
%
%
\end{proof}

Bonnafé and Rouquier \cite[Proposition 13.4.2]{BR} have proven an astonishing property of the non-graded decomposition matrices of baby Verma modules. We cannot go into details here about the proof but note that this is a consequence of the Calogero--Moser cell theory developed in \cite{BR}.

\begin{theorem}[Bonnafé--Rouquier] \label{rank_1_property}
For every fixed Calogero–Moser $c$-family $\mathcal{F}$ there is an $\ol{\HHH}_c$-module $L_c(\mathcal{F})$ such that
\begin{equation}
\lbrack \Delta_c(\lambda) \rbrack = \dim \lambda \cdot \lbrack L_c(\mathcal{F}) \rbrack 
\end{equation}
in $\msf{G}_0(\ol{\HHH}_c)$ for all $\lambda \in \mathcal{F}$.
\end{theorem}

There are two immediate consequences:

\begin{corollary}
Let $\lambda,\mu \in \Irr W$. Then $\lbrack \Delta_c(\lambda):L_c(\mu) \rbrack$ is a multiple of $\dim \lambda$.
\end{corollary}

\begin{corollary} \label{diagonal_enough}
If $\lambda,\mu \in \Irr W$ lie in the same Calogero--Moser $c$-family, then
\begin{equation}
\dim \mu \ \lbrack \Delta_c(\lambda) : L_c(\eta) \rbrack = \dim \lambda \ \lbrack \Delta_c(\mu) : L_c(\eta) \rbrack 
\end{equation}
in $\msf{G}_0(\ol{\HHH}_c)$ for all $\eta \in \Irr W$. In particular the decomposition matrix of baby Verma modules for a fixed Calogero–Moser family is of rank 1 and
\begin{equation}
\frac{\dim \mu}{\dim \lambda} \ \lbrack \Delta_c(\lambda) : L_c(\lambda) \rbrack =  \lbrack \Delta_c(\mu) : L_c(\lambda) \rbrack  \;,
\end{equation}
so once we know the diagonal of the decomposition matrix, we know the whole decomposition matrix!
\end{corollary}

\subsection{Main problem} \label{rrca_problems}

The main problem about restricted rational Cherednik algebras is now the following:

\begin{problem}
Determine for any $c \in \mscr{C}$ and any $\lambda \in \Irr W$ the graded decomposition of $\Delta_c(\lambda)$ into the simple modules $L_c(\mu)$, $\mu \in \Irr W$.
\end{problem}

We have seen in Lemma \ref{simple_W_struct_baby_Verma_struct} that this is equivalent to the following problem:

\begin{problem}
Determine for any $c \in \mscr{C}$ and any $\lambda \in \Irr W$ the graded $W$-module character of $L_c(\lambda)$.
\end{problem}

It is a consequence of Corollary \ref{diagonal_enough} that this problem already gives the solution to the following:

\begin{problem}
Determine the Calogero–Moser $c$-families for any $c \in \mscr{C}$.
\end{problem}

It is clear that our main problem furthermore gives the solutions to the following sub-problems:

\begin{problem} 
Determine for any $c \in \mscr{C}$ and $\lambda \in \Irr W$:
\begin{enumerate}
\item The dimension of $L_c(\lambda)$.
\item The Poincaré series of $L_c(\lambda)$.
\item The $W$-module structure of $L_c(\lambda)$.
\item The decomposition of $\Delta_c(\lambda)$ into the simple modules $L_c(\mu)$, $\mu \in \Irr W$.
\end{enumerate}
\end{problem}

\section{Generic representation theory} \label{gen_rep_theory}

The main problems in Section \ref{rrca_problems} are posed for arbitrary parameters $c \in \mscr{C}$. In this chapter we want to show how one can reduce this to a finite problem by taking non-closed points of $\CCC^\natural$ into account. This framework allows us to introduce the two ``genericity loci'' $\DecGen(\ol{\HHH})$ and $\BlGen(\ol{\HHH})$, which should be considered as additional invariants of $W$ and play an important in understanding the representation theory of $\ol{\HHH}_c$ for arbitrary $c$. 

\subsection{Decomposition maps}

Let $\ccc$ be a prime ideal of $\CCC$. Let us define
\begin{equation}
\ol{\HHH}|_\ccc \dopgleich \ol{\HHH}/\ccc \ol{\HHH} \;.
\end{equation}
This is a $(\CCC/\ccc)$-algebra which is free and finitely generated as a $(\CCC/\ccc)$-module. We can identify the spectrum of $\CCC/\ccc$ with the zero locus $\msf{V}(\ccc)$ of $\ccc$ in $\CCC^\natural$ and the specialization of $\ol{\HHH}|_\ccc$ in a point $c \in \msf{V}(\ccc)$ is precisely $\ol{\HHH}_c$. As an example, we can take $\ccc = \bullet$ to be the generic point defined by the zero ideal in which case we have $\ol{\HHH}|_\bullet = \ol{\HHH}$.

From now on, we assume that $\CCC/\ccc$ is \textit{normal}.\footnote{This assumption is only needed to ensure uniquely characterized decomposition maps in the following. We do not know if this assumption can actually be removed.} The theory of decomposition maps by Geck and Rouquier \cite{Geck-Rouquier} (see also \cite{Geck-Pfeiffer} and \cite{Thiel-Dec}) shows that for any $c \in \msf{V}(\ccc)$ there is a morphism
\begin{equation}
\msf{d}_{\ol{\HHH}|_\ccc}^c : \msf{G}_0(\ol{\HHH}_\ccc) \rarr \msf{G}_0(\ol{\HHH}_c)
\end{equation} 
uniquely characterized by the equation
\begin{equation}
\msf{d}_{\ol{\HHH}|_\ccc}^c \lbrack V \rbrack = \lbrack \wt{V}/\fm \wt{V} \rbrack
\end{equation}
for any finite-dimensional $\ol{\HHH}_\ccc$-module $V$ and any $\mscr{O}$-free $\mscr{O}\ol{\HHH}|_\ccc$-form $\wt{V}$ of $V$ for any valuation ring $\mscr{O}$ in $\CCC/\ccc$ whose maximal ideal $\fm$ lies above $c$. Here, we use that $\ol{\HHH}_c$ splits by Proposition \ref{rrca_split} so that we can identify $\msf{G}_0(\ol{\HHH}_c) \simeq \msf{G}_0( \msf{k}_\mscr{O}(\fm) \ol{\HHH}_c)$. This morphism generalizes reduction of modules in $c$. In \cite[Theorem 1.22]{Thiel-Dec} we have shown that it is always possible to use a \textit{discrete} valuation ring $\mscr{O}$ for the construction of $\msf{d}_{\ol{\HHH}|_\ccc}^c$, where we use the fact that $\CCC/\ccc$ is noetherian.  

It is possible to refine the decomposition map to work with graded modules. This has been done in a general setting by Chlouveraki and Jacon \cite{CJ}. We thus  have a morphism
\begin{equation}
\msf{d}_{\ol{\HHH}|_\ccc}^{c,\mrm{gr}}: \msf{G}_0^{\mrm{gr}}(\ol{\HHH}_\ccc) \rarr \msf{G}_0^{\mrm{gr}}(\ol{\HHH}_c)
\end{equation}
which is similarly uniquely characterized as $\msf{d}_{\ol{\HHH}|_\ccc}^c$, but this time for graded modules. By construction, this map is compatible with shifts, i.e., 
\begin{equation}
\msf{d}_{\ol{\HHH}|_\ccc}^{c,\mrm{gr}}( \lbrack V[m] \rbrack) = \msf{d}_{\ol{\HHH}|_\ccc}^{c,\mrm{gr}}(\lbrack V \rbrack)[m] \;,
\end{equation}
and it fits into the commutative diagram
\begin{equation}
\begin{tikzcd}
 \msf{G}_0^{\mrm{gr}}(\ol{\HHH}_\ccc) \arrow{r}{\msf{d}_{\ol{\HHH}|_\ccc}^{c,\mrm{gr}}} \arrow{d} & \msf{G}_0^{\mrm{gr}}(\ol{\HHH}_c) \arrow{d} \\
 \msf{G}_0(\ol{\HHH}_\ccc) \arrow{r}[swap]{\msf{d}_{\ol{\HHH}|_\ccc}^{c}} & \msf{G}_0(\ol{\HHH}_c)
\end{tikzcd}
\end{equation}
where the vertical morphisms are obtained by forgetting about the grading. 

In the following we use the notation ``(gr)'' to signify that we can work with graded or non-graded modules. Recall that $\msf{G}_0(\ol{\HHH}_c)$ is a free $\bbZ$-module and that $\msf{G}_0^{\mrm{gr}}(\ol{\HHH}_c)$ is a free $\bbZ \lbrack q,q^{-1} \rbrack$-module, both having the classes of the simple modules $L_c(\lambda)$, $\lambda \in \Irr W$, as basis. We call the matrix of $\msf{d}_{\ol{\HHH}|_\ccc}^{c,\mrm{(gr)}} $ in such a standard basis the (graded) \word{decomposition matrix} of $\ol{\HHH}|_\ccc$ in $c$ and denote it by $\msf{D}_{\ol{\HHH}|_\ccc}^{c,\mrm{(gr)}}$. It is of course only unique up to rearranging the bases. An interesting property we immediately obtain from Theorem \ref{baby_verma_heads} is:

\begin{lemma}
For any $c \in \msf{V}(\ccc)$ the decomposition matrix $\msf{D}_{\ol{\HHH}|_\ccc}^{c,\mrm{(gr)}}$ is a square matrix.
\end{lemma}

For the definition of the baby Verma modules (\ref{baby_verma_def}) we actually do not need to work over a field. We can equally well define the \word{generic baby Verma module} for $\ol{\HHH}|_\ccc$ as
\begin{equation}
\ol{\HHH}|_\ccc \otimes_{(\CCC/\ccc) \lbrack \fh^* \rbrack^{\mrm{co}(W)} \rtimes W} M
\end{equation}
attached to a $KW$-module $M$. This is easily seen to be a $(\CCC/\ccc)$-free $\ol{\HHH}|_\ccc$-form of $\Delta_\ccc(\lambda)$ and we get
\begin{equation} \label{dec_map_verma_compat}
\msf{d}_{\ol{\HHH}|_\ccc}^{c,(\mrm{gr})}( \lbrack \Delta_\ccc(\lambda) \rbrack) = \lbrack \Delta_c(\lambda) \rbrack \;.
\end{equation}

An application of Lemma \ref{radical_dec_matrix} now shows:

\begin{lemma} \label{L_is_quot_of_dec}
Let $c \in \msf{V}(\ccc)$. Then for any $\lambda \in \Irr W$ the simple module $L_c(\lambda)$ is a constituent of $\msf{d}_{\ol{\HHH}|_\ccc}^{c,(\mrm{gr})}(\lbrack L_\ccc(\lambda) \rbrack)$. In particular, $\dim_{\msf{k}_\CCC(c)}L_c(\lambda) \leq \dim_{\msf{k}_\CCC(\ccc)}L_\ccc(\lambda)$.
\end{lemma}

The group algebra $\msf{k}_\CCC(c)W$ is a (graded) subalgebra of $\ol{\HHH}_c$. In the same way as above we also have a decomposition map
\begin{equation}
\msf{d}_{(\CCC/\ccc) W}^{c,\mrm{(gr)}}: \msf{G}_0^{\mrm{(gr)}}(\msf{k}_\CCC(\ccc)W) \rarr \msf{G}_0^{\mrm{(gr)}}(\msf{k}_\CCC(c)W)
\end{equation}
for both graded and nongraded modules over the group algebra of $W$ over $\CCC/\ccc$. It is clear that any $\mscr{O}$-free $\mscr{O}\ol{\HHH}|_\ccc$-form of a finite-dimensional $\ol{\HHH}_\ccc$-module is at the same time also an $\mscr{O}$-free $\mscr{O}W$-form, so the diagram
\begin{equation} \label{dec_w_compatibility_1}
\begin{tikzcd}
 \msf{G}_0^{\mrm{gr}}(\ol{\HHH}_\ccc) \arrow{r}{\msf{d}_{\ol{\HHH}|_\ccc}^{c,\mrm{(gr)}}} \arrow{d} & \msf{G}_0^{\mrm{gr}}(\ol{\HHH}_c) \arrow{d} \\
  \msf{G}_0^{\mrm{(gr)}}(\msf{k}_\CCC(\ccc)W) \arrow{r}[swap]{\msf{d}_{(\CCC/\ccc) W}^{c,\mrm{(gr)}}} & \msf{G}_0^{\mrm{(gr)}}(\msf{k}_\CCC(c)W)
\end{tikzcd}
\end{equation}
commutes. Now, remember that $KW$ splits. This implies that any finite-dimensional $\msf{k}_\CCC(\ccc)W$-module actually has a $KW$-form. The extension of a $KW$-form to $\mscr{O}$ is clearly an $\mscr{O}$-free $\mscr{O} W$-form and reduction in the maximal ideal $\fm$ of $\mscr{O}$ does not change anything. In other words, together with the isomorphisms $\msf{G}_0^{\mrm{(gr)}}(KW) \simeq \msf{G}_0^{\mrm{(gr)}}(\msf{k}_\CCC(\ccc)W)$ and $\msf{G}_0^{\mrm{(gr)}}(KW) \simeq \msf{G}_0^{\mrm{(gr)}}(\msf{k}_\CCC(c) W)$ given by scalar extension, diagram (\ref{dec_w_compatibility_1}) collapses to the commutative diagram
\begin{equation} \label{dec_w_compatibility}
\begin{tikzcd}
 \msf{G}_0^{\mrm{(gr)}}(\ol{\HHH}_\ccc) \arrow{dr} \arrow{rr}{\msf{d}_{\ol{\HHH}|_\ccc}^{c,\mrm{(gr)}}}  && \msf{G}_0^{\mrm{gr}}(\ol{\HHH}_c) \arrow{dl}  \\
 & \msf{G}_0^{\mrm{(gr)}}(KW)
\end{tikzcd}
\end{equation}

This simply means:

\begin{lemma} \label{dec_preserves_W}
The decomposition map $\msf{d}_{\ol{\HHH}|_\ccc}^{c,\mrm{(gr)}}$ preserves the (graded) $W$-module structure.
\end{lemma}

Recall from Proposition \ref{graded_w_restriction_injective} that the restriction maps to the \textit{graded} Grothendieck group of $KW$ in diagram (\ref{dec_w_compatibility}) are injective.

\begin{lemma} \label{dec_trivial_lemma}
For any $c \in \msf{V}(\ccc)$ the following are equivalent:
\begin{enumerate}
\item \label{dec_trivial_lemma:1} $\msf{d}_{\ol{\HHH}|_\ccc}^{c}( \lbrack L_\ccc(\lambda) \rbrack) = \lbrack L_c(\lambda) \rbrack$ for all $\lambda \in \Irr W$.
\item \label{dec_trivial_lemma:2} $\msf{D}_{\ol{\HHH}|_\ccc}^{c}$ is the identity matrix (up to row and column permutation).
\item \label{dec_trivial_lemma:3} $\msf{d}_{\ol{\HHH}|_\ccc}^{c,\mrm{(gr)}}( \lbrack L_\ccc(\lambda) \rbrack) = \lbrack L_c(\lambda) \rbrack$ for all $\lambda \in \Irr W$.
\item \label{dec_trivial_lemma:4} $\msf{D}_{\ol{\HHH}|_\ccc}^{c,\mrm{(gr)}}$ is the identity matrix (up to row and column permutation). 
\end{enumerate}
\end{lemma}

\begin{proof}
The equivalences (\ref{dec_trivial_lemma:1}) $\Leftrightarrow$ (\ref{dec_trivial_lemma:2}) and (\ref{dec_trivial_lemma:3}) $\Leftrightarrow$  (\ref{dec_trivial_lemma:4}) are obvious. By Corollary \ref{L_is_quot_of_dec}, properties (\ref{dec_trivial_lemma:1}) and (\ref{dec_trivial_lemma:3}) are equivalent to $\dim_{\msf{k}_\CCC(c)} L_c(\lambda) = \dim_{\msf{k}_\CCC(\ccc)} L_\ccc(\lambda)$ for all $\lambda$. These two assertions are thus clearly equivalent.
\end{proof}

\begin{definition} \label{dec_trivial_def}
We say that $\msf{d}_{\ol{\HHH}|_\ccc}^{c}$ is \word{trivial} if it satisfies the conditions in Lemma \ref{dec_trivial_lemma}. We define
\begin{equation}
\DecGen(\ol{\HHH}|_\ccc) \dopgleich \lbrace c \in \msf{V}(\ccc) \mid \msf{d}_{\ol{\HHH}|_\ccc}^c \tn{ is trivial} \rbrace 
\end{equation}
and
\begin{equation}
\msf{DecEx}(\ol{\HHH}|_\ccc) \dopgleich \msf{V}(\ccc) \setminus \DecGen(\ol{\HHH}|_\ccc) \;.
\end{equation}
\end{definition}

From Lemma \ref{dec_preserves_W} we immediately obtain:

\begin{lemma}
If $c \in \DecGen(\ol{\HHH}|_\ccc)$, then for all $\lambda \in \Irr(W)$ we have:
\begin{enumerate}
\item The graded $W$-module structures of $L_\ccc(\lambda)$ and $L_c(\lambda)$ are identical for all $\lambda \in \Irr W$. In particular, their Poincaré series and dimensions coincide.
\item The graded decomposition numbers of $\Delta_\ccc(\lambda)$ and $\Delta_c(\lambda)$ into simple modules are identical. 
\end{enumerate}
\end{lemma}

\begin{theorem} \label{decgen_open}
The set $\DecGen(\ol{\HHH}|_\ccc)$ is a non-empty open subset of  $\msf{V}(\ccc)$. It is $K^\times$-stable if $\msf{V}(\ccc)$ is $K^\times$-stable, e.g., if $\ccc = \bullet$.  Moreover, 
\begin{equation}
c \in \DecGen(\ol{\HHH}|_\ccc) \tn{ if and only if } \dim_{\msf{k}_\CCC(\ccc)} \Rad \ol{\HHH}_\ccc = \dim_{\msf{k}_\CCC(c)} \Rad \ol{\HHH}_c \;.
\end{equation}
\end{theorem}

\begin{proof}
We can apply \cite[Theorem 2.3]{Thiel-Dec} to $\ol{\HHH}|_\ccc$ and this shows that $\DecGen(\ol{\HHH}|_\ccc)$ is a non-empty open subset of $\msf{V}(c) = \Spec(\CCC/\ccc)$. The stability under $K^\times$ follows from Lemma \ref{verma_simple_scaling}. The last assertion is \cite[Theorem 2.2]{Thiel-Dec}. 
\end{proof}

%
%
%
%
 The following is a refinement of Theorem \ref{decgen_open} and follows from an application of Proposition \ref{decgen_for_simple_general}.

\begin{proposition} \label{decgen_for_simple}
For $\lambda \in \Irr W$ we have 
\begin{equation}
\lbrace c \in \msf{V}(\ccc) \mid \msf{d}_{\ol{\HHH}|_\ccc}^c(\lbrack L_\ccc(\lambda) \rbrack) \tn{ is simple} \rbrace = \lbrace c \in \msf{V}(\ccc) \mid \msf{d}_{\ol{\HHH}|_\ccc}^c(\lbrack L_\ccc(\lambda) \rbrack) = \lbrack L_c(\lambda) \rbrack \rbrace
\end{equation}
and this set is a neighborhood of the generic point in $\msf{V}(\ccc)$. We denote it by $\DecGen(\ol{\HHH}|_\ccc,L_\ccc(\lambda))$. If $\msf{V}(\ccc)$ is $K^\times$-stable, so is $\DecGen(\ol{\HHH}|_\ccc,L_\ccc(\lambda))$.
\end{proposition}

It is an immediate consequence of Lemma \ref{dec_trivial_lemma} that
\begin{equation} \label{decgen_intersect_of_decgen_simple}
\DecGen(\ol{\HHH}|_\ccc) = \bigcap_{\lambda \in \Irr W} \DecGen(\ol{\HHH}|_\ccc,L_\ccc(\lambda)) \;.
\end{equation}

Using the preceding genericity results and the splitting of $\ol{\HHH}_c$ from Proposition \ref{rrca_split}, it is now straightforward to see that the fundamental Theorem \ref{rca_reps_geometry} for simple $\HHH_c$-modules supported in the origin of $\BBB$ actually holds for arbitrary $c \in \CCC^\natural$. We continue this generalization in Theorem \ref{singleton_cm_fams}.

\begin{lemma}[Etingof--Ginzburg] \label{rrca_dim_upper_bound}
The following holds for any $\lambda \in \Irr W$ and any $c \in \CCC^\natural$:

\begin{enumerate}
\item \label{rrca_dim_upper_bound_1} $\dim_{\msf{k}_\CCC(c)} L_c(\lambda) \leq |W|$.

\item If $\dim_{\msf{k}_\CCC(c)} L_c(\lambda) = |W|$, then $L_c(\lambda) \simeq \msf{k}_\CCC(c)W$ as $W$-modules.
\end{enumerate}
\end{lemma}

\begin{definition}
Let $c \in \CCC^\natural$ and let $\lambda \in \Irr W$. We say that $L_c(\lambda)$ is \word{smooth} if it satisfies the conditions in Lemma \ref{rrca_dim_upper_bound}, otherwise we call it \word{singular}. We say that $\ol{\HHH}_c$ is \word{smooth}, resp. \word{singular}, if all its simple modules are smooth, resp. singular. 
\end{definition}

\subsection{Semisimplicity}

An easy consequence of the splitting of $\ol{\HHH}_c$ is the following, see \cite{Bellamy-Thiel} for details.

\begin{lemma}[Bellamy-T.] \label{rrca_semisimple}
The $\msf{k}_\CCC(c)$-algebra $\ol{\HHH}_c$ is semisimple if and only if $\Delta_c(\lambda)$ is already irreducible for all $\lambda \in \Irr W$. 
\end{lemma} 

From this we obtain:

\begin{lemma} 
If $\ol{\HHH}_c$ is semisimple, then $W$ must be abelian.
\end{lemma}

\begin{proof}
We know from Lemma \ref{rrca_semisimple} that $\Delta_c(\lambda)$ is already irreducible for all $\lambda \in \Irr W$. Hence, $\dim_{\msf{k}_\CCC(c)} \Delta_c(\lambda) = \dim_{\msf{k}_\CCC(c)} L_c(\lambda)$. We know from Lemma \ref{rrca_dim_upper_bound} $\dim_{\msf{k}_\CCC(c)} L_c(\lambda) \leq |W|$, hence $\dim_{\msf{k}_\CCC(c)} \Delta_c(\lambda) = |W| \cdot \dim_{K} \lambda \leq |W|$, and this implies that $\dim_K \lambda \leq 1$. Hence, if $\ol{\HHH}_c$ is semisimple, then all simple $KW$-modules are one-dimensional. This implies that $W$ is abelian.
\end{proof}

In combination with the results about cyclic groups from Section \ref{Cyclic_section} we now obtain:

\begin{corollary}
The algebra $\ol{\HHH}_c$ is semisimple if and only if $W$ is abelian and $c \in \BlGen(\ol{\HHH})$.
\end{corollary}

\subsection{Semi-continuity of Calogero--Moser families} \label{cm_semicont}

The following semicontinuity property is proven in \cite[Theorem C]{Thiel-Blocks} in a general context.

\begin{theorem} \label{semicontinuity}
Let $\ccc \in \CCC^\natural$. The following holds:
\begin{enumerate}
\item We have
\begin{equation}
\# \CM_c \leq \# \CM_\ccc
\end{equation}
for all $c \in \msf{V}(\ccc)$. 
\item Let $c \in \msf{V}(\ccc)$. Then $\# \CM_c = \#\CM_\ccc$ if and only if $\CM_c = \CM_\ccc$. If $\# \CM_c \leq \#\CM_\ccc$, then $\CM_c$ is obtained by gluing some families of $\CM_\ccc$.
\item The function $\msf{V}(\ccc) \rarr \bbN$, $c \mapsto \# \CM_c$, is lower semicontinuous, so for each $n \in \bbN$ the set
\begin{equation}
\lbrace c \in \msf{V}(\ccc) \mid \#\CM_c \leq n \rbrace
\end{equation}
is closed in $\msf{V}(\ccc)$.
\end{enumerate}
\end{theorem}

For $\ccc \in \CCC^\natural$ we define 
\begin{equation}
\BlGen(\ol{\HHH}|_\ccc) \dopgleich \lbrace c \in \msf{V}(\ccc) \mid \# \CM_c = \# \CM_\ccc \rbrace =  \lbrace c \in \msf{V}(\ccc) \mid \CM_c = \CM_\ccc \rbrace \;.
\end{equation}
By Theorem \ref{semicontinuity} this is a closed subset of $\msf{V}(\ccc)$. It plays a similar role for blocks as the set $\DecGen(\ol{\HHH}|_\ccc)$ plays for the simple modules. Let us denote the complement of $\BlGen(\ol{\HHH}|_\ccc)$ in $\msf{V}(\ccc)$ by $\BlEx(\ol{\HHH}|_\ccc)$. The next theorem is a consequence of \cite[Theorem E]{Thiel-Blocks}.

\begin{theorem}
The set $\BlGen(\ol{\HHH}|_\ccc)$ is a reduced Weil divisor in $\msf{V}(\ccc)$, i.e., it is either empty or pure of codimension one with finitely many irreducible components.
\end{theorem}

We can in fact give a rather explicit description of $\BlGen(\ol{\HHH}|_\ccc)$. Let $c \in \CCC^\natural$ be arbitrary. Since $\ol{\HHH}_c$ splits, every $z \in \ZZ(\ol{\HHH}_c)$ acts on a simple $\ol{\HHH}_c$-module $L_c(\lambda)$ by a scalar $\boldsymbol{\Omega}_\lambda^c(z)$. The resulting map
\begin{equation}
\begin{array}{rcl}
\boldsymbol{\Omega}_\lambda^c: \ZZ(\ol{\HHH}_c) & \longrightarrow & \msf{k}_\CCC(c) \\
z & \longmapsto & \boldsymbol{\Omega}_\lambda^c(z)
\end{array}
\end{equation}
is a morphism of $\msf{k}_\CCC(c)$-algebras, the so-called \word{central character} of $L_c(\lambda)$. It is a standard fact that the central characters determine the families, so $\lambda,\mu \in \Irr W$ lie in the same Calogero--Moser $c$-family if and only if $\boldsymbol{\Omega}_\lambda^c = \boldsymbol{\Omega}_\mu^c$.

Let us now concentrate on the generic point $\bullet$ of $\CCC^\natural$. Note that  $\ol{\HHH} \subs \ol{\HHH}_\bullet$ and $\ZZ(\ol{\HHH}) \subs \ZZ(\ol{\HHH}_\bullet) = \ZZ(\ol{\HHH})_\bullet$. Since the base ring $\CCC$ of $\ol{\HHH}$ is normal, it is a standard fact that the image of the restriction of $\boldsymbol{\Omega}_\lambda^\bullet$ to $\ZZ(\ol{\HHH})$ is contained in $\CCC \subs \msf{k}_\CCC(\bullet)$, so by restriction we get a $\CCC$-algebra morphism
\begin{equation}
\boldsymbol{\Omega}_\lambda': \ZZ(\ol{\HHH}) \rarr \CCC \;.
\end{equation}

The following lemma is straightforward.

\begin{lemma}
Suppose that $Z$ is a $\CCC$-subalgebra of $\ZZ(\ol{\HHH})$ such that $Z_\bullet = \msf{k}_\CCC(\bullet) \otimes_\CCC Z \subs \ol{\HHH}_\bullet$ contains all block idempotents of $\ol{\HHH}_\bullet$. Then $\lambda,\mu \in \Irr W$ lie in the same Calogero--Moser $\bullet$-family if and only if $\boldsymbol{\Omega}_\lambda'|_{Z} = \boldsymbol{\Omega}_\mu'|_Z$.
\end{lemma}

An obvious example is $Z=\ZZ(\ol{\HHH})$. But from Theorem \ref{mueller_blocks} we get a better example, namely $\ol{\ZZZ}$ since we have $\Bl(\ol{\HHH}_\bullet) \simeq \Bl(\ol{\ZZZ}_\bullet)$. Every $z \in \HHH$ also acts as a scalar $\wt{\boldsymbol{\Omega}}_\lambda(z)$ on $L_\bullet(\lambda)$. Of course, $\wt{\boldsymbol{\Omega}}_\lambda(z) = \boldsymbol{\Omega}'_\lambda(z \ \msf{mod} \ 0 )$, where $0$ corresponds to the origin in $\BBB \subs \ZZZ$. We thus have a $\CCC$-algebra morphism
\begin{equation}
\wt{\boldsymbol{\Omega}}_\lambda: \ZZZ \rarr \CCC
\end{equation}
and in total we obtain:

\begin{corollary} \label{omega_tilde_chars}
Two simple modules $\lambda,\mu \in \Irr W$ lie in the same Calogero--Moser $\bullet$-family if and only if $\wt{\boldsymbol{\Omega}}_\lambda = \wt{\boldsymbol{\Omega}}_\mu$. For this, it is sufficient to check that $\wt{\boldsymbol{\Omega}}_\lambda(z_i) = \wt{\boldsymbol{\Omega}}_\mu(z_i)$ for all $z_i$ in a $\CCC$-algebra generating system $\lbrace z_1,\ldots,z_r \rbrace$ of $\ZZZ$. 
\end{corollary}

This description of the Calogero--Moser $\bullet$-families behaves well under specialization. The following theorem is proven in \cite[Theorem G]{Thiel-Blocks} in a general context.\footnote{We use here the just established fact the $\bullet$-families are already distinguished by the $\wt{\boldsymbol{\Omega}}_\lambda$.}

\begin{theorem} \label{blgen_explicit}
Let $\lbrace z_1,\ldots,z_r \rbrace$ be a $\CCC$-algebra generating system of $\ZZZ$. Let $c \in \CCC^\natural$. Then $\lambda,\mu \in \Irr W$ lie in the same Calogero--Moser $c$-family if and only if 
\begin{equation}
\wt{\boldsymbol{\Omega}}_\lambda(z_i) \equiv \wt{\boldsymbol{\Omega}}_\mu(z_i) \ \msf{mod} \ c
\end{equation}
for all $i=1,\ldots,r$. Hence, for any $\ccc \in \CCC^\natural$ we have
\begin{equation}
\BlEx(\ol{\HHH}|_\ccc) = \bigcup_{\substack{ \lambda,\mu \in \Irr W \\ \lambda \tn{ and } \mu \tn{ lie in } \\ \tn{distinct } \CM_\ccc\tn{-families}}} \bigcap_{i=1}^r \msf{V}(\wt{\boldsymbol{\Omega}}_\lambda(z_i) - \wt{\boldsymbol{\Omega}}_\mu(z_i)) \cap \msf{V}(\ccc)\;.
\end{equation}
\end{theorem}

It is possible to give explicit formulas for $\wt{\boldsymbol{\Omega}}_\lambda(z)$ for $z \in \ZZZ$ written in a PBW basis, see \cite{BR} and \cite{Bonnafe-Thiel}. Theorem \ref{blgen_explicit} is one ingredient in the approach of Bonnafé and the author \cite{Bonnafe-Thiel} to explicitly determine the Calogero--Moser families for many exceptional complex reflection groups using a computer algebra system. We refer to \cite{Bonnafe-Thiel} for further details and the results. We note that one can choose $(\bbN \times \bbN)$-homogeneous generators of $\ZZZ$ and then the elements $\wt{\boldsymbol{\Omega}}_\lambda(z_i) \in \CCC$ are $(\bbN \times \bbN)$-homogeneous, too. From this we immediately deduce:

\begin{lemma} \label{blgen_Kstar_stable}
If $\msf{V}(\ccc)$ is $K^\times$-stable, so are the sets $\BlGen(\ol{\HHH}|_\ccc)$ and $\BlEx(\ol{\HHH}|_\ccc)$.
\end{lemma}

It is conjectured that $\BlEx(\ol{\HHH})$ is in fact a union of hyperplanes, see Conjecture \ref{hyperplane_conjecture}. This is true in all known cases.

\subsection{Two genericity loci}

The following is proven in \cite[Theorem J]{Thiel-Blocks} in a general context.

\begin{theorem} \label{decgen_in_blgen}
For any $\ccc \in \CCC^\natural$ we have $\DecGen(\ol{\HHH}|_\ccc) \subs \BlGen(\ol{\HHH}|_\ccc)$.
\end{theorem}

In Conjecture \ref{decgen_blgen_conjecture} below we conjecture that we have in fact equality in Theorem \ref{decgen_in_blgen}, at least for $\ccc$ being the generic point. There is one case where we can show this, namely when the Calogero–Moser space is smooth for some $c$. To this end, we first continue to generalize Theorem \ref{rca_reps_geometry} to non-closed points.

\begin{lemma}\label{singleton_cm_fams}
For any $\lambda \in \Irr W$ and any $c \in \CCC^\natural$ the following are equivalent:
\begin{enumerate}
\item \label{singleton_cm_fams_1} $\dim_{\msf{k}_\CCC(c)} L_c(\lambda) = |W|$.
\item \label{singleton_cm_fams_2} $\lambda$ lies in a singleton Calogero--Moser $c$-family.
\newcounter{enumTemp}
    \setcounter{enumTemp}{\theenumi}
\end{enumerate}
\end{lemma}

\begin{proof}
Theorem \ref{rca_reps_geometry} shows the equivalence for closed points $c$ assuming that $K$ is algebraically closed. Since $\ol{\HHH}_c$ splits by Proposition \ref{rrca_split}, it also holds without the assumption on $K$ being algebraically closed. In \cite[Proposition 1]{Thiel-Counter} we have given a straightforward character-theoretic proof for (\ref{singleton_cm_fams_2}) $\Rightarrow$ (\ref{singleton_cm_fams_1}) which works for any $c$. We note that this relies on Theorem \ref{rrca_dim_upper_bound} for arbitrary $c$. Now, suppose that $\dim_{\msf{k}_\CCC(c)} L_c(\lambda) = |W|$. Even if $\CCC/c$ is not normal,  there is still for any $c' \in \msf{V}(c)$ a (possibly not unique) decomposition map $\msf{d}_{\ol{\HHH}|_c}^{c'}: \msf{G}_0(\ol{\HHH}_c) \rarr \msf{G}_0(\ol{\HHH}_{c'})$, see \cite{Thiel-Dec}. It follows from \cite[Theorem 2.3]{Thiel-Dec} that there is a non-empty open subset $U$ of $\msf{V}(c)$ such that all decomposition maps $\msf{d}_{\ol{\HHH}|_c}^{c'}$ for $c' \in U$ are trivial. Let $c'$ be a closed point in $U$. Then $\msf{d}_{\ol{\HHH}|_c}^{c'}(\lbrack L_c(\lambda) \rbrack) = \lbrack L_{c'}(\lambda) \rbrack$, so $\dim_{\msf{k}_\CCC(c')}L_{c'}(\lambda)=|W|$.  Now, it follows from the implication (\ref{singleton_cm_fams_1}) $\Rightarrow$ (\ref{singleton_cm_fams_2}) for closed points that $L_{c'}(\lambda)$ lies in a singleton Calogero--Moser family. Since $c' \in \msf{V}(c)$, the semicontinuity Theorem \ref{semicontinuity} implies that also $L_c(\lambda)$ must lie in a singleton Calogero--Moser family.
\end{proof}

\begin{corollary} \label{decgen_blgen_smooth}
If $\XXX_c$ is smooth for some $c \in \CCC^\natural$, then $\DecGen(\ol{\HHH}) = \BlGen(\ol{\HHH})$.
\end{corollary}

\begin{proof}
If one fiber in the Calogero–Moser fibration $\XXX \rarr \CCC^\natural$ is smooth, it is also smooth over some closed point, so we can assume that $c$ is closed. It then follows from Theorem \ref{rca_reps_geometry} that the Calogero–Moser $c$-families are singletons and that $\dim_{\msf{k}_\CCC(c)} L_c(\lambda)=|W|$ for all $\lambda \in \Irr W$. By Lemma \ref{dec_trivial_lemma} and Lemma \ref{rrca_dim_upper_bound} the set $\DecGen(\ol{\HHH})$ must consist of all $c' \in \CCC^\natural$ such that $\dim_{\msf{k}_\CCC(c')} L_{c'}(\lambda) = |W|$. Similarly, by the semicontinuity Theorem \ref{semicontinuity} we know that $\BlGen(\ol{\HHH})$ consists of all $c'$ such that the Calogero--Moser $c'$-families are singletons. Now, both sets are the same by Lemma \ref{singleton_cm_fams}.
\end{proof}

\section{Toolbox} \label{toolbox}
 
In this Chapter we collect several additional results about restricted rational Cherednik algebras which turned out to be quite useful in applications. 

\subsection{Euler families}

Let $c \in \CCC^\natural$. The \word{Euler element} is the element of $\HHH_c$ defined as
\begin{equation}
\eu \dopgleich \sum_{i=1}^n y_ix_i + \sum_{s \in \Ref(W)} \frac{1}{\eps_s-1}c(s)s \;,
\end{equation}
where $(y_i)_{i=1}^n$ is a basis of $\fh$ with dual basis $(x_i)_{i=1}^n$, and $\eps_s$ denotes the non-trivial eigenvalue of $s$. Here, $c(s)$ denotes the image of $\ccc(s)$ in $c$. It is not hard to see that $\eu$ does not depend on the choice of the basis of $\fh$ and that 
\begin{equation}
\eu = \sum_{i=1}^n x_iy_i + \sum_{s \in \Ref(W)} \frac{\eps_s}{\eps_s-1}c(s)s \;.
\end{equation}
Furthermore, it is straightforward to verify that $\eu$ commutes with all $x_i$, $y_i$, and $w \in W$, so:

\begin{lemma}
The Euler element is central in $\HHH_c$.
\end{lemma}

The Euler element has already played an important role at $t \neq 0$ in the work of Ginzburg--Guay--Opdam--Rouquier \cite{GGOR}. At $t=0$, it also plays a very important role, see for example the work of Bonnafé and Rouquier \cite{BR} or the paper \cite{Thiel-Counter} by the author. Its image in $\ol{\HHH}_c$ is again a non-trivial central element we will again denote by $\eu$. 

\begin{definition}
We say that $\lambda,\mu \in \Irr W$ are in the same \word{Euler $c$-family} if $\boldsymbol{\Omega}_\lambda^c(\eu) = \boldsymbol{\Omega}^c_\mu(\eu)$. We denote the set of Euler $c$-families by $\Eu_c$.
\end{definition}

The Calogero–Moser $c$-families refine the Euler $c$-families since the former are determined by values of the central characters $\boldsymbol{\Omega}_\lambda^c$ on \textit{all} central elements. It follows directly from the definition of the Euler element and the fact that $\fh^*$ acts trivially on $\Delta_c(\lambda)$ that $\eu$ acts by the following scalar on $\Delta_c(\lambda)$, and thus on $L_c(\lambda)$.

\begin{lemma}
Let $c \in \CCC^\natural$. For any $\lambda \in \Irr W$ we have
\begin{equation}
\boldsymbol{\Omega}_\lambda^c(\eu) = \frac{1}{\chi_\lambda(1)} \sum_{s \in \Ref(W)} \frac{\eps_s}{\eps_s-1} c(s) \chi_\lambda(s) \;,
\end{equation}
where $\chi_\lambda$ is the character of $\lambda$ and $c(s)$ is the image of $\ccc(s)$ in $\msf{k}_\CCC(c)$.
\end{lemma}

\begin{corollary}
Let $c \in \CCC^\natural$. Then $\lambda,\mu \in \Irr W$ lie in the same Euler $c$-family if and only if
\begin{equation} \label{same_euler_family_equ}
\sum_{s \in \Ref(W)} \frac{\eps_s}{\eps_s-1}c(s) \left( \frac{\chi_\lambda(s)}{\chi_\lambda(1)} - \frac{\chi_\mu(s)}{\chi_\mu(1)} \right) = 0 \;.
\end{equation} 
\end{corollary}

It is a natural question to ask how close the Euler families are to the Calogero--Moser families. They would clearly coincide if $\ZZZ = \PPP \lbrack \eu \rbrack$. However, this is rarely the case as the following result by Bonnafé and Rouquier \cite[Proposition 5.5.9]{BR} shows:

\begin{proposition}[Bonnafé--Rouquier] \label{cyclic_center_by_euler}
We have $\ZZZ = \PPP \lbrack \eu \rbrack$ if and only if $W$ is of rank $1$.
\end{proposition}

Nonetheless, the Euler families turn out to be a quite useful tool. If we take indeterminates $\ccc(s)$ in (\ref{same_euler_family_equ}) we obtain:

\begin{corollary} \label{generic_euler_formula}
Two simple modules $\lambda,\mu \in \Irr W$ lie in the same Euler $\bullet$-family if and only if
\begin{equation}
\chi_\mu(1)\chi_\lambda(s) = \chi_\lambda(1)\chi_\mu(s)
\end{equation}
for all $s \in \Ref(W)$.
\end{corollary}

One can now prove the following result, see \cite[Proposition 4]{Thiel-Counter}.

\begin{proposition} \label{linear_generic_singleton}
The linear characters of $W$ lie in singleton Euler $\bullet$-families, thus in singleton Calogero--Moser $\bullet$-families.
\end{proposition}

\subsection{Rigid modules}

Recall from Lemma \ref{verma_degree_zero} that $L_c(\lambda)_0 \simeq \lambda$ as $W$-modules. 

\begin{definition}
We say that $L_c(\lambda)$ is \word{rigid} if it is concentrated in degree zero, i.e., $L_c(\lambda) \simeq \lambda$ as $W$-modules. We say that $\lambda \in \Irr W$ is \word{$c$-rigid} if $L_c(\lambda)$ is rigid.
\end{definition}

Rigid modules have been introduced and studied in \cite{Thiel-Diss}. Recently, they played an important role in the classification of the cuspidal Calogero--Moser families, see \cite{Bellamy-Thiel}. Rigid modules are easily detected. Namely, both $\fh$ and $\fh^*$ act trivially on $L_c(\lambda)$ if it is rigid. The question is thus when an irreducible representation $\lambda:W \rarr \GL_r(\msf{k}_{\CCC}(c)W)$ defines a representation of $\ol{\HHH}_c$ with $\fh$ and $\fh^*$ acting trivially and $w$ acting via $\lambda$. This is the case if and only if it respects the commutator relation (\ref{gen_RCA_def_yx}). We thus obtain:

\begin{lemma} \label{rigid_lemma}
An irreducible representation $\lambda:W \rarr \GL_r(\msf{k}_{\CCC}(c)W)$ is $c$-rigid if and only if
\begin{equation} \label{rigid_equation}
\sum_{s \in \Ref(W)} c(s)(y,x)_s \lambda(s) = 0
\end{equation}
for all $y \in \fh$ and $x \in \fh^*$, where $(y,x)_s$ is as defined in (\ref{cherednik_coefficient}).

\end{lemma}

Of course, it is sufficient to check (\ref{rigid_equation}) for a basis $(y_i)_{i=1}^n$ of $\fh$ with dual basis $(x_i)_{i=1}^n$. Note that (\ref{rigid_equation}) is a matrix equation in $\GL_r(\msf{k}_{\CCC}(c)W)$. For each $\lambda$ this equation defines a closed $K^\times$-stable subscheme of $\CCC^\natural$. Once we have explicit realizations of the irreducible $W$-representations, we can \textit{in principle} explicitly determine the $c$-rigid ones. Clearly, a necessary condition for $\lambda$ being $c$-rigid is the equation we get from taking traces in (\ref{rigid_equation}), i.e.,
\begin{equation} \label{weakly_rigid_equation}
\sum_{s \in \Ref(W)} c(s)(y_i,x_j)_s \chi_\lambda(s) = 0
\end{equation}
for all $i,j \in \lbrace 1,\ldots,n \rbrace$, where $\chi_\lambda$ is the character of $\lambda$.

\begin{definition}
We say that $\lambda \in \Irr W$ is \word{weakly $c$-rigid} if it satisfies (\ref{weakly_rigid_equation}).
\end{definition}

%
%
%

Bellamy and the author \cite{Bellamy-Thiel} classified the $c$-rigid irreducible representations for all non-exceptional Coxeter groups and all parameters. 

\subsection{Poincaré series of smooth simple modules}

Bellamy \cite[Lemma 3.3]{Bellamy-Singular} has given a formula for the Poincaré series of smooth simple modules for the restricted rational Cherednik algebra.\footnote{The proof works word for word for arbitrary $c \in \CCC^\natural$.}\footnote{In the \cite[Lemma 3.3]{Bellamy-Singular} there is a typo: $\bbC \lbrack \fh^* \rbrack^{\mrm{co}(W)}$ should be $\bbC \lbrack \fh \rbrack^{\mrm{co}(W)}$.}

\begin{lemma}[Bellamy] \label{bellamy_lemma}
Suppose that $L_c(\lambda)$ is smooth, i.e., $\dim_{\msf{k}_{\CCC}(c)} L_c(\lambda) = |W|$. Then
\begin{equation} \label{bellamy_formula} 
P_{L_c(\lambda)}(q) = \frac{\dim_K(\lambda) \cdot q^{b_{\lambda^*}} \cdot P_{K\lbrack \fh \rbrack^{\mrm{co}(W)}}(q)}{f_{\lambda^*}(q)} \;.
\end{equation}
Here, $f_{\lambda^*}(q)$ is the fake degree of $\lambda^*$ as defined in (\ref{fake_degree_formula}) and $b_{\lambda^*}$ is the trailing degree of $f_{\lambda^*}$.
\end{lemma}

This gives us an effective method to find singular simple modules.

\begin{definition}
We say that $\lambda \in \Irr W$ is \word{supersingular} if $f_{\lambda^*}(q)$ does not divide the numerator in (\ref{bellamy_formula}).
\end{definition}

\begin{corollary} \label{supersingular}
If $\lambda$ is supersingular, then $L_c(\lambda)$ is singular for all $c \in \CCC^\natural$, i.e., $\dim_{\msf{k}_\CCC(c)} L_c(\lambda) < |W|$. In particular $\XXX_c$ is singular for all $c \in \CCC^\natural$.
\end{corollary}

In \cite{Thiel-Counter} we have introduced the notion of good Euler families.

\begin{definition} \label{good_euler_defn}
An Euler $c$-family $\mscr{F}$ is \word{good} if it is of one of the following types:
\begin{enumerate}
\item $|\mscr{F}| = 1$.
\item $|\mscr{F}| = 2$ and at least one character in $\mscr{F}$ is supersingular.
\item $|\mscr{F}| = 3$ and all characters in $\mscr{F}$ are supersingular.
\end{enumerate}
\end{definition}

As an easy consequence of Theorem \ref{singleton_cm_fams} we obtain:

\begin{lemma} \label{good_euler_cm}
Every good Euler $c$-family is already a Calogero--Moser $c$-family.
\end{lemma}

 \subsection{The case $c=0$} \label{c0_case}

It immediately follows from Lemma \ref{rigid_lemma} that all irreducible $W$-representations are $0$-rigid. We thus know the graded $W$-module characters of the simple $\ol{\HHH}_0$-modules, thus also the graded decomposition numbers of the baby Verma modules. More explicitly, the restriction map $\msf{G}_0^{\mrm{gr}}(\ol{\HHH}_0) \rarr \msf{G}_0^{\mrm{gr}}(KW)$ maps simple modules to simple modules, and as it is injective by Proposition \ref{graded_w_restriction_injective}, it induces a bijection between the simple modules. Hence, the graded $W$-module character of $\Delta_0(\lambda)$ from (\ref{baby_verma_graded_w_formula}) is in fact the graded decomposition of $\Delta_0(\lambda)$ into the simple $\ol{\HHH}_0$-modules.

\begin{lemma} \label{cm_0_one_family}
In $c=0$ there is just a single Calogero--Moser $c$-family.
\end{lemma}

\begin{proof}
This follows at once from Theorem \ref{blgen_explicit} since, as noted after the theorem, $\wt{\boldsymbol{\Omega}}_\lambda$ is $(\bbN \times \bbN)$-homogeneous. 
\end{proof}

 \subsection{Further properties of the genericity loci} 

\begin{lemma}
If $W \neq 1$, then $\DecGen(\ol{\HHH})$ does not contain $0 \in \mscr{C}$. 
\end{lemma}

\begin{proof}
Let $\lambda$ be a linear character of $W$. We know from Proposition \ref{linear_generic_singleton} that $\lambda$ lies alone in its Calogero--Moser $\bullet$-family. Hence, the only constituent of $\Delta_\bullet(\lambda)$ is $L_\bullet(\lambda)$. Suppose that $\dim_{\msf{k}_\CCC(\bullet)} L_\bullet(\lambda) = 1$, i.e., $L_\bullet(\lambda) \simeq \lambda^{\msf{k}_\CCC(\bullet)}$ as $\msf{k}_\CCC(\bullet)W$-modules by Lemma \ref{verma_degree_zero}. Then
\[
\lbrack \Delta_\bullet(\lambda) : L_\bullet(\lambda) \rbrack_W = \lbrack \Delta_\bullet(\lambda) : \lambda \rbrack_W = \lbrack K \lbrack \fh \rbrack^{\mrm{co}(W)} : \lambda \rbrack_W = \lbrack KW : \lambda \rbrack_W = 1 \;,
\]
so $\Delta_\bullet(\lambda) = \lambda$ and $\dim_{\msf{k}_\CCC(\bullet)} \Delta_\bullet(\lambda) = 1$. But this is not possible since $\Delta_\bullet(\lambda) \simeq KW$ and $W \neq 1$. Hence, we must have $\dim_{\msf{k}_\CCC(\bullet)} L_\bullet(\lambda) > 1$. But then we cannot have $0 \in \msf{DecGen}(\ol{\HHH})$ since $\lambda$ is $0$-rigid and therefore $\dim_K L_0(\lambda) = 1$.
\end{proof}

The $K^\times$-stability of $\DecGen(\ol{\HHH})$ from Theorem \ref{decgen_open} thus implies:

\begin{corollary} \label{decgen_single_class}
Suppose that $W \neq 1$ has just one conjugacy class of reflections. Then
\begin{equation}
\DecGen(\ol{\HHH}) = \CCC^\natural \setminus \lbrace 0 \rbrace \;.
\end{equation}
\end{corollary}

\begin{lemma}
If $W \neq 1$, then $\BlGen(\ol{\HHH})$ does not contain $0 \in \mscr{C}$. 
\end{lemma}

\begin{proof}
The group $W$ has at least two linear characters, since the determinant $\msf{det}:W \rarr K^\times$ is nontrivial. Hence, there are at least two Calogero--Moser $\bullet$-families by Proposition \ref{linear_generic_singleton}, whereas there is just one Calogero--Moser $0$-family by Lemma \ref{cm_0_one_family}.
\end{proof}

The $K^\times$-stability of $\BlGen(\ol{\HHH})$ from Lemma \ref{blgen_Kstar_stable} thus implies:

\begin{corollary} \label{blgen_single_class}
Suppose that $W \neq 1$ has just one conjugacy class of reflections. Then
\begin{equation}
\BlGen(\ol{\HHH}) = \CCC^\natural \setminus \lbrace 0 \rbrace \;.
\end{equation}
\end{corollary}

\begin{corollary}
We have $\DecGen(\ol{\HHH}) = \BlGen(\ol{\HHH})$ if $W \neq 1$ has just one conjugacy class of reflections. 
\end{corollary}

\section{Explicit results} \label{what_we_know}

\subsection{Extreme cases} 

We recall that for $c=0$ we know the solutions to all problems in Section \ref{rrca_problems} by Section \ref{c0_case}, so we can ignore this case. On the other extreme end, for a smooth simple module $L_c(\lambda)$, i.e., $\dim_{\msf{k}_\CCC(c)} L_c(\lambda) = |W|$ we know that it is in a singleton Calogero--Moser family by Theorem \ref{singleton_cm_fams}, we know its Poincaré series by Lemma \ref{bellamy_lemma}, we know that it is isomorphic to $KW$ as $W$-module by Theorem \ref{rrca_dim_upper_bound}, and we know the (non-graded) multiplicity in its baby Verma module. What we do \textit{not} know, however, is its graded $W$-module character.

\subsection{Smooth Calogero--Moser spaces} \label{sn_and_gmpn}

Assume that $K$ is algebraically closed. Etingof and Ginzburg \cite[Theorem 1.13, Corollary 1.14]{EG} have shown that for symmetric groups, the groups $G(m,1,n)$, and the cyclic groups there is an isomorphism between the Calogero--Moser space $\XXX_c$ and a certain Nakajima quiver variety $\mathfrak{M}_c$ for all closed points $c$ in an open subset of $\CCC^\natural$. The latter variety was proven to be smooth, see \cite[Lemma 1.12]{EG}, hence the Calogero--Moser space is smooth for generic $c$ for these groups. 
Bellamy \cite{Bellamy-Singular} has proven that for the exceptional group $G_4$ all simple $\ol{\HHH}_c$-modules are of dimension equal to the order of $G_4$ for all closed points $c$ in an open subset of $\CCC^\natural$. By Theorem \ref{rca_reps_geometry} this implies that $\XXX_c$ is smooth for all such $c$. 
Bellamy \cite{Bellamy-Singular} has furthermore shown that for $W$ different from the symmetric groups, the groups $G(m,1,n)$, the cyclic groups, and the group $G_4$ there always exists a supersingular character as defined in Definition \ref{supersingular}. Hence, for all these groups $\XXX_c$ is singular for all $c$. 
 From these results the complete classification of all $W$ such that $\XXX_c$ is smooth for some $c$ was obtained. To summarize:

\begin{corollary}[Bellamy, Etingof--Ginzburg, Gordon, Martino] \label{cm_smooth_if_rrca_smooth}
Assume that $K$ is algebraically closed and that $W$ is irreducible. Then the Calogero--Moser space $\XXX_c$ is smooth for some $c \in \mscr{C}$ if and only if $W$ is a symmetric group, a group of the form $G(m,1,n)$, a cyclic group, or the exceptional group $G_4$.
\end{corollary}
 
From the way the results have been obtained in the singular cases we obtain the following surprising fact:
 
\begin{corollary}[Bellamy, Etingof--Ginzburg, Gordon, Martino] 
Suppose that $K$ is algebraically closed  and that $W$ is irreducible. Let $c \in \mscr{C}$. Then the Calogero--Moser space $\XXX_c$ is smooth if and only if $\ol{\HHH}_c$ is smooth (equivalently, all closed points of $\Upsilon_c^{-1}(0)$ are smooth in $X_c$).
\end{corollary}

Due to the splitting of $\ol{\HHH}_c$ for arbitrary $c \in \CCC^\natural$ by Proposition \ref{rrca_split} and our results in Section \ref{gen_rep_theory}, we get for arbitrary $K$:

\begin{corollary}
Assume that $W$ is irreducible. Then the algebra $\ol{\HHH}_\bullet$ is smooth if and only if $W$ is a symmetric group, a group of the form $G(m,1,n)$, a cyclic group, or the exceptional group $G_4$.
\end{corollary}

\subsection{Symmetric groups}

As the symmetric group has just one conjugacy class of reflections, we only have to consider an arbitrary $0 \neq c \in \mscr{C}$ due to the $K^\times$-stability of all properties. Since $\XXX_c$ is smooth by \ref{sn_and_gmpn} the only thing we do not already know is the graded $W$-module character of the simple modules $L_c(\lambda)$. This was solved by Gordon \cite[Theorem 6.4]{Gordon-Baby} who showed that it is given by certain Kostka polynomials. Hence, all problems in Section \ref{rrca_problems} are answered for symmetric groups.

\subsection{Dihedral groups}

The classification of rigid modules by Bellamy and the author \cite{Bellamy-Thiel} shows that almost all simple modules are rigid. In Chapter \ref{Dihedral} we consider the non-rigid simple modules and solve in this way the case of dihedral groups completely.

\subsection{Calogero--Moser families for $G(m,p,n)$}

Martino \cite{Martino-blocks} has given a complete description of the Calogero--Moser families for the groups $G(m,p,n)$ for all parameters.

\subsection{Cyclic groups}

Recall from Proposition \ref{cyclic_center_by_euler} that if $W$ is cyclic, then $\ZZZ = \PPP \lbrack \eu \rbrack$. Using the theory in Section \ref{cm_semicont} the Calogero–Moser $c$-families are thus simply the Euler $c$-families for all $c$ and these can be explicitly computed. In Section \ref{Cyclic_section} we explicitly compute the simple $\ol{\HHH}_c$-modules for all $c$, and from this one can derive the solutions to the problems in Section \ref{rrca_problems}.

\subsection{Exceptional groups}

In \cite{Thiel-Champ} the author has explicitly computed the solutions to all Problems in Section \ref{rrca_problems} for $G_4$ and for all parameters using computational methods. In \cite{Thiel-Counter} we have shown that for precisely the groups 
\[
G_4, \ G_5, \ G_6, \ G_8, \ G_{10}, \ G_{23} = H_3, \ G_{24}, \ G_{25}, \ G_{26}
\]
there are only good Euler $\bullet$-families as defined in Definition \ref{good_euler_defn}, hence they are already equal to the Calogero--Moser $\bullet$-families by Lemma \ref{good_euler_cm}. Because of Corollary \ref{generic_euler_formula} one only needs the character table to compute this. Since $G_{23}=H_3$ and $G_{24}$ have just one conjugacy class of reflections, we know the Calogero--Moser families for all parameters by Corollary \ref{blgen_single_class}.   \\

With much more computational effort we have determined in \cite{Thiel-Champ} the complete solutions to all problems in Section \ref{rrca_problems} for generic parameters for the groups
\[
G_4, G_5,  G_6,  G_7, G_8, G_9,  G_{10}, G_{12}, G_{13}, G_{14}, G_{15}, G_{16}, G_{20}, G_{22}, G_{23} =H_3, G_{24} \;.
\]

In recent work by Bonnafé and the author \cite{Bonnafe-Thiel} we could compute the Calogero--Moser families for all parameters for many exceptional complex reflection groups. This includes in particular the Weyl group $G_{28}=F_4$. We refer to \cite{Bonnafe-Thiel} for the details. All results computed so far are available on the author's websites
\begin{center}
\url{http://www.mathematik.uni-stuttgart.de/~thiel/RRCA}
\end{center}
and 
\begin{center}
\url{http://thielul.github.io/CHAMP}
\end{center}
On the latter website the Cherednik Algebra Magma Package CHAMP presented by the author in \cite{Thiel-Champ} is freely available. This is a package based on the computer algebra system Magma for performing basic computations in rational Cherednik algebras at arbitrary parameters and in baby Verma modules for restricted rational Cherednik algebras. Part of this package is a new Las Vegas algorithm for computing the head and the constituents of a local module in characteristic zero which we used to explicitly compute simple modules for restricted rational Cherednik algebras. It contains a database with all the results computed so far. In the following example we show how to access the dimensions of the simple $\ol{\HHH}_\bullet$-modules for the exceptional complex reflection group $G_{20}$. 

\begin{lstlisting}[basicstyle=\tiny, basicstyle=\ttfamily]
> W := ExceptionalComplexReflectionGroup(20);
> g := Gordon(W); //the data record
> Keys(g); // this is BlGen
{
    k_{1,1} - k_{1,2}, k_{1,1} + 2*k_{1,2},
    k_{1,1}, 3*k_{1,1} - k_{1,2},
    2*k_{1,1} - k_{1,2}, k_{1,1} - 2*k_{1,2},
    2*k_{1,1} + k_{1,2}, k_{1,1} + k_{1,2},
    k_{1,2}, 3*k_{1,1} - 2*k_{1,2},
    k_{1,1} - 3*k_{1,2}, 2*k_{1,1} - 3*k_{1,2},
    1
}
> g[1]`SimpleDims;
[ 360, 360, 360, 216, 72, 216, 72, 216, 72, 27, 3, 27, 
3, 27, 3, 180, 180, 180, 180, 180, 180, 360, 360, 360, 
42, 42, 42 ]
\end{lstlisting}

The data available in CHAMP also allows to do a first sanity check on conjectures. For example, we can show that in all covered cases the Poincaré series of simple $\ol{\HHH}_c$-modules is always palindromic for generic $c$.  
\begin{lstlisting}[basicstyle=\tiny, basicstyle=\ttfamily]
for n in {4..37} do; 
W:=ExceptionalComplexReflectionGroup(n); 
try g:=Gordon(W); catch e; end try; 
if assigned g[1]`SimplePSeries then; 
for f in g[1]`SimplePSeries do; 
assert IsPalindromic(f); 
end for; end if; end for;
\end{lstlisting}
But we can also find examples where this does not hold anymore for special $c$. The following shows that for $G_4$ on the hyperplane defined by $2k_{1,1}-k_{1,2}$ there is a simple module for the restricted rational Cherednik algebra with non-palindromic Poincaré series.
\begin{lstlisting}[basicstyle=\tiny, basicstyle=\ttfamily]
for n in {4..37} do; 
W:=ExceptionalComplexReflectionGroup(n); 
try g:=Gordon(W); catch e; end try; 
for H in Keys(g) do; 
if not assigned g[H]`SimplePSeries then; 
continue; end if; 
for f in g[H]`SimplePSeries do; 
if not IsPalindromic(f) then; print n,H; break n; 
end if; end for; end for; end for;
4 2*k1_1 - k1_2
\end{lstlisting}

\begin{table}[htbp] 
\centering 

\begin{tabular}{|c|c|c|c|c|c|}
\hline
$W$ & $\DecGen(\ol{\HHH})$ & $\BlGen(\ol{\HHH})$ & $\CM_c$ & $P_{L_c(\lambda)}$ & $\lbrack L_c(\lambda) \rbrack_W^{\mrm{gr}}$ \\ \hline \hline
$S_n$ & $\checkmark$ & $\checkmark$ & $\checkmark$ & $\checkmark$ & $\checkmark$ \\ \hline \hline
$G(m,p,n)$ & -- & $\checkmark$ & $\checkmark$ & -- & -- \\ \hline
$G(m,1,n)$, $m > 1$ & $\checkmark$ & $\checkmark$ & $\checkmark$ & $\bullet$ & -- \\ \hline
$G(m,m,2) = Dih_m$ & $\checkmark$ & $\checkmark$ & $\checkmark$ & $\checkmark$ & $\checkmark$ \\ \hline
\specialcell{$G(m,m,n)$ \\ $m>1$, $n>2$} & $\checkmark$ & $\checkmark$ & $\checkmark$ & -- & -- \\ \hline \hline
$C_m$ & $\checkmark$ & $\checkmark$ & $\checkmark$ & $\checkmark$ & $\checkmark$ \\ \hline \hline
$G_4$ & $\checkmark$ & $\checkmark$ & $\checkmark$ & $\checkmark$ & $\checkmark$ \\ \hline
$G_5$ & -- & $\checkmark$ & $\checkmark$ & $\bullet$ & $\bullet$ \\ \hline
$G_6$ & -- & $\checkmark$ & $\checkmark$ & $\bullet, \bullet_H$ & $\bullet, \bullet_H$ \\ \hline
$G_7$ & -- & $\checkmark$ & $\checkmark$ & $\bullet$ & $\bullet$ \\ \hline
$G_8$ & -- & $\checkmark$ & $\checkmark$ & $\bullet, \bullet_H$ & $\bullet, \bullet_H$ \\ \hline
$G_9$ & -- & $\checkmark$ & $\checkmark$ & $\bullet$ & $\bullet$ \\ \hline
$G_{10}$ & -- & $\checkmark$ & $\checkmark$ & $\bullet$ & $\bullet$ \\ \hline
$G_{11}$ & -- & $\checkmark$ & $\checkmark$ & -- & -- \\ \hline
$G_{12}$ & $\checkmark$ & $\checkmark$ & $\checkmark$ & $\checkmark$ & $\checkmark$ \\ \hline
$G_{13}$ & -- & $\checkmark$ & $\checkmark$ & $\bullet,\bullet_H$ & $\bullet,\bullet_H$ \\ \hline
$G_{14}$ & -- & $\checkmark$ & $\checkmark$ & $\bullet,\bullet_H$ & $\bullet,\bullet_H$ \\ \hline
$G_{15}$ & -- & $\checkmark$ & $\checkmark$ & $\bullet$ & $\bullet$ \\ \hline
$G_{16}$ & -- & -- & $\bullet$ & $\bullet$ & $\bullet$ \\ \hline
$G_{17}$--$G_{19}$ & -- & -- & -- & -- & -- \\ \hline
$G_{20}$ & -- & $\checkmark$ & $\checkmark$ & $\bullet,\bullet_H$ & $\bullet,\bullet_H$ \\ \hline
$G_{21}$ & -- & -- & -- & -- & -- \\ \hline
$G_{22}$ & $\checkmark$ & $\checkmark$ & $\checkmark$ & $\checkmark$ & $\checkmark$ \\ \hline
$G_{23}=H_3$ & $\checkmark$ & $\checkmark$ & $\checkmark$ & $\checkmark$ & $\checkmark$ \\ \hline
$G_{24}$ & $\checkmark$ & $\checkmark$ & $\checkmark$ & $\checkmark$ & $\checkmark$ \\ \hline
$G_{25}$ & -- & $\checkmark$ & $\checkmark$ & -- & -- \\ \hline
$G_{26}$ & -- & $\checkmark$ & $\checkmark$ & -- & -- \\ \hline
$G_{27}$ & -- & $\checkmark$ & $\checkmark$ & -- & -- \\ \hline
$G_{28}=F_4$ & -- & $\checkmark$ & $\checkmark$ & -- & -- \\ \hline
$G_{29}$--$G_{37}$ & -- & -- & -- & -- & -- \\ \hline
\end{tabular}  \caption{Summary of results about $\ol{\HHH}$ so far. Here, ``$\checkmark$'' denotes that we know the result (for all $c$), the symbol ``$\bullet$'' denotes that we know the result for the generic point (thus for generic $c$),  the symbol ``$\bullet_H$'' denotes the we know the result for the generic point of the irreducible components (hyperplanes) of $\BlEx(\ol{\HHH})$, and ``--'' denotes that we do not know anything so far.} \label{summary_table}
\end{table}

\section{Conjectures and further problems} \label{conjectures_chapter}

We state the following conjecture:

\begin{conjecture} \label{decgen_blgen_conjecture}
We have $\DecGen(\ol{\HHH}) = \BlGen(\ol{\HHH})$.
\end{conjecture}

We know from Corollary \ref{decgen_single_class} and Corollary \ref{blgen_single_class} that this holds whenever $W$ has just a single conjugacy class of reflections. From Corollary \ref{decgen_blgen_smooth} we furthermore know that Conjecture \ref{decgen_blgen_conjecture} holds whenever $\XXX_c$ is smooth for some $c$. Hence, in total, Conjecture \ref{decgen_blgen_conjecture} holds at least for the following groups:
\begin{quote}
 $S_n$, $G(m,1,n)$, $G(m,m,n)$ with $n>2$ or $n=2$ and $m$ even, $C_m$, $G_{12}$, $G_{22}$, $G_{23} = H_3$, $G_{24}$, $G_{27}$, $G_{29}$, $G_{30} = H_4$, $G_{31}$, $G_{33}$, $G_{34}$, $G_{35}=E_6$, $G_{36}=E_7$, $G_{37}=E_8$, $G_4$.
\end{quote}

The following problem is due to Bonnafé--Rouquier \cite{BR} (see also \cite{Thiel-Champ}).

\begin{problem}
Is the set $\BlEx(\ol{\HHH})$ a union of hyperplanes? 
\end{problem}

It follows from \cite{Thiel-Blocks} that $\BlEx(\ol{\HHH})$ is a reduced Weil divisor, so if the above problem has a positive answer, then $\BlEx(\ol{\HHH})$ is indeed a \textit{finite} union of hyperplanes. Recently, Bellamy \cite{Bellamy-Counting} has shown that the above problem has a positive answer in case the Calogero--Moser space is generically smooth. The proof relies on deeper geometric properties of Calogero–Moser spaces. We would like to introduce the following strengthening of the above problem.

\begin{problem} \label{hyperplane_conjecture}
Is the set $\BlEx(\ol{\HHH})$ a union of hyperplanes with \textnormal{integral} coefficients? 
\end{problem}

\begin{problem}
Assuming that $\BlEx(\ol{\HHH})$ is a union of hyperplanes, study properties of the corresponding hyperplane arrangement.
\end{problem}

\begin{problem}
Solve the problems in Section \ref{rrca_problems}.
\end{problem}

\begin{problem}
Is the Poincaré series of $L_c(\lambda)$ palindromic for generic $c$? This problem was raised in \cite{Thiel-Champ}.
\end{problem}

\begin{problem}
Find an abstract proof of Theorem \ref{refl_grp_same_type_if_iso} (i.e., a proof not using the Shephard--Todd classification).
\end{problem}

\begin{problem}
Find an effective condition for rigidity not relying on an explicit realization of the irreducible representation. For example for $G_{34}$ we cannot compute the rigid representations simply since we do not have realizations of many of the irreducible representations.\footnote{They are so far not available in CHEVIE \cite{CHEVIE} or in the development version of the CHEVIE package of GAP3 by Michel \cite{Michel-CHEVIE}.}
\end{problem}

\begin{problem}
Find a conceptual proof of Corollary \ref{cm_smooth_if_rrca_smooth}.
\end{problem}


\begin{problem}
Is $\ol{\HHH}_c$ cellular in the sense of Graham–Lehrer \cite{GL}? Here one might have to restrict to Coxeter groups $W$.
\end{problem}

\begin{problem}
Let $\mscr{O}$ be the character value ring of $W$, i.e., the ring of integers in the character field of $W$. 
Does $\ol{\HHH}_c$ have an $\mscr{O}$-free $\mscr{O}$-form (for appropriate choices of $\mscr{O}$-valued $c$)? See \cite[\S4.2]{Thiel-Champ} for further details.
\end{problem}

\begin{problem}
The preceding problem is related to the following: Let $W$ be a complex reflection group and let $\mscr{O}$ be the ring of integers in its character field $K$ (i.e., the field generated by the values of its irreducible characters). Does the coinvariant algebra $K \lbrack \fh \rbrack^{\mrm{co}W}$ have an $\mscr{O}$-free $\mscr{O}$-form?
\end{problem}

\begin{problem}
Are the conditions in Theorem \ref{morita_dcp_closed_points} in fact equivalent to $\ZZ(\ol{\HHH}_c) \simeq \ol{\ZZZ}_c$ by \textnormal{some} algebra isomorphism?
\end{problem}

\appendix

\section{Cyclic groups} \label{Cyclic_section}

Throughout, we assume that $\fh \dopgleich \bbC$ and that $W \subs \GL(\fh)$ is a cyclic reflection group of order $m \geq 2$. We denote by $w \in W$ a generator acting by a primitive $m$-th root of unity $\zeta$ on $\fh$. By $y \dopgleich 1$ we denote the standard basis vector of $\fh$ and by $x \in \fh^* = \bbC$ we denote its dual. The conjugacy classes of reflections in $W$ are $(w^r)$ for $1 \leq r \leq m-1$, so we have parameters $c \dopgleich (c_r)_{1 \leq r \leq m-1} \in \mscr{C}$ for the rational Cherednik algebra. We will use here instead the parameters $k \dopgleich (k_r)_{1 \leq r \leq m-1}$ introduced in \cite{GGOR} as the results will have a much nicer presentations with these parameters. We furthermore set $k_0 \dopgleich 0$ and treat the indices of the parameters modulo $m$. The group $W$ has $m$ irreducible representations $\rho_0,\ldots,\rho_{m-1}$ defined by $\rho_r(w) = w^r$. For $0 \leq r \leq m-1$ we set $\Delta_{r,k} :=\Delta_k(\rho_r)$ and $L_{r,k} \dopgleich L_k(\rho_r)$. We have $\bbC \lbrack \fh \rbrack^{\mrm{co}W} = \bbC \lbrack x \rbrack/(x^m)$, so it has the $\bbC$-basis $x^r$ for $0 \leq r \leq m-1$. In particular, also $\Delta_{r,k}$ has this $\bbC$-basis. 

\subsection{Simple modules}
We will give explicit formulas for the action of $\ol{\HHH}_k$ on $\Delta_{r,k}$. For $l \in \bbZ$ we define
\begin{equation}
\gamma_{r,k}(l) \dopgleich  \left\lbrace \begin{array}{ll} 
 \sum_{j=0}^{m-1} \left( \sum_{q=1}^{m-1} \zeta^{q(j+l+r-1)} \sum_{t=0}^{l-1} (\zeta^{-q})^t \right) (k_{j+1}-k_j) & \tn{if } l \geq 0 \\ 0 & \tn{if } l \leq 0 \end{array} \right. 
\end{equation}
and
\begin{equation}
\tau(l) = \left\lbrace \begin{array}{ll} m & \tn{if } l \in m \bbZ \\ 0 & \tn{else} \end{array} \right. \;.
\end{equation}
It is not hard to see that for any $l \in \bbZ$ the relation
\begin{equation}
\sum_{q=0}^{m-1} (\zeta^l)^q = \tau(l) 
\end{equation}
holds and that furthermore for all $l \in \bbZ$ and $0 \leq r \leq m-1$ the relation
\begin{equation}
\gamma_{r,k}(l) =  \left\lbrace \begin{array}{ll} \sum_{j=1}^{m-1} ( \tau(j+r-1) - \tau(l+j+r-1) ) k_j & \tn{if } l \geq 0 \\ 0 & \tn{if } l \leq 0 \end{array} \right.
\end{equation}
holds. With this one can now prove that for $0 \leq l < m$ the relation
\begin{equation}
\gamma_{r,k}(l) = m(k_{m+1-r} - k_{m+1-r-l})
\end{equation}
holds and that for $0 \leq a \leq b < m$ the relation
\begin{equation}
\prod_{t=a}^b \gamma_{r,k}(t) = m^{b-a+1} \prod_{t=a}^b (k_{m+1-r} - k_{m+1-r-t})
\end{equation}
holds. In particular, 
\begin{equation}
\Gamma_{r,k} \dopgleich \prod_{t=1}^{m-1} \gamma_{r,k}(t) = m^{m-1} \prod_{\substack{t=0 \\ t \not\equiv m+1-r \ \textnormal{mod} \ m}}^{m-1} (k_{m+1-r} - k_{t}) \;.
\end{equation}
From this we obtain:

\begin{lemma} \label{cyc_verma_module_action}
For the operation of $\ol{\HHH}_k$ on the Verma module $\Delta_{r,k}$ the following holds for $0 \leq l \leq m-1$:
\begin{align*}
x.x^l &= x^{l+1} \\
y.x^l &= -\gamma_{r,k}(l) x^{l-1} \\
w.x^l &= \zeta^{l+r} x^l \;.
\end{align*}
More generally, the relation
\begin{align*}
(x^i y^j w^q).x^l &= (-1)^j \zeta^{q(r+l)} x^{l+i-j} \prod_{t=0}^{j-1} \gamma_{r,k}(l-t) =  (-1)^j \zeta^{q(r+l)} x^{l+i-j} \prod_{t=l-j+1}^{l} \gamma_{r,k}(t) \\ &= 
\left( (-m)^j \zeta^{q(r+l)} \prod_{t=l-j+1}^{l}(k_{m+1-r} - k_{m+1-r-t}) \right)  x^{l+i-j} 
\end{align*}
holds.
\end{lemma}

\begin{definition}
For $0 \leq r \leq m-1$ let $\eps_{r,k}$ be the minimum of the set 
\[
\lbrace 0 < l \leq m-1 \mid \gamma_{r,k}(l) = 0 \rbrace
\]
if this set is not empty, and $\eps_{r,k} \dopgleich m$ if it is empty. 
\end{definition}

\begin{theorem} \label{cyc_simple_modules}
The radical of $\Delta_{r,k}$ is equal to $\langle x^{\eps_{r,k}}, \ldots, x^{m-1} \rangle_\bbC$. In particular we have
\[
\dim L_{r,k} = \msf{codim}(  \Rad(\Delta_{r,k}))  = \eps_{r,k} 
\]
and $\Delta_{r,k}$ is irreducible if and only if $\eps_{r,k} = m$.
\end{theorem}

\begin{proof}
Let $V := \langle x^{\eps_{r,k}}, \ldots, x^{m-1} \rangle_\bbC$. This subspace is obviously closed under the operation of $x$ and $w$. If $\eps_{r,k} = m$, then $V = 0$ and this is an $\ol{\HHH}_k$-submodule. For $\eps_{r,k} < m$ we obviously have $y.x^l \in V$ for $l > \eps_{r,k}$ and $y.x^{\eps_{r,k}} = -\gamma_{r,k}({\eps_{r,k}}) x^{\eps_{r,k}-1} = 0$. Hence, $V$ is an $\ol{\HHH}_k$-submodule. Moreover, $V$ is a proper submodule of $\Delta_{r,k}$ since $\eps_{r,k} > 0$. It remains to shows that $V$ is maximal. So, let $U$ be a submodule with $V < U \leq \Delta_{r,k}$. Then there exists $u \in U$ with $l \in \msf{Supp}(u)$ for some $0 \leq l < \eps_{r,k}$. Let $l \dopgleich \msf{max} \lbrace i \in \msf{Supp}(u) \mid i < \eps_{r,k} \rbrace$. Then 
\[
u = \sum_{i=0}^{m-1} \alpha_i x^i = \sum_{i=0}^{l-1} \alpha_ix^i + \alpha_lx^l + \sum_{i = l+1}^{m-1} \alpha_i x^i
\]
for certain $\alpha_i \in L$. Because of the formula in Lemma  \ref{cyc_verma_module_action} we have
\[
y^l.x^l = (-1)^l \prod_{t=0}^{l-1} \gamma_{r,k}(l-t) x^0.
\]
If $l = 0$, then the coefficient of $x^0$ is equal to $1$. Otherwise, the factors $\gamma_{r,k}(l)$ to $\gamma_{r,k}(1)$ occur in the product and since  $l < \eps_{r,k}$, all these factors are non-zero. So, in general we have $y^l.(\alpha_lx^l) \neq 0$. For $i < l$ we have $y^l.x^i = 0$. For $i > l$ we have $y.(\alpha_ix^i) \in V$ since either $i < \eps_{r,k}$ and therefore $i \notin \msf{Supp}(u)$ because of the maximality of $l$, or $i \geq \eps_{r,k}$ and therefore $\alpha_ix^i \in V$ by definition. Consequently,
\[
U \ni y^l.u = \alpha_l'x^0 + \underbrace{ \sum_{i=l+1}^{m-1} \alpha_i' x^{i-l} }_{\in V}
\]
for certain $\alpha_i' \in L$ with $\alpha_l' \neq 0$. Hence, $\alpha_l'x^0 \in U$ and therefore $x^0 \in U$. This finally shows that $x^i = x^i.x^0 \in U$ for all $i$, i.e., $U = \Delta_{r,k}$. So, $V$ is maximal and therefore $V = \Rad(\Delta_{r,k})$.
\end{proof}

\begin{definition}
We call the tuple $\eps_k \dopgleich (\eps_{0,k},\ldots,\eps_{m-1,k})$ the \word{type} of $\ol{\HHH}_k$.
\end{definition}

\begin{corollary} \label{rrca_cyclic_semisimple}
The $\bbC$-algebra $\ol{\HHH}_k$ is semisimple if and only if $\eps_{r,k} = m$ for all  $0 \leq r \leq m-1$, or, equivalently, $\Gamma_{r,k} \neq 0$ for all $0 \leq r \leq m-1$. 
\end{corollary}

\begin{proof}
This follows from Theorem \ref{cyc_simple_modules} in conjunction with Lemma \ref{rrca_semisimple}.
\end{proof}

\begin{proposition}
For $0 \leq a \leq m$ let $U_{r}^a \dopgleich \langle x^a,\ldots,x^{m-1} \rangle_\bbC \subs \Delta_{r,k}$. In case that $U_r^a$ is an $\ol{\HHH}_k$-submodule of $\Delta_{r,k}$, its $\ol{\HHH}_k$-character is given by
\begin{align*}
\chi_{U_r^a}(x^iy^jw^q) &= \delta_{ij} (-1)^i \zeta^{qr} \sum_{l=a}^{m-1} \left( \zeta^{ql} \prod_{t=0}^{i-1} \gamma_{r,k}({l-t}) \right) \\ &= \delta_{ij} (-m)^i \zeta^{qr} \sum_{l = \mrm{max} \lbrace a,i \rbrace }^{m-1} \left( \zeta^{ql} \prod_{t=l-i+1}^l (k_{m+1-r} - k_{m+1-r-t}) \right) \;. 
\end{align*}
\end{proposition}

\begin{proof}
If $i-j \neq 0$, then it follows from the formula for the action of $\ol{\HHH}_k$ on $\Delta_{r,k}$ given in Lemma \ref{cyc_verma_module_action} that $\chi_r(x^iy^j w^q ) = 0$. For $i-j = 0$ the asserted relation also follows from this formula.
\end{proof}

\begin{corollary}
The $\ol{\HHH}_k$-character of $L_{r,k}$ is given by 
\[
\chi_{r,k} \dopgleich \chi_{L_{r,k}} = \chi_{\Delta_{r,k}} - \chi_{\Rad(\Delta_{r,k})} = \chi_{U_r^0} - \chi_{U_r^{\eps_{r,k}}} \;. 
\]
More explicitly, it is
\begin{align*}
\chi_{r,k}(x^i y^j w^q) &= \delta_{ij} (-m)^i \zeta^{qr} \sum_{l=i}^{\msf{min} \lbrace \msf{max} \lbrace \eps_{r,k},i \rbrace, m-1 \rbrace} \left( \zeta^{ql} \prod_{t=l-i+1}^l (k_{m+1-r} - k_{m+1-r-t}) \right) \;. 
\end{align*}
In particular, the relation
\begin{align*}
\boldsymbol{\chi}_r(x^i y^j w^q) &= \delta_{ij} (-m)^i \zeta^{qr} \sum_{l=i}^{m-1} \left( \zeta^{ql} \prod_{t=l-i+1}^l (k_{m+1-r} - k_{m+1-r-t}) \right) \\ &= \delta_{ij} (-m)^i \zeta^{qr} \sum_{l=i}^{m-1} \left( \zeta^{ql} \prod_{t=m+1-r-l}^{m-r-l+i} (k_{m+1-r} - k_{t}) \right) \;.
\end{align*}
holds.
\end{corollary}

\subsection{The Schur elements}
\label{rrca_cyclic_schur_elements}

Recall from Corollary \ref{restrictions_symmetric} and Proposition \ref{rrca_split} that $\ol{\HHH}_k$ is a split symmetric Frobenius $\bbC$-algebra with symmetrizing trace $\Phi_k$. Hence, by the theory in \cite[\S7]{Geck-Pfeiffer} there is a Schur element $S_{r,k}$ defined for each simple $\ol{\HHH}_k$-module $L_{r,k}$. 

\begin{theorem} \label{schur_elements_cyclic_rrca}
Suppose that $\ol{\HHH}_k$ is semisimple, see Corollary \ref{rrca_cyclic_semisimple}. The Schur element $S_{r,k}$ of $L_{r,k}$ is equal to
\[
S_{r,k} = (-1)^{m-1}m \Gamma_{r,k} = (-1)^{m-1} m^m \prod_{\substack{t=0 \\ t \not\equiv m+1-r \ \textnormal{mod}\ m}}^{m-1} (k_{m+1-r} - k_{t}).
\]
\end{theorem}

\begin{proof}
Since $\ol{\HHH}_k$ is semisimple, the Schur elements are all non-zero and are uniquely determined by the linear system
\begin{align*}
\Phi_k( x^i y^j w^q) = \sum_{r=0}^{m-1} \frac{1}{S_{r,k}} \chi_{r,k}(x^iy^j w^q) \;,
\end{align*}
where $1 \leq i,j,q \leq m-1$. 
For $i \neq j$ both sides are equal to zero and therefore equal. So, let $i=j$. Then
\begin{align*}
& \sum_{r=0}^{m-1} \frac{1}{S_{r,k}} \chi_{r,k}(x^iy^j w^q) \allowdisplaybreaks[3]\\
&= \sum_{r=0}^{m-1} \frac{1}{(-1)^{m-1} m^m \prod_{\substack{t=0 \\ t \not\equiv m+1-r \ \textnormal{mod}\ }}^{m-1} (k_{m+1-r} - k_{t})} \allowdisplaybreaks[3]\\ 
& \ \ (-m)^i \zeta^{qr} \sum_{l=i}^{m-1} \left( \zeta^{ql} \prod_{t=m+1-r-l}^{m-r-l+i} (k_{m+1-r} - k_{t}) \right) \allowdisplaybreaks[1]\\ 
&= (-1)^{m-1-i} m^{i-m}  \sum_{r=0}^{m-1} \sum_{l=i}^{m-1} \zeta^{q(r+l)} \frac{\prod_{t=m+1-r-l}^{m-r-l+i} (k_{m+1-r} - k_{t})}{\prod_{\substack{t=0 \\ t \not\equiv m+1-r \ \textnormal{mod}\ m}}^{m-1} (k_{m+1-r} - k_{t})}  \allowdisplaybreaks[1]\\ 
& \underbrace{=}_{m+1-r \rarr r}  (-1)^{m-1-i} m^{i-m}  \sum_{r=0}^{m-1} \sum_{l=i}^{m-1} \zeta^{q(m+1-r+l)} \frac{\prod_{t=r-l}^{r-l+i-1} (k_{r} - k_{t})}{\prod_{\substack{t=0 \\ t \neq r}}^{m-1} (k_{r} - k_{t})} \allowdisplaybreaks[1]\\ 
&= (-1)^{m-1-i} m^{i-m} \zeta^q  \sum_{r=0}^{m-1} \sum_{l=i}^{m-1} \zeta^{-q(r-l)} \frac{\prod_{t=r-l}^{r-l+i-1} (k_{r} - k_{t})}{\prod_{\substack{t=0 \\ t \neq r}}^{m-1} (k_{r} - k_{t})} \allowdisplaybreaks[1]\\ 
& \underbrace{=}_{r-l \rarr d} (-1)^{m-1-i} m^{i-m} \zeta^q  \sum_{r=0}^{m-1} \sum_{d=r-m+1}^{r-i} \zeta^{-qd} \frac{\prod_{t=d}^{d+i-1} (k_{r} - k_{t})}{\prod_{\substack{t=0 \\ t \neq r}}^{m-1} (k_{r} - k_{t})} \allowdisplaybreaks[1]\\ 
&= (-1)^{m-1-i} m^{i-m} \zeta^q \sum_{d=1}^{2m-1-i} \zeta^{-qd} \sum_{r=d-1}^{d+i-m} \frac{ \prod_{t=d}^{d+i-1} (k_r-k_t)}{ \prod_{\substack{t=0 \\ t\neq r}}^{m-1} (k_r-k_t)} \\ &= (-1)^{m-1-i} m^{i-m} \zeta^q \sum_{d=1}^{2m-1-i} \zeta^{-qd} \sum_{r=d-1}^{(d-1)+i-(m-1)} \frac{ \prod_{t=d}^{d+i-1} (k_r-k_t)}{ \prod_{\substack{t=0 \\ t\neq r}}^{m-1} (k_r-k_t)}.
\end{align*}
If $i< m-1$, then $(d-1)+i-(m-1) < d-1$, so that the above sum is empty and therefore equal to zero. Hence, in this case we have
\[
\sum_{r=0}^{m-1} \frac{1}{S_{r,k}} \chi_{r,k}(x^iy^j w^q) = 0 = \Phi(x^iy^jw^q).
\]
On the other hand, for $i=m-1$ we have
\begin{align*}
\sum_{r=0}^{m-1} \frac{1}{S_{r,k}} \chi_{r,k}(x^iy^j w^q) & = \frac{1}{m} \zeta^q \sum_{d=1}^{m} \zeta^{-qd} \sum_{r=d-1}^{d-1} \frac{ \prod_{t=d}^{d+m-2} (k_r-k_t)}{ \prod_{\substack{t=0 \\ t\neq r}}^{m-1} (k_r-k_t)} \allowdisplaybreaks[1]\\ &=  \frac{1}{m} \zeta^q \sum_{d=1}^{m} \zeta^{-qd}  \frac{ \prod_{\substack{t=0 \\ t \neq d-1}}^{m-1} (k_{d-1}-k_t)}{ \prod_{\substack{t=0 \\ t\neq d-1}}^{m-1} (k_{d-1}-k_t)} \allowdisplaybreaks[1]\\ &=  \frac{1}{m} \zeta^q \sum_{d=1}^{m} \zeta^{-qd} = \frac{1}{m} \zeta^q \sum_{d=0}^{m-1} \zeta^{qd}.
\end{align*}
If $q>0$, then $ \sum_{d=0}^{m-1} \zeta^{qd} = 0$ and therefore
\[
\sum_{r=0}^{m-1} \frac{1}{S_{r,k}} \chi_{r,k}(x^iy^j w^q) = 0 = \Phi_k(x^iy^jw^q).
\]
If $q=0$, then $ \sum_{d=0}^{m-1} \zeta^{qd} = m$ and therefore
\[
\sum_{r=0}^{m-1} \frac{1}{S_{r,k}} \chi_{r,k}(x^iy^j w^q) = \frac{1}{m} m = 1 = \Phi_k(x^iy^jw^q).
\]
This shows that the elements $S_{r,k}$ are indeed the Schur elements.
\end{proof}

\subsection{Character tables}

\begin{theorem}
Suppose that $\ol{\HHH}_k$ is semisimple. Then the dimension of $\ol{\HHH}_k/ \lbrack \ol{\HHH}_k, \ol{\HHH}_k \rbrack$ is equal to $|W|=m$ and has  
\[
(\Omega \Omega^* g)_{g \in W}
\]
as basis, where $\Omega \dopgleich x^{m-1}$, resp.\ $\Omega^* \dopgleich y^{m-1}$,  are the fundamental classes of $\bbC \lbrack \fh \rbrack^{\mrm{co}W}$, resp.\ of $\bbC \lbrack \fh^* \rbrack^{\mrm{co}W}$.
\end{theorem}

\begin{proof}
Since $\ol{\HHH}_k$ is split semisimple, we have $\dim \ol{\HHH}_k/ \lbrack \ol{\HHH}_k, \ol{\HHH}_k \rbrack = \# \Irr \ol{\HHH}_k = |W|$. It thus remains to show that the given elements are linearly independent modulo $\lbrack \ol{\HHH}_k,\ol{\HHH}_k \rbrack$. So, suppose that $\sum_{q=0}^{m-1} \alpha_q \Omega \Omega^* w^q \equiv 0 \ \textnormal{mod} \ \lbrack \ol{\HHH}_k,\ol{\HHH}_k \rbrack$ for certain $\alpha_q \in \bbC$, i.e., $\sum_{q=0}^{m-1} \alpha_q \Omega \Omega^* w^q \in \lbrack \ol{\HHH}_k,\ol{\HHH}_k \rbrack$. For every $0 \leq r \leq m-1$ we know that $\chi_{r,k}$ is a class function on $\ol{\HHH}_k$ and therefore
\begin{align*}
0 & = \sum_{q=0}^{m-1} \alpha_q \chi_{r,k}(\Omega \Omega^* w^q) \\
& = \sum_{q=0}^{m-1} \alpha_q (-m)^{m-1} \zeta^{qr} \sum_{l=m-1}^{m-1} \left( \zeta^{ql} \prod_{t=m+1-r-l}^{m-r-l+m-1} (k_{m+1-r} - k_{t}) \right) \\ &= (-m)^{m-1} \sum_{q=0}^{m-1} \alpha_q \zeta^{qr} \zeta^{q(m-1)} \prod_{2-r}^{m-r} (k_{m+1-r} - k_t) \\ & = (-m)^{m-1} \prod_{2-r}^{m-r} (k_{1-r} - k_t) \sum_{q=0}^{m-1} \alpha_q (\zeta^{r-1})^q \;.
\end{align*}
This is satisfied if and only if
\[
\sum_{q=0}^{m-1} \alpha_q (\zeta^{r-1})^q = 0 \;.
\]
The linear system 
\[
\sum_{q=0}^{m-1} \alpha_q (\zeta^{r})^q = 0 \quad \textnormal{for all} \quad 0 \leq r \leq m-1
\]
has a non-trivial solution over $\bbC$ if and only if the determinant of the matrix $( (\zeta^r)^q )_{ (r,q) \in \lbrack 0,m-1 \rbrack^2}$ is zero. This is a Vandermonde matrix and as $\zeta^r \neq \zeta^{r'}$ for $0 \leq r,r' \leq m-1$ and $r \neq r'$, this determinant is indeed non-zero and so there is just the trivial solution, i.e., $\alpha_q = 0$ for all $q$.
\end{proof}

\begin{example}
Suppose that $m=3$, i.e., $W$ is the cyclic group of order $3$. If $\ol{\HHH}_k$ is semisimple, its character table is equal to 

{ \centering \small

\begin{tabular}{|l|lll|}
\hline
& $\Omega \Omega^* w^0$ & $\Omega \Omega^* w^1$ & $\Omega \Omega^* w^2$ \\
\hline
$\chi_{0,k}$ & $9k_1^2 - 9k_1k_2$ & $(-9\zeta - 9)k_1^2 + (9\zeta + 9)k_1k_2$ & $9\zeta k_1^2 - 9\zeta k_1k_2$ \\
$\chi_{1,k}$ & $9k_1k_2$ & $9k_1k_2$ & $9k_1k_2$ \\
$\chi_{2,k}$ & $-9k_1k_2 + 9k_2^2$ & $-9\zeta k_1k_2 + 9\zeta k_2^2$ & $(9\zeta  + 9)k_1k_2 + (-9\zeta  - 9)k_2^2$ \\
\hline
\end{tabular}
}
\end{example}

\section{Dihedral groups} \label{Dihedral}

In the following we will deal with dihedral groups. We repeat some bits of \cite{Bellamy-Thiel} here to be able to use the same setup. Throughout, we assume that $m \geq 5$ and choose a primitive $m$-th root of unity $\zeta \in \bbC$. Let $W$ be the Coxeter group of type $I_2(m)$. This is the dihedral group of order $2m$. It has two natural presentations, namely the Coxeter presentation $\langle s,t \mid s^2=t^2=(st)^m=1 \rangle$ and the geometric presentation $\langle s,r \mid r^m=1, s^2=1, s^{-1}rs = r^{-1} \rangle$ with a generating rotation $r \dopgleich st$ for the symmetries of a regular $m$-gon. 
The representation theory of $W$ depends on the parity of $m$. If $m$ is odd, the conjugacy classes of $W$ are
\[
\lbrace 1 \rbrace, \; \lbrace r^{\pm 1} \rbrace, \; \lbrace r^{\pm 2 } \rbrace, \; \ldots, \; \lbrace r^{\pm (m-1)/2 } \rbrace, \; \lbrace r^ls \mid 0 \leq  l \leq m-1 \rbrace  \;,
\]
and so the total number of conjugacy classes is $(m+3)/2$. There are two irreducible one-dimensional representations: the trivial one $1_W$ and the sign representation $\eps:W \rarr \bbC$ with
\[
\eps(s) = -1 \;, \quad \eps(t) = -1\;,  \quad \eps(r) = 1 \;.
\]
The remaining $(m+3)/2-2 = (m-1)/2$ irreducible representations $\varphi_i$, $1 \leq i \leq (m-1)/2$, are all two-dimensional and are given by
\[
\varphi_i(s) = \begin{pmatrix} 0 & 1 \\ 1 & 0 \end{pmatrix}, \; \quad \varphi_i(t) \dopgleich \begin{pmatrix} 0 & \zeta^{-i} \\ \zeta & 0 \end{pmatrix} \;, \quad \varphi_i(r) = \begin{pmatrix} \zeta^i & 0 \\ 0 & \zeta^{-i} \end{pmatrix} \;.
\]
We denote the character of $\varphi_i$ by $\chi_i$. If $m$ is even, then the conjugacy classes of $W$ are
\begin{align*}
& \lbrace 1 \rbrace \;, \; \lbrace r^{\pm 1} \rbrace \;, \; \lbrace r^{\pm 2 } \rbrace \;, \; \ldots,\lbrace r^{\pm m/2 } \rbrace \; , \\
& \lbrace r^{2k}s \mid 0 \leq k \leq (m/2)-1 \rbrace \;,\; \lbrace r^{2k+1}s \mid  0 \leq k \leq (m/2)-1 \rbrace \;,
\end{align*}
and so the total number of conjugacy classes is $(m+6)/2$. There are four irreducible one-dimensional representations: the trivial one $1_W$, the sign representation $\eps$, and two further representations $\eps_1,\eps_2$ with
\[
\begin{array}{llll}
\eps(s) = -1 \;, & \eps(t) = -1\;, & \eps(r) = 1 \;, \\
\eps_1(s) = 1 \;, & \eps_1(t) = -1 \;, & \eps_1(r) = -1 \;, \\
\eps_2(s) = -1 \;, & \eps_2(t) = 1 \;, & \eps_2(r) = -1 \;.
\end{array}
\]
The remaining $(m+6)/2-4 = (m-2)/2$ irreducible representations $\varphi_i$, $1 \leq i \leq (m-2)/2$, are all two-dimensional and are defined as in case $m$ is odd. Again, we denote the character of $\varphi_i$ by $\chi_i$.

\subsection{Reflections and parameters} \label{dihedral_reflections}
The two-dimensional faithful irreducible representation $\varphi_1$ of $W$ is a reflection representation in which precisely the elements $s_l \dopgleich r^l s$ for $0 \leq l \leq m-1$ act as reflections. We will always fix this representation as the reflection representation of $W$. Let $(y_1,y_2)$ be the standard basis of $\fh \dopgleich \bbC^2$ and let $(x_1,x_2)$ be the dual basis. Roots and coroots for the reflections $s_l$ are given by
\[
\alpha_{s_l} = x_1 - \zeta^{-l} x_2 \quad \tn{and} \quad \alpha_{s_l}^\vee = y_1 - \zeta y_2 \;.
\]
With this we see that the Cherednik coefficients $(y_i,x_j)_{s_l} = -(y_i,\alpha_{s_l})(\alpha_{s_l}^\vee, x_j)$ are:
\[
  (y_1,x_1)_{s_l} = -1 \;, \quad (y_1,x_2)_{s_l} =  \zeta^{-l} \;, \quad (y_2,x_1)_{s_l} = \zeta^{l}  \;, \quad (y_2,x_2)_{s_l} = -1 \;.
\]
If $m$ is odd, there is just one conjugacy class of reflections in $W$, namely the one of $s$ which is $\lbrace s_l \mid 0 \leq l \leq m-1 \rbrace$. If $m$ is even, there are two conjugacy classes of reflections in $W$, namely the one of $s$ which is $\lbrace s_{2l} \mid 0 \leq l \leq \frac{m}{2}-1 \rbrace$ and the one of $t$ which is $\lbrace s_{2l+1} \mid 0 \leq l \leq \frac{m}{2} -1 \rbrace$. 
Note that 
\[
\varphi_i(s_l) = \begin{pmatrix} 0 & \zeta^{il} \\ \zeta^{-il} & 0 \end{pmatrix} \;.
\]

If $c:\Ref(W) \rarr \bbC$ is a function which is constant on conjugacy classes, then we define
\begin{equation}
b \dopgleich c(s) \;, \quad a \dopgleich c(t) \;.
\end{equation}
We fix such a function from now on and assume that $c \neq 0$. Note that if $m$ is odd, we have $a=b$.

\begin{remark}
In \cite{Bellamy-Thiel} we used a slightly different relation for the rational Cherednik algebra but this simply amounts to replacing our parameters $a,b$ by $-2a,-2b$. Since everything is $\bbC^*$-stable, this does not change anything.
\end{remark}

\subsection{Rigid modules}

Table \ref{dihedral_bigtable} is a summary of the results obtained in \cite{Bellamy-Thiel}. To simplify notations, we denote by $\mathcal{F}$ the set of two-dimensional irreducible characters of $W$. 
To allow a presentation which is independent of the parity of $m$ we set 
\[
\mathcal{R} \dopgleich \left\lbrace \begin{array}{ll} \lbrace \varphi_i \mid 1 < i \leq (m-1)/2 \rbrace = \mathcal{F} \setminus \lbrace \varphi_1 \rbrace& \tn{if } m \tn{ is odd} \\ \lbrace \varphi_i \mid 1 < i < (m-2)/2 \rbrace = \mathcal{F} \setminus \lbrace \varphi_1, \varphi_{(m-2)/2} \rbrace& \tn{if } m \tn{ is even.} \end{array} \right.
\]
We make the convention that we ignore $\eps_1$ and $\eps_2$ whenever $m$ is odd.

\begin{theorem}[Bellamy-T.] \label{dihedral_CM_theorem}
The (cuspidal) Calogero–Moser families and rigid representations of $\ol{\HHH}_{c}(W)$ are as listed in Table \ref{dihedral_bigtable}. 
\begin{table}[htbp]
\begin{tabular}{|c|c|c|c|}
\hline
Parameters & CM families & \specialcell{rigid \\ representations} & \specialcell{cuspidal \\ CM families} \\ \hline \hline
$a,b \neq 0$ and $a \neq \pm b$ & $\lbrace 1 \rbrace$, $\lbrace \eps \rbrace$, $\lbrace \eps_1 \rbrace$, $\lbrace \eps_2 \rbrace$, $\mathcal{F}$ & $\mathcal{R}$ & $\mathcal{F}$ \\ \hline
$a = 0$ and $b \neq 0$ & $\lbrace 1, \eps_2 \rbrace$, $\lbrace \eps,\eps_1 \rbrace$, $\mathcal{F}$ &  $\mathcal{R}$ & $\mathcal{F}$ \\ \hline
$a \neq 0$ and $b = 0$ &  $\lbrace 1, \eps_1 \rbrace$, $\lbrace \eps,\eps_2 \rbrace$, $\mathcal{F}$ & $\mathcal{R}$ & $\mathcal{F}$ \\ \hline
$a = b \neq 0$ & $\lbrace 1 \rbrace$, $\lbrace \eps \rbrace$, $\lbrace \eps_1, \eps_2 \rbrace \cup \mathcal{F}$ & $\eps_1$, $\eps_2$, $\varphi_{|\mathcal{F}|}$, $\mathcal{R}$ & $\lbrace \eps_1, \eps_2 \rbrace \cup \mathcal{F}$ \\ \hline
$a = -b \neq 0$ & $\lbrace \eps_1 \rbrace$, $\lbrace \eps_2 \rbrace$, $\lbrace 1,\eps \rbrace \cup \mathcal{F}$ & $1,\eps,\varphi_1$, $ \mathcal{R}$ & $\lbrace 1,\eps \rbrace \cup \mathcal{F}$ \\ \hline
\end{tabular}
\caption{The (cuspidal) Calogero–Moser families and rigid representations for dihedral groups.}
\label{dihedral_bigtable}
\end{table}
\end{theorem}

\subsection{Basis of the coinvariant algebra}

By \cite[Chapter II, \S8]{LT} a system fundamental invariants of $\bbC \lbrack \fh \rbrack^W$ is formed by $x_1^m + x_2^m$ and $x_1x_2$. 

\begin{proposition} \label{dihedral_coinv_basis}
Fix the lexicographical ordering on $\bbC \lbrack \fh \rbrack = \bbC \lbrack x_1,x_2 \rbrack$ with $x_1 > x_2$. Then the elements 
\[
x_1^m+x_2^m, \ x_1x_2, \tn{ and } x_2^{m+1}
\]
form a Gröbner basis for the Hilbert ideal of the action of $W$ on $\bbC \lbrack \fh \rbrack$ and the (images of the) elements 
\[
1, \ x_1^i \tn{ for } 1 \leq i \leq m-1, \tn{ and } x_2^j \tn{ for } 1 \leq j \leq m
\]
form a monomial basis of the coinvariant algebra of $W$. Replacing $x$ by $y$ gives the analogous statement for $W^*$.
\end{proposition}

\begin{proof}
 The Hilbert ideal is generated by $f_1 \dopgleich x_1^m+x_2^m$ and $f_2 \dopgleich x_1x_2$. We will now complete this generating system to a Gröbner basis by employing the Buchberger algorithm, see \cite[\S21.5]{Gathen}. To this end, we first have to compute the S-polynomial of $f_1$ and $f_2$. In general, the S-polynomial of two polynomials $p,q \in \bbC \lbrack x_1,x_2 \rbrack$ is defined as
 \[
 S(p,q) \dopgleich \frac{\mathbf{x}^{\gamma(p,q)}}{\mrm{LT}(p)}p - \frac{\mathbf{x}^{\gamma(p,q)}}{\mrm{LT}(q)}q \;,
 \]
 where $\mathbf{x} \dopgleich \lbrace x_1,x_2 \rbrace$ and $\mrm{LT}$ denotes the leading terms and 
 \[
 \gamma(p,q) \dopgleich (\max\lbrace \alpha_1,\beta_1\rbrace, \max\lbrace\alpha_2,\beta_2\rbrace)
 \]
  with $\alpha \dopgleich \mrm{mdeg}(p)$ being the multi-degree of $p$ and $\beta \dopgleich \mrm{mdeg}(q)$ being the multi-degree of $q$. In our case we have $\mrm{mdeg}(f_1) = (m,0)$ and $\mrm{mdeg}(f_2) = (1,1)$. Hence, $\gamma(f_1,f_2) = (m,1)$ and therefore
 \begin{align*}
 f_3 \dopgleich S(f_1,f_2) &= \frac{\mathbf{x}^{(m,1)}}{x_1^m}f_1 - \frac{\mathbf{x}^{(m,1)}}{x_1x_2}f_2 = \frac{x_1^mx_2}{x_1^m}(x_1^m+x_2^m) - \frac{x_1^mx_2}{x_1x_2} x_1x_2 \\
 & = x_1^mx_2 + x_2^{m+1} - x_1^mx_2 = x_2^{m+1} \;.
 \end{align*}
 Now, we have to compute the remainder $f_3 \ \mrm{rem} \ (f_1,f_2)$ of $f_3$ when dividing by $(f_1,f_2)$ using the multivariate division algorithm, see \cite[21.11]{Gathen}. The leading terms of $f_1$ and $f_2$ do not divide $f_3$ and therefore  
 \[
 f_3 \ \mrm{rem} \ (f_1,f_2) = f_3 \;.
 \]
 According to the Buchberger algorithm this means that we have to add $f_3$ to the basis. In the next round of this algorithm the remainder of $f_3 = S(f_1,f_2)$ when dividing by $(f_1,f_2,f_3)$ is of course zero. Furthermore, we have
 \begin{align*}
 S(f_1,f_3) &= \frac{\mathbf{x}^{(m,m+1)}}{x_1^m}f_1 - \frac{\mathbf{x}^{(m,m+1)}}{x_2^{m+1}}f_3 = \frac{x_1^mx_2^{m+1}}{x_1^m}(x_1^m+x_2^m) - \frac{x_1^mx_2^{m+1}}{x_2^{m+1}}x_2^{m+1} \\ &= x_2^{2m+1} 
 \end{align*}
 and
 \[
 S(f_2,f_3) = \frac{\mathbf{x}^{(1,m+1)}}{x_1x_2}f_2 - \frac{\mathbf{x}^{(1,m+1)}}{x_2^{m+1}}f_3 = \frac{x_1x_2^{m+1}}{x_1x_2}x_1x_2 - \frac{x_1x_2^{m+1}}{x_2^{m+1}}x_2^{m+1} = 0 \;.
 \]
 As 
 \[
 S(f_1,f_3) - \frac{\mrm{LT}(S(f_1,f_3))}{\mrm{LT}(f_3)}f_3 = x_2^{2m+1} - \frac{x_2^{2m+1}}{x_2^{m+1}}x_2^{m+1} = 0 \;,
 \]
 the residue of $f_3$ when dividing by $(f_1,f_2,f_3)$ is equal to zero. So, the residue of the S-polynomial of any pair in $(f_1,f_2,f_3)$ when dividing by this triple is equal to zero and so the Buchberger criterion implies that they form a Gröbner basis of the Hilbert ideal. 
By \cite[Theorem 1.2.8]{Geck-AG} a monomial basis of the coinvariant algebra $K \lbrack x_1,x_2 \rbrack/\langle f_1,f_2,f_3 \rangle$ of $W$ is now formed by the images of the elements
\[
\lbrace \mathbf{x}^\alpha \mid \alpha \in \bbN^2, \mathbf{x}^\alpha \tn{ is not divisible by any of } x_1^m,\  x_1x_2, \ x_2^{m+1} \rbrace \;.
\]
The monomials given in the statement are precisely those satisfying this property.
\end{proof}

\subsection{Odd dihedral groups}

Let $c \neq 0$. From Proposition \ref{linear_generic_singleton} we know that the $L_c(\lambda)$ attached to one-dimensional $\lambda$ lie in singleton Calogero--Moser $\bullet$-families, so from Theorem \ref{singleton_cm_fams} we obtain:

\begin{corollary}
Both $L_c(1)$ and $L_c(\eps)$ are smooth. 
\end{corollary} 

Due to the classification of rigid modules in Theorem \ref{dihedral_CM_theorem} only the simple module $L_c(\varphi_1)$ is not yet understood. To simplify some formulas, we introduce the following notation. Let $1 \leq i,j \leq 2$. For any $k \in \bbN$ we set $x_{2k+i} \dopgleich x_i$ and $y_{2k+j} \dopgleich y_j$, i.e., we extend the indices $2$-periodically. The following lemma is a straightforward computation.

\begin{lemma} \label{dihedral_commutator_relation}
 Let $1 \leq i,j \leq 2$. In $\ol{\HHH}_c$ we have
 \[
 \lbrack y_j, x_i \rbrack = (-1)^{(i-j)} \frac{1}{2} c \sum_{l=0}^{m-1} \zeta^{l(i-j)} r^l s \;,
 \]
 and if $r>1$ we have
 \[
 \lbrack y_j,x_i^r \rbrack = (-1)^{i-j} \frac{1}{2} c \sum_{l=0}^{m-1} \left( \zeta^{l \left( (-1)^i r + (|i-j|-1)(-1)^i \right)} x_{i+1}^{r-1} + \zeta^{l(i-j)} x_i^{r-1}  \right) r^ls \;.
 \]
\end{lemma}

A basis of $\Delta_c(\varphi_p)$ for $1 \leq p \leq (m-1)/2$ is given by the elements $1 \otimes x_q$, $x_1^r \otimes x_q$ for $1 \leq r \leq m-1$, and $x_2^r \otimes x_q$ for $1 \leq r \leq m$, where always $1 \leq q \leq m$. Using the above commutator formulas, we can derive formulas for the action of $\ol{\HHH}_c$ on the Verma modules in this basis.

\begin{lemma}
For $1 \leq i,j,q \leq 2$ and $1 \leq p \leq (m-1)/2$ the action of $\ol{\HHH}_c$ on $\Delta_c(\rho_p)$ is given by
 \[
 y_j.(x_i \otimes x_q) = (-1)^{i-j+1} \frac{1}{2} c \sum_{l=0}^{m-1} \zeta^{l(i-j+(-1)^qp)} \otimes x_{q+1} \;
 \]
 and
 \begin{align*}
 y_j.(x_i^r \otimes x_q) & = (-1)^{i-j+1} \frac{1}{2} c \left( \sum_{l=0}^{m-1} \left( \zeta^{(-1)^i r + (|i-j|-1)(-1)^i + (-1)^q p }\right)^l \right) x_{i+1}^{r-1} \otimes x_{q+1} \\
 & \ \ + (-1)^{i-j+1} \frac{1}{2} c \left( \sum_{l=0}^{m-1} \left( \zeta^{i-j+(-1)^q p}\right)^l \right) x_i^{r-1} \otimes x_{q+1} 
 \end{align*}
 for $r>1$.
\end{lemma}

\begin{theorem} \label{odd_dihedral_simple}
 The radical of $\Delta_c(\varphi_1)$ has the following elements as basis:
 \begin{align*}
     & x_1 \otimes x_1, \; x_2 \otimes x_2, \; x_1^2 \otimes x_1, \; x_2^2 \otimes x_2 \;, \\
     & x_1^r \otimes x_q \quad \tn{for } 2< r \leq m-1 \tn{ and } 1 \leq q \leq 2 \;, \\
     & x_2^r \otimes x_q \quad \tn{for } 2< r \leq m \tn{ and } 1 \leq q \leq 2 \;.
 \end{align*}
 Hence,
 \[
 \dim L_c(\varphi_1) =  6 \;.
 \]
 Moreover, as a graded $W$-module we have
 \[
 L_c(\varphi_1)_W = (1+t^2) \cdot \varphi_1 + t \cdot 1 + t \cdot \eps \;.
 \]
\end{theorem}

\begin{proof}
Using the formulas so far it is not hard to verify that the defined subspace $J$ of $\Delta_c(\varphi_1)$ is invariant under the standard generators of $\ol{\HHH}_c$ and is thus an $\ol{\HHH}_c$-submodule of $\Delta_c(\varphi_1)$. It remains to verify that the quotient $\Delta_c(\varphi_1)/J$ is simple. To compute the quotient $\Delta_c(\varphi_1)/J$, we consider the $\bbC$-vector space complement of $J$ in $\Delta_c(\varphi_1)$ spanned by
  \[
  1 \otimes x_1, \; 1 \otimes x_2, \; x_1 \otimes x_2, \; x_2 \otimes x_1, \; x_1^2 \otimes x_2, \; x_2^2 \otimes x_1 \;.
  \]
  The action of $\ol{\HHH}_c$ on $\Delta_c(\varphi_1)/J$ is now determined by the following relations which are not hard to verify using the formulas we discussed so far and the structure of the Gröbner basis of the coinvariant algebra given in \ref{dihedral_coinv_basis}. First, we consider the action of $x_1$:
  \begin{align*}
      & x_1.(1 \otimes x_1) = x_1 \otimes x_1 \equiv 0 \ \msf{mod} \ J \;, \\
      & x_1.(1 \otimes x_2) = x_1 \otimes x_2 \;, \\
      & x_1.(x_1 \otimes x_2) = x_1^2 \otimes x_2 \;, \\
      & x_1.(x_2 \otimes x_1) = x_1x_2 \otimes x_1 = 0 \;, \\
      & x_1.(x_1^2 \otimes x_2) = x_1^3 \otimes x_2 \equiv 0 \ \msf{mod} \  J \;, \\
      & x_1.(x_2^2 \otimes x_2) = x_1x_2^2 \otimes x_2 = 0 \;.
   \end{align*}
  The matrix of the actions of $x_1$ and $x_2$ on $\Delta_c(\varphi_1)/J$ are thus given by
  \[
  \begin{pmatrix}
      0 & 0 & 0 & 0 & 0 & 0 \\
      0 & 0 & 0 & 0 & 0 & 0 \\
      0 & 1 & 0 & 0 & 0 & 0 \\
      0 & 0 & 0 & 0 & 0 & 0 \\
      0 & 0 & 1 & 0 & 0 & 0 \\
      0 & 0 & 0 & 0 & 0 & 0 
  \end{pmatrix} 
  \quad \tn{and} \quad
  \begin{pmatrix}
      0 & 0 & 0 & 0 & 0 & 0 \\
      0 & 0 & 0 & 0 & 0 & 0 \\
      0 & 0 & 0 & 0 & 0 & 0 \\
      1 & 0 & 0 & 0 & 0 & 0 \\
      0 & 0 & 0 & 0 & 0 & 0 \\
      0 & 0 & 0 & 1 & 0 & 0 
  \end{pmatrix} \;.
  \]
  For the actions of $y_1$ and $y_2$ on $\Delta_c(\varphi_1)/J$ we obtain
  \[
  \begin{pmatrix}
      0 & 0 & 0 & 0 & 0 & 0 \\
      0 & 0 & 0 & \frac{1}{2}mc & 0 & 0 \\
      0 & 0 & 0 & 0 & 0 & 0 \\
      0 & 0 & 0 & 0 & -\frac{1}{2}mc & 0 \\
      0 & 0 & 0 & 0 & 0 & 0 \\
      0 & 0 & 0 & 0 & 0 & 0 
  \end{pmatrix}
  \quad \tn{and} \quad
  \begin{pmatrix}
      0 & 0 & \frac{1}{2}mc & 0 & 0 & 0 \\
      0 & 0 & 0 & 0 & 0 & 0 \\
      0 & 0 & 0 & 0 & 0 & 0 \\
      0 & 0 & 0 & 0 & 0 & -\frac{1}{2}mc \\
      0 & 0 & 0 & 0 & 0 & 0 \\
      0 & 0 & 0 & 0 & 0 & 0
  \end{pmatrix} \;.
  \] 
  Finally, for the actions of $s$ and $r^l$ on $\Delta_c(\varphi_1)/J$ we obtain
  \[
  \begin{pmatrix}
      0 & 1 & 0 & 0 & 0 & 0 \\
      1 & 0 & 0 & 0 & 0 & 0 \\
      0 & 0 & 0 & 1 & 0 & 0 \\
      0 & 0 & 1 & 0 & 0 & 0 \\
      0 & 0 & 0 & 0 & 0 & 1 \\
      0 & 0 & 0 & 0 & 1 & 0 
  \end{pmatrix}
  \quad \tn{and} \quad
  \begin{pmatrix}
      \zeta^l & 0 & 0 & 0 & 0 & 0 \\
      0 & \zeta^{-l} & 0 & 0 & 0 & 0 \\
      0 & 0 & 1 & 0 & 0 & 0 \\
      0 & 0 & 0 & 1 & 0 & 0 \\
      0 & 0 & 0 & 0 & \zeta^l & 0 \\
      0 & 0 & 0 & 0 & 0 & \zeta^{-l} 
  \end{pmatrix} \;.
  \]
  Due to the $\bbC^*$-stability, we can assume that $c= 2/m$. Then $c$ and $m$ disappear in the matrices above. The resulting 6-dimensional $\ol{\HHH}_{2m}$-module is easily seen to be irreducible (this can also easily be verified computationally by applying a modular reduction and the \textsc{MeatAxe} to the above family of matrices over $\bbC$, which do not contain contain any parameter any more). Hence, $\Delta_c(\varphi_1)/J$ is irreducible. Hence, this quotient is equal to $L_c(\varphi_1)$ which is therefore 6-dimensional. Using the above matrices for the action of $W$ on $L_c(\varphi_1)$, we can also immediately read off the structure of $L_c(\varphi_1)$ as a graded $W$-module.
  \end{proof}

\begin{corollary}
For the dimension of the Jacobson radical of $\ol{\HHH}_c$ we have 
\begin{equation}
\dim \msf{Rad}( \ol{\HHH}_c) = 8m^3 - 8m^2 - 2m - 30 \;.
\end{equation}
\end{corollary}

\begin{proof}
    We have
    \begin{align*}
    &\quad\ \dim \msf{Rad}( \ol{\HHH}_c) = \dim \ol{\HHH}_c - \sum_{S \in \Irr \ol{\HHH}_c} (\dim S)^2 \\ & = (2m)^3 - \left( (\dim \Delta_c(1))^2 + (\dim \Delta_c(\eps))^2 + (\dim \Delta_c(\rho_1))^2 + \sum_{i=2}^{(m-1)/2} (\dim \Delta_c(\rho_i))^2 \right) \\
    & = 8m^3 - \left( (2m)^2 + (2m)^2 + 6^2 +  \left(\frac{m-1}{2} - 1\right) \cdot 2^2  \right) \\
    & = 8m^3 - 4m^2 - 4m^2 - 36 - 2(m-1) + 4 \\
    & = 8m^3 - 8m^2 - 2m - 30 \;. 
    \end{align*}
\end{proof}

\subsection{Even dihedral groups}

Using similar argumentation as in the proof of Theorem \ref{odd_dihedral_simple} we obtain the solution for even dihedral groups. We omit the details of the computations and just list the results here. \\

\textbf{Generic case.}

\begin{align}
& L_c(\varphi_1)= (t^2+1) \cdot \varphi_1 + t \cdot 1+ t \cdot \eps \\
& L_c(\varphi_{(m-2)/2}) = (t^2+1) \cdot \varphi_{\frac{m-2}{2}} + t \cdot \eps_1 + t \cdot \eps_2 \;.
\end{align}

{\boldmath$a=0, b \neq 0$}\textbf{.}

\begin{align}
& L_c(\varphi_1) = (t^2+1) \cdot \varphi_1 + t \cdot 1+ t \cdot \eps \; &\dim L_c(\varphi_1) = 6 \\
& L_c(\varphi_{(m-2)/2}) = (t^2+1) \cdot \varphi_{\frac{m-2}{2}} + t \cdot \eps_1 + t \cdot \eps_2 \;, & \dim L_c(\varphi_{(m-2)/2}) = 6 \\
& L_c(1) = \sum_{i=1}^{(m-2)/2} \varphi_i \cdot t^i + 1 + t^{(m-2)/2 + 1} \cdot \eps_1 \;, & \dim L_c(1) = m \\ \displaybreak[0]
& L_c(\eps) = \sum_{i=1}^{(m-2)/2} \varphi_i \cdot t^i + \eps + t^{(m-2)/2+1} \eps_2 \;, & \dim L_c(\eps) = m \\ \displaybreak[0]
& L_c(\eps_1) = \sum_{i=1}^{(m-2)/2+1-i} \varphi_i + t^{(m-2)/2} \cdot 1 + \eps_1 \;, & \dim L_c(\eps_1) = m \\ \displaybreak[0]
& L_c(\eps_2) = \sum_{i=1}^{(m-2)/2+1-i} \varphi_i + t^{(m-2)/2} \cdot \eps + \eps_2 \;, & \dim L_c(\eps_2) = m
\end{align}

{\boldmath$a\neq 0, b = 0$}\textbf{.}

\begin{align}
& L_c(\varphi_1) = (t^2+1) \cdot \varphi_1 + t \cdot 1+ t \cdot \eps \; &\dim L_c(\varphi_1) = 6  \\
& L_c(\varphi_{(m-2)/2}) = (t^2+1) \cdot \varphi_{\frac{m-2}{2}} + t \cdot \eps_1 + t \cdot \eps_2 \;, & \dim L_c(\varphi_{(m-2)/2}) = 6 \\
& L_c(1) = \sum_{i=1}^{(m-2)/2} \varphi_i \cdot t^i + 1 + t^{(m-2)/2 + 1} \cdot \eps_2 \;, & \dim L_c(1) = m \\ \displaybreak[0]
& L_c(\eps) = \sum_{i=1}^{(m-2)/2} \varphi_i \cdot t^i + \eps + t^{(m-2)/2+1} \eps_1 \;, & \dim L_c(\eps) = m \\ \displaybreak[0]
& L_c(\eps_1) = \sum_{i=1}^{(m-2)/2+1-i} \varphi_i + t^{(m-2)/2} \cdot \eps + \eps_1 \;, & \dim L_c(\eps_1) = m \\ \displaybreak[0]
& L_c(\eps_2) = \sum_{i=1}^{(m-2)/2+1-i} \varphi_i + t^{(m-2)/2} \cdot 1 + \eps_2 \;, & \dim L_c(\eps_2) = m
\end{align}

{\boldmath$a=b \neq 0$}\textbf{.}

\begin{equation}
L_c(\rho_1) = (t^2+1) \cdot \rho_1 + t \cdot 1+ t \cdot \eps \;.
\end{equation}

{\boldmath$a=-b \neq 0$}\textbf{.}

\begin{equation}
L_c(\rho_{(m-2)/2}) = (t^2+1) \cdot \rho_{(m-2)/2} + t \cdot \eps_1 + t \cdot \eps_2 \;.
\end{equation}

\section{Double centralizer properties} \label{dcp_appendix}

Let $A$ be a ring and let $e \in A$ be an non-zero idempotent. Left multiplication by $e$ yields a functor 
\begin{equation}
E:A\tn{-}\msf{mod} \rightarrow {eAe}\tn{-}\msf{mod} \;.
\end{equation}
As in \cite{Fang-Koenig} we call it the \textit{Schur functor} associated to $e$. Note that $Ae$ is naturally an $(A,eAe)$-bimodule and so $\Hom_A(Ae, M)$ is naturally a left $eAe$-module for any left $A$-module $M$. Mapping $\varphi \in \Hom_A(Ae, M)$ to $\varphi(e)$ yields an $eAe$-module isomorphism 
\begin{equation} \label{dcp_iso_1}
\Hom_A(Ae, M) \overset{\sim}{\longrightarrow} eM \;.
\end{equation}
The inverse maps $m \in eM$ to the map $Ae \rarr M$ mapping $ae$ to $am$, see \cite[Proposition 21.6]{Lam-Noncommutative}. By (\ref{dcp_iso_1}) we get a natural isomorphism 
\begin{equation}
E  \overset{\sim}{\longrightarrow} \Hom_A(Ae,-)
\end{equation}
of functors. In particular, $E$ is representable by a finitely generated projective $A$-module and thus exact. For $M=Ae$ the isomorphism in (\ref{dcp_iso_1}) is in fact a \textit{ring} isomorphism 
\begin{equation} \label{dcp_iso_4}
\End_{A^{op}}(Ae)^{op}  \overset{\sim}{\longrightarrow} eAe \;.
\end{equation}
The inverse maps an element of $eAe$ to right multiplication on $Ae$ with this element. By \cite[\S4.2]{Rouquier-q-Schur} the functor $E$ has right adjoint given by 
\[
G \dopgleich \Hom_{eAe}( E(A), -) = \Hom_{eAe}(Ae, -) : eAe\tn{-}\msf{mod} \rarr A\tn{-}\msf{mod} \;.
\]
This functor is fully faithful and so the unit $\eps: EG \rarr \id$ of the adjunction $E \dashv G$ is an isomorphism. The counit $\eta:\id \rarr GE$ evaluated at $A$ is a left $A$-module morphism
\begin{equation}
A \rarr GE(A) = \Hom_{eAe}( Ae, \Hom_A(Ae, A)) = \End_{eAe}(Ae) \;.
\end{equation}
In fact, this morphism maps $a$ to left multiplication by $a$. This map is in general neither injective nor surjective.

The next lemma is well-known, see \cite[Proposition 4.33]{Rouquier-q-Schur}. 

\begin{lemma} \label{dcp_lemma}
The following are equivalent:
\begin{enumerate}
\item $A \rarr \End_{eAe}(Ae)$ is an isomorphism.
\item The restriction of $E$ to the category ${A}\tn{-}\msf{proj}$ of finitely generated projective $A$-modules is fully faithful.
\item $E$ induces an equivalence between ${A}\tn{-}\msf{proj}$ and ${eAe}\tn{-}\msf{proj}$.
\end{enumerate}
If this holds, we say that the $(A,e)$ satisfies the \tn{double centralizer property} and say that $E$ is a \tn{cover}.
\end{lemma}

Because of the trivial isomorphism (\ref{dcp_iso_4}) the pair $(A,e)$ satisfies the double centralizer property if and only if the $(A,eAe)$-bimodule $Ae$ satisfies the double centralizer property as defined in \cite[II.5]{Skowronski-Yamagata}. The following lemma is also well-known, see \cite[Lemme 5.2.9]{BR}.

\begin{lemma} \label{satake_iso}
If $(A,e)$ satisfies the double centralizer property, then multiplication by $e$ induces an isomorphism $\ZZ(A) \stackrel{\sim}{\longrightarrow} \ZZ(e A e)$.
\end{lemma}

The isomorphism in Lemma \ref{satake_iso} is called the \word{Satake isomorphism}.

\section{Decomposition maps}

In this section we assume that $R$ is a noetherian normal ring with fraction field $K$ and that $A$ is an $R$-algebra which is free and finitely generated as an $R$-module. We assume furthermore that for any $\fp \in \Spec(R)$ the $\msf{k}(\fp)$-algebra $A(\fp) \dopgleich \msf{k}(\fp) \otimes_R A$ splits. We then have a unique decomposition map $\msf{d}_A^\fp:\msf{G}_0(A^K) \rarr \msf{G}_0(A(\fp))$ as defined by Geck and Rouquier \cite{Geck-Rouquier}, see also \cite{Thiel-Dec}. We recall from \cite{Thiel-Dec} that $\msf{d}_A^\fp$ can be realized by a \textit{discrete} valuation ring in $K$ with maximal ideal lying above $\fp$.

\begin{lemma} \label{radical_dec_matrix}
Let $\fp \in \Spec(R)$. Let $V$ be a finite-dimensional $A^K$-module, let $\mscr{O}$ be a discrete valuation ring with maximal ideal $\fm$ lying above $\fp$, and let $\wt{V}$ be an $\mscr{O}$-free $A^\mscr{O}$-form of $V$. If $\wt{V}/\fm \wt{V}$ has simple head $S$, then $\lbrack S \rbrack$ is a constituent of $\msf{d}_{A}^{\fp}(\lbrack V/\Rad(V) \rbrack)$.
\end{lemma}

\begin{proof}
Let $\wt{J} \dopgleich \wt{V} \cap \Rad(V)$. It follows from \cite[Proposition 23.7]{CR} that $\wt{J}$ is a pure submodule of the $\mscr{O}$-module $\wt{V}$, i.e., the quotient $\wt{Q} \dopgleich \wt{V}/\wt{J}$ is $\mscr{O}$-torsion free. Since $\mscr{O}$ is a \textit{discrete} valuation ring, it follows that $\wt{Q}$  is already $\mscr{O}$-free. Clearly, $\wt{Q}$ is an $A^\mscr{O}$-form of $V/\Rad(V)$, so $\msf{d}_{A}^\fp(\lbrack V/\Rad(V) \rbrack) = \lbrack \wt{Q}/\fm\wt{Q} \rbrack$.  We have an exact sequence
\[
0 \rarr \wt{J} \rarr \wt{V} \rarr \wt{Q} \rarr 0
\]
and tensoring with the residue field $\msf{k}(\fm)$ yields an exact sequence
\[
\wt{J}/\fm \wt{J} \rarr \wt{V}/\fm \wt{V} \rarr \wt{Q}/\fm\wt{Q} \rarr 0 \;.
\]
The image of the map $\wt{J}/\fm \wt{J} \rarr \wt{V}/\fm \wt{V}$ is equal to $(\wt{J} + \fm \wt{V})/\fm \wt{V}$, so 
\[
\wt{Q}/\fm\wt{Q} \simeq (\wt{V}/\fm\wt{V})/( (\wt{J}+\fm\wt{V})/\fm \wt{V}) \simeq \wt{V}/(J+\fm \wt{V}) \;.
\]
If we can show that $\wt{J}+\fm \wt{V}$ is a proper submodule of $\wt{V}$, it follows that $S$ is a constituent of $\wt{Q}/\fm\wt{Q}$. Since $\fm = \Rad(\mscr{O})$, we have $\fm (A^\mscr{O}) \subs \Rad(A^\mscr{O})$ by \cite[Corollary 5.9]{Lam-First}. By \cite[Proposition 5.6(iii)]{CR} we have $\Rad(A^\mscr{O}) \wt{V} \subs \Rad(\wt{V})$. Hence, if $\wt{J}+\fm \wt{V} = \wt{V}$, then $\wt{J} + \Rad(\wt{V})\wt{V} = \wt{V}$. This implies that $\wt{J} = \wt{V}$ by Nakayama's lemma \cite[Corollary 5.3]{CR}. By \cite[Proposition 23.7]{CR} this is not possible, since $\wt{J}$ is a pure $\mscr{O}$-form of the proper submodule $\Rad(V)$.
\end{proof}

\begin{proposition} \label{decgen_for_simple_general}
Let $S$ be a simple $A^K$-module. Then the set
\begin{equation}
\DecGen(A,S) \dopgleich \lbrace \fp \in \Spec(R) \mid \msf{d}_{A}^\fp(\lbrack S \rbrack) \tn{ is simple}  \rbrace
\end{equation}
is a neighborhood of the generic point in $\Spec(R)$. 
\end{proposition}

\begin{proof}
Let 
\[
\rho: A^K \rarr \Mat_r(K)
\] 
be a $K$-algebra morphism corresponding to the simple $A^K$-module $S$. Since $A^K$ splits, the morphism $\rho$ is surjective. In \cite[Proposition 4.3]{Thiel-Dec} we have shown that the set
\[
\msf{Gen}(\rho) \dopgleich \lbrace \fp \in \Spec(R) \mid \rho( A_\fp ) = \Mat_r( R_\fp )  \rbrace
\]
is a neighborhood of the generic point in $\Spec(R)$. If $\fp \in \msf{Gen}(\rho)$, then $\rho$ restricts to a surjective $R_\fp$-algebra morphism
\[
\rho|_{A_\fp}: A_\fp \twoheadrightarrow \Mat_r( R_\fp ) 
\] 
and reduction in $\fp$ yields a surjective $\msf{k}(\fp)$-algebra morphism
\[
\ol{\rho|_{A_\fp}}: A(\fp) \twoheadrightarrow \Mat_r(\msf{k}(\fp)) \;.
\]
The morphism $\ol{\rho|_{A_\fp}}$ describes an $A(\fp)$-module $\ol{S}$ and since it is surjective, the module $\ol{S}$ must be simple. Furthermore, the morphism $\rho|_{A_\fp}$ describes an $R_\fp$-free $A_\fp$-form $\wt{S}$ of $S$. We thus have $\msf{d}_{A}^\fp(\lbrack S \rbrack) = \lbrack \ol{S} \rbrack$, the class of a simple module. Hence, $\msf{Gen}(\rho) \subs \DecGen(A,S)$.
\end{proof}

\end{document}